\documentclass[12pt,reqno]{amsart}
	\usepackage[centering,text={432pt,624pt}]{geometry}               
\usepackage{caption}
\usepackage{diagbox}
\usepackage{graphicx}
\usepackage{tikz}
\usepackage{amssymb}
\usepackage{pdfsync}
\usepackage{hyperref}
\usepackage{verbatim} 
\usepackage{epstopdf}
\usepackage{esint} 
\usepackage{amssymb}
\usepackage{latexsym}
\usepackage{amsmath}
\usepackage{bbm}
\usepackage[mathscr]{euscript}
\usepackage{amsthm}

\renewcommand{\d}{\,{\rm d}} 

\newcommand{\sph}[1]{\mathbb{S}^{#1}}

\def\R{\mathbb{R}}

\def\N{\mathbbm{N}}

\def\Z{\mathbb{Z}}
\def\Co{\mathbb{C}}

\newcommand{\te}{\theta}
\newcommand{\la}{\lambda}
\newcommand{\vphi}{\varphi}

\newcommand{\eps}{\varepsilon}
\newcommand{\Mop}{\mathrm{M}}
\newcommand{\Lop}{\mathrm{L}}
\providecommand{\ab}[1]{\vert #1\vert}
\providecommand{\abs}[1]{\Bigl\vert #1 \Bigr\vert}
\providecommand{\Abs}[1]{\biggl\vert #1 \biggr\vert}
\newcommand{\ds}{\displaystyle}

\providecommand{\norma}[1]{\Vert #1 \Vert}

\renewcommand{\leq}{\leqslant}
\renewcommand{\geq}{\geqslant}

\renewcommand{\frak}{\mathfrak}
\newcommand{\EEC}{\partial\mathfrak{U}}

\theoremstyle{plain}
\newtheorem{theorem}{Theorem}[section]
\newtheorem{corollary}[theorem]{Corollary}

\newtheorem{proposition}[theorem]{Proposition}
\newtheorem{lemma}[theorem]{Lemma}

\theoremstyle{definition}
\newtheorem{remark}[theorem]{Remark}

\numberwithin{equation}{section}

\title[Smoothness of solutions of a convolution equation on $\sph{d-1}$]{Smoothness of solutions of a convolution equation of restricted-type on the sphere}

\author{Diogo Oliveira e Silva}
\address{
		Diogo Oliveira e Silva.
        School of Mathematics\\
        University of Birmingham\\
        Edgbaston, Birmingham\\
        B15 2TT, England.}
\email{d.oliveiraesilva@bham.ac.uk}

\author{Ren\'e Quilodr\'an}
\address{Ren\'e Quilodr\'an.}
\email{rquilodr@dim.uchile.cl}

\begin{document}

\begin{abstract}
Let $\mathbb{S}^{d-1}$ denote the unit sphere in Euclidean space $\mathbb{R}^d$, $d\geq 2$, equipped with surface measure $\sigma_{d-1}$. 
An instance of our main result concerns the regularity of solutions of the convolution equation
\[ a\cdot(f\sigma_{d-1})^{\ast {(q-1)}}\big\vert_{\mathbb{S}^{d-1}}=f,\text{ a.e. on }\mathbb{S}^{d-1}, \]
where $a\in C^\infty(\mathbb{S}^{d-1})$,  $q\geq 2(d+1)/(d-1)$ is an integer, and the only {\it a priori} assumption is $f\in L^2(\mathbb{S}^{d-1})$. 
We prove that any such solution belongs to the class $C^\infty(\mathbb{S}^{d-1})$. 
In particular, we show that all critical points associated to the sharp form of the corresponding adjoint Fourier restriction inequality on  $\mathbb{S}^{d-1}$ are $C^\infty$-smooth.
This extends previous work of Christ \& Shao \cite{CS12b} to arbitrary dimensions and general even exponents, and plays a key role in the companion paper \cite{OSQ19}.
\end{abstract}

\maketitle
\setcounter{tocdepth}{1}
\tableofcontents

\section{Introduction}

Sharp Fourier Restriction Theory has attracted a great deal of recent interest.
In the particular case of the unit sphere equipped with surface measure, $(\sph{d-1},\sigma_{d-1})$, a natural starting point is that of the Tomas--Stein inequality,
\begin{equation}\label{eq:TS}
\|\widehat{f\sigma}_{d-1}\|_{L^q(\R^d)}\leq {\bf T}_{d,q} \|f\|_{L^2(\sph{d-1})},
\end{equation}
which is known to hold \cite{St93, To75} with ${\bf T}_{d,q}<\infty$ provided $d\geq 2$ and $q\geq q_d:=2\frac{d+1}{d-1}$; see \eqref{eq:extension} below for the precise definition of the  Fourier extension operator. 
Here  ${\bf T}_{d,q}$ denotes the optimal constant given by
\begin{equation}\label{eq:bestconstant}
{\bf T}_{d,q}=\sup_{{\bf 0}\neq f\in L^2}\frac{\|\widehat{f\sigma}_{d-1}\|_{L^q(\R^d)}}{\|f\|_{L^2(\sph{d-1})}}.
\end{equation}
By a {\it maximizer} of \eqref{eq:TS} we mean a nonzero, complex-valued function $f\in L^2(\sph{d-1})$ for which
$\|\widehat{f\sigma}_{d-1}\|_{L^q(\R^d)}= {\bf T}_{d,q} \|f\|_{L^2(\sph{d-1})}$.

The existence of maximizers for the Tomas--Stein inequality \eqref{eq:TS} has been investigated in the works \cite{CS12a, FVV11, FLS16, Sh16}, but the explicit form of the maximizers is only known in very few, special cases \cite{COS15, Fo15}.
Once maximizers are known to exist, it is natural to investigate their properties with methods from the calculus of variations.
In the present paper, we study the associated Euler--Lagrange equation, and show that the corresponding critical points are $C^\infty$-smooth whenever the exponent $q$ is an even integer.
Our motivation is two-fold.
On the one hand, our main result is used in the companion paper \cite{OSQ19} to establish that constant functions are the unique real-valued maximizers for a number of new sharp instances of inequality \eqref{eq:TS}, and to fully characterize all complex-valued maximizers.
On the other hand, we extend the main results of Christ \& Shao \cite{CS12b} to arbitrary dimensions and general even exponents.

Let $d\geq 2$ and $q\geq q_d$ be given.
Consider the Fourier extension operator $\mathcal E(f)=\widehat{f\sigma}_{d-1}$, acting on functions $f:\sph{d-1}\to\Co$ via
\begin{equation}\label{eq:extension}
\widehat{f\sigma}_{d-1}(x)=\int_{\sph{d-1}} f(\omega) e^{-ix\cdot\omega} \,\textup{d}\sigma_{d-1}(\omega).
\end{equation}
The operator $\mathcal E$ is bounded from $L^2(\sph{d-1})$ to $L^{q}(\R^d)$ in light of \eqref{eq:TS}.
Its adjoint equals the restriction operator, $\mathcal{E}^\ast (g)={g}^\vee\vert_{\sph{d-1}}$, and is bounded from $L^{q'}(\R^d)$ to $L^2(\sph{d-1})$; here, $q'=q/(q-1)$ denotes the conjugate Lebesgue exponent of $q$. 
Suppose that $f$ maximizes the functional $\Phi_{d,q}$ associated to \eqref{eq:TS},
\begin{equation}\label{eq:PhidqDef}
\Phi_{d,q}(f)=\frac{\|\widehat{f\sigma}_{d-1}\|^q_{L^q(\R^d)}}{\|f\|^q_{L^2(\sph{d-1})}},
\end{equation}
and further assume $f$ to be $L^2$-normalized, $\|f\|_{L^2(\sph{d-1})}=1$. We can then estimate the operator norm of the extension operator as follows:
\begin{align}
\|\mathcal{E}\|_{L^2\to L^q}^q
&=\|\mathcal{E} (f)\|_{L^q(\R^d)}^q=\langle |\mathcal{E} (f)|^{q-2}\mathcal{E} (f), \mathcal{E} (f) \rangle
=\langle\mathcal{E}^*(|\mathcal{E} (f)|^{q-2}\mathcal{E} (f)),f\rangle_{L^2(\sph{d-1})}\notag\\
&\leq \|\mathcal{E}^*(|\mathcal{E} (f)|^{q-2}\mathcal{E} (f))\|_{L^2(\sph{d-1})}
\leq \| \mathcal{E}^*\|_{L^{q'}\to L^2} \||\mathcal{E} (f)|^{q-2}\mathcal{E} (f) \|_{L^{q'}(\R^d)}\notag\\
&= \| \mathcal{E}^*\|_{L^{q'}\to L^2} \|\mathcal{E} (f)\|_{L^q(\R^d)}^{q-1}
=\|\mathcal{E}\|_{L^2\to L^q}^q,\label{eq:chain}
\end{align}
where $\langle\cdot,\cdot\rangle$ denotes the $L^{q'}-L^q$ pairing in $\R^d$, and $\langle\cdot,\cdot\rangle_{L^2(\sph{d-1})}$ denotes the $L^2$ pairing on $\sph{d-1}$. Besides easy algebraic manipulations, the first inequality in \eqref{eq:chain} amounts to an application of the Cauchy--Schwarz inequality, 
and  the second inequality in \eqref{eq:chain} holds because the adjoint operator $\mathcal{E}^*$ is bounded from $L^{q'}$ to $L^2$.
In the last identity, we also used the fact that the operator norms of $\mathcal{E},\mathcal{E}^\ast$ coincide, $\|\mathcal{E}\|_{L^2\to L^q}=\|\mathcal{E}^\ast\|_{L^{q'}\to L^2}$. 
Since the first and the last terms in the chain of inequalities \eqref{eq:chain} coincide, all inequalities are forced to be equalities.
In particular, equality holds in the application of the Cauchy--Schwarz inequality, which in turn implies the existence of a constant $\mu$, for which 
$$\mathcal{E}^*(|\mathcal{E} (f)|^{q-2}\mathcal{E} (f))=\mu f$$
holds outside a set of zero $\sigma_{d-1}$-measure.
Thus we see that a maximizer of \eqref{eq:TS} necessarily satisfies 
\begin{equation}\label{eq:ELnonconv}
\Bigl(|\widehat{f\sigma}_{d-1}|^{q-2} \widehat{f\sigma}_{d-1}\Bigr)^{\vee}\Bigl\vert_{\sph{d-1}}
=\lambda \|f\|_{L^2(\sph{d-1})}^{q-2} f,
\quad\sigma_{d-1}\text{-a.e. on }\sph{d-1},
\end{equation}
for some $\lambda\in\Co$.
This is the Euler--Lagrange equation associated to the variational problem \eqref{eq:bestconstant}; see \cite{CQ14} for a more general statement.
To determine the parameter $\lambda\in\Co$, one simply multiplies both sides of \eqref{eq:ELnonconv} by $\bar{f}$ and integrates with respect to surface measure to check that
$\lambda=\Phi_{d,q}(f)$.
In particular, $f$ is a maximizer of inequality  \eqref{eq:TS} if and only if \eqref{eq:ELnonconv} holds with $\lambda={\bf T}_{d,q}^q$.

General non-zero solutions of the Euler--Lagrange equation \eqref{eq:ELnonconv} are called {\it critical points} of the functional $\Phi_{d,q}$.
As noted in \cite{CQ14}, it  follows at once  that constant functions satisfy  \eqref{eq:ELnonconv} for some $\lambda> 0$, simply because $|\widehat{\sigma}_{d-1}|^{q-2}\widehat{\sigma}_{d-1}$ is a radial function,
the inverse Fourier transform of any radial function is radial, and the restriction of any radial function on $\R^d$ to $\sph{d-1}$ is constant.

If $q=2n$ is an even integer, $n\in\N$, then the Tomas--Stein inequality \eqref{eq:TS} can be equivalently stated in convolution form via Plancherel's Theorem as 
\begin{equation}\label{eq:TSconv}
\|(f\sigma_{d-1})^{\ast n}\|^2_{L^2(\R^d)}\leq (2\pi)^{-d}{\bf T}^{2n}_{d,2n} \|f\|^{2n}_{L^2(\sph{d-1})},
\end{equation}
where the $n$-fold convolution measure $(f\sigma_{d-1})^{\ast n}$ is recursively defined for integral values of $n\geq 2$ via
\begin{equation}\label{eq:recsigmaast}
(f\sigma_{d-1})^{\ast 2}=f\sigma_{d-1}\ast f\sigma_{d-1}, \text{ and }
 (f\sigma_{d-1})^{\ast (n+1)}=(f\sigma_{d-1})^{\ast n}\ast f\sigma_{d-1}.
\end{equation}   
The functional $\Phi_{d,2n}$ can then be rewritten as
\begin{equation}\label{eq:PhiConvForm}
\Phi_{d,2n}(f)
=(2\pi)^d\frac{\|(f\sigma_{d-1})^{\ast n}\|_{L^2(\R^d)}^2}{\|f\|_{L^2(\sph{d-1})}^{2n}},
\end{equation}
and the  Euler--Lagrange equation \eqref{eq:ELnonconv} translates into 
\begin{equation}\label{eq:simpleEL}
\Bigl((f\sigma_{d-1})^{\ast n}\ast(f_\star\sigma_{d-1})^{\ast(n-1)}\Bigr) \Big\vert_{\mathbb S^{d-1}}=(2\pi)^{-d}\la \|f\|_{L^2(\sph{d-1})}^{2n-2} f,\quad\sigma_{d-1}\text{-a.e. on }\mathbb S^{d-1},
\end{equation}
where $f_\star$ denotes the {\it conjugate reflection} of $f$ around the origin, defined via 
$$f_\star(\omega)=\overline{f(-\omega)},\quad \text{ for all }\omega\in\sph{d-1}.$$
A function $f:\sph{d-1}\to\Co$ is said to be {\it antipodally symmetric} if $f=f_\star$,
in which case basic properties of the Fourier transform imply that $\widehat{f\sigma}_{d-1}$ is real-valued.

The convolution structure of equation \eqref{eq:simpleEL} induces some extra regularity on its solutions, 
a phenomenon which turns out to hold in greater generality.
To describe it precisely, consider the multilinear operator $\Mop\colon L^2(\mathbb S^{d-1})^{m+1}\to L^2(\mathbb S^{d-1})$,
\begin{equation}\label{eq:definitionMop}
\Mop(f_1,\dotsc,f_{m+1})=(f_1\sigma_{d-1}\ast \dotsm\ast f_{m+1}\sigma_{d-1})\Big\vert_{\mathbb S^{d-1}},
\end{equation}
which is well defined for integral values of $m\geq 4$ if $d=2$ and $m\geq 2$ if $d\geq 3$
 in view of the chain of inequalities \eqref{eq:chain}; see also \cite[Prop.\@ 2.4]{CQ14}.
Further consider the conjugate reflection operator $R\colon L^2(\mathbb S^{d-1})\to L^2(\mathbb S^{d-1})$,  $R(f)=f_\star$. Given an integer
$k\in\N_0$, 
the powers $R^k$ are defined in the usual way via composition, with the understanding that $R^0=\operatorname{Id}$.
We are interested in solutions of the general equation
\begin{equation}\label{eq:generalEL}
a\cdot\Mop(R^{k_1}(f),\dotsc,R^{k_{m+1}}(f))=\la f, \quad\sigma_{d-1}\text{-a.e. on }\sph{d-1},
\end{equation}
where $(k_1,\dotsc,k_{m+1})\in\{0,1\}^{m+1}$, $a\in C^\infty(\sph{d-1})$, and $\la\in\Co$. 
The additional factor $a\in C^\infty(\sph{d-1})$ brings no further complications to the analysis, but can be used to address the smoothness of critical points for weighted measures on $\sph{d-1}$ and, by an additional scaling argument, on ellipsoids. 

Our main result concerns regularity properties of generic solutions of equation \eqref{eq:generalEL}.
\begin{theorem}\label{thm:smoothnessTheorem}
	Let $d\geq 2$, and let $m$ be an integer satisfying
	 $m\geq 4$ if $d=2$, and  $m\geq 2$ if $d\geq 3$.
	Let $(k_1,\dotsc,k_{m+1})\in\{0,1\}^{m+1}$, $a\in C^\infty(\sph{d-1})$, and $\la\in\mathbb{C}\setminus\{0\}$. 
	If $f\in L^2(\sph{d-1})$ is a complex-valued solution of equation \eqref{eq:generalEL}, then $f\in C^\infty(\mathbb S^{d-1})$.
\end{theorem}

\noindent 
The special case $(d,m)=(3,2)$ of Theorem \ref{thm:smoothnessTheorem} implies  \cite[Theorem 1.1]{CS12b}.
Thus Theorem \ref{thm:smoothnessTheorem} extends \cite[Theorem 1.1]{CS12b} to arbitrary dimensions and general even exponents.
Interestingly, our proof of Theorem \ref{thm:smoothnessTheorem} bypasses the Banach fixed point argument from  \cite{CS12b}, and as such could be considered more elementary and of independent value.
Moreover, the case $(d,m)=(2,4)$ of Theorem \ref{thm:smoothnessTheorem} completes the proof of the main result in \cite{Sh16b}, where the following issue was detected:
 in \cite[Proof of Prop.\@ 3.6]{Sh16b}, the first (unnumbered) displayed equation on p.\@ 9 seems to be incorrect.
We further believe that the argument in \cite{Sh16b} cannot be repaired without studying the regularity of the 4-fold convolution $\sigma_1^{\ast 4}$, such as a H\"older-type estimate of the kind established in \S\ref{sec:24} below.
The following result is an immediate consequence of Theorem \ref{thm:smoothnessTheorem}, and is used in a crucial manner in the companion paper \cite{OSQ19}.
\begin{corollary}
Let $d\geq 2$ and $q\geq 2\frac{d+1}{d-1}$ be an even integer.
If $f\in L^2(\sph{d-1})$ is a critical point of the functional  $\Phi_{d,q}$, then $f\in C^\infty(\sph{d-1})$. 
In particular, maximizers of $\Phi_{d,q}$ are $C^\infty$-smooth.
\end{corollary}

\subsection{Outline}

In \S \ref{sec:prelim}, we recall some useful facts about the special orthogonal group, and define the appropriate smoothness spaces on $\sph{d-1}$ on which our estimates will be based.
In \S \ref{sec:prelimest}, we collect some simple properties of the multilinear operator $\Mop$, defined in \eqref{eq:definitionMop}.
A fundamental distinction arises, depending on whether or not the parameters $(d,m)$ from Theorem \ref{thm:smoothnessTheorem} lie on the ``boundary'' of the set of admissible values. 
In the latter case, there is an automatic uniform gain in the initial regularity, which leads to a quick proof of the smoothing property of $\Mop$ in the ``non-boundary'' case; see Lemma \ref{lem:MLambdaBound}. 
This is not possible if $(d,m)$ lies on the boundary, since in that case the corresponding functional is  essentially scale-invariant.
The analysis is then more delicate, and relies on H\"older-type estimates for certain convolution operators, which are the subject of \S \ref{sec:Holder}.
In turn, these estimates are used in \S \ref{sec:HsEEC} to find a suitable replacement for Lemma \ref{lem:MLambdaBound} in the boundary case; see Lemma \ref{cor:EEC}.
The final \S \ref{sec:Smoothness} is devoted to the proof of Theorem \ref{thm:smoothnessTheorem}.
We  proceed in two steps:
firstly,  we establish an initial ``kick'' in the regularity of any solution of equation \eqref{eq:generalEL}; 
secondly, we use a bootstrapping procedure to promote the initial gain in regularity to $C^\infty$-smoothness.

\subsection{Notation}
The set of natural numbers is $\N=\{1,2,3,\ldots\}$, and $\N_0=\N\cup\{0\}$.
Given a set $E\subset \R^d$, its indicator function is denoted by $\mathbbm{1}_E$, its Lebesgue measure by $|E|$, and its complement by $E^\complement=\R^d\setminus E$.
Given $r>0$, we let  $B(x,r)\subset \R^d$ denote the closed ball of radius $r$ centered at $x\in\R^d$, and abbreviate $B_r=B(0,r)$.
We will continue to denote by $(f\sigma_{d-1})^{\ast k}$ the $k$-fold convolution measure, recursively defined in \eqref{eq:recsigmaast}. We denote $\mathbf{1}:\sph{d-1}\to\R$ the function $\mathbf{1}(\omega)\equiv 1$ and the zero function by ${\bf 0}:\sph{d-1}\to\R$, ${\bf 0}(\omega)\equiv 0$.
We use $X\lesssim Y$, $Y\gtrsim X$, or $X=O(Y)$ to denote the estimate $|X|\leq CY$ for an absolute constant $C$, and $X\simeq Y$ to denote the estimates $X\lesssim Y \lesssim X$. We will often require the implied constant $C$ in the above notation to depend on additional parameters, which we will indicate by subscripts (unless explicitly omitted), thus for instance $X \lesssim_j Y$ denotes an estimate of the form $|X| \leq C_jY$ for some $C_j$ depending on $j$.

\section{Function spaces}\label{sec:prelim}

The special orthogonal group $\textup{SO}(d)$ consists of all $d\times d$ orthogonal matrices of unit determinant, and acts transitively on the unit sphere $\sph{d-1}$ in the natural way.
This action extends to actions on functions $f:\sph{d-1}\to\Co$ by $\Theta f=f\circ\Theta$ for $\Theta\in \textup{SO}(d)$,
and on finite Borel measures $\mu$ on $\R^d$ by $\Theta(\mu)(E)=\mu(\Theta(E))$, for $E\subseteq \R^d$.
This extension interacts well with convolutions, in the sense that
$\Theta(\mu\ast\nu)= \Theta(\mu)\ast \Theta(\nu).$
In particular, 
for any $\Theta\in \textup{SO}(d)$,
\begin{equation}\label{eq:Theta}
\Theta(f_1\sigma_{d-1}\ast \cdots\ast f_k\sigma_{d-1})= (\Theta f_1)\sigma_{d-1}\ast\cdots \ast(\Theta f_k)\sigma_{d-1}.
\end{equation}
For further information on the special orthogonal group, see \cite{Ha15} and the references therein.\\

Given $\alpha\in (0,1)$, let $\Lambda_\alpha(\R^d)$ denote the space of  H\"older continuous functions $f:\R^d\to\Co$ of order $\alpha$, with norm
\begin{equation}\label{eq:Holdernorm}
\|f\|_{\Lambda_\alpha(\R^d)}=\|f\|_{C^0(\R^d)}+\sup_{x\neq x'} |x-x'|^{-\alpha} |f(x)-f(x')|.
\end{equation}
Given $1<\alpha\notin\N$, write $\alpha=k+\delta$, with $k\in\N$ and $\delta\in(0,1)$. We then say that $f\in\Lambda_\alpha(\R^d)$ if $f$ is $k$ times continuously differentiable, $f\in C^k(\R^d)$, and all the $k$-th order partial derivatives of $f$ belong to $\Lambda_{\delta}(\R^d)$. An equivalent definition of the space $\Lambda_\alpha(\R^d)$ via Littlewood--Paley projections is available, but we shall delay its precise formulation until the need arises in the proof of Proposition \ref{lem:HolderregularityConvo} below.
Given $\alpha\in (0,1)$, the space of H\"older continuous functions $f:\sph{d-1}\to \Co$ of order $\alpha$, denoted $\Lambda_\alpha(\sph{d-1})$, is defined in a similar way to \eqref{eq:Holdernorm}.
We further consider the space $\textrm{Lip}(\sph{d-1})$ of Lipschitz continuous functions $f:\sph{d-1}\to\Co$, equipped with the norm
\[ \|f\|_{\textup{Lip}(\sph{d-1})}=\|f\|_{C^0(\sph{d-1})}+\sup_{\omega\neq \omega'} |\omega-\omega'|^{-1} |f(\omega)-f(\omega')|.\] 

By $H^s=H^s(\sph{d-1})$ we mean the usual Sobolev space of functions having $s\geq 0$ derivatives in $L^2(\sph{d-1})$, defined via spherical harmonic expansions e.g.\@ as in \cite[\S 1.7.3, Remark 7.6]{LM72},
or by considering a smooth partition of unity and diffeomorphisms onto the unit ball in $\R^{d-1}$ together with the usual Sobolev norm on $\R^{d-1}$; we 
set $H^0=L^2(\sph{d-1})$.
If $s$ is an integer, then the following norm is equivalent to any other norm for $H^s$:
\begin{equation}\label{eq:HsNorm}
\|f\|_{H^s}=\|f\|_{L^2(\sph{d-1})}+\sum_{1\leq i<j\leq d} \|X_{i,j}^s f\|_{L^2(\sph{d-1})},
\end{equation}
where the derivatives are given by
\begin{equation}\label{eq:Dij}
X_{i,j}=x_i\partial_j-x_j\partial_i=\frac{\partial}{\partial \theta_{i,j}},\, X_{i,j}^s=\frac{\partial^s}{\partial \theta_{i,j}^s},
\end{equation}
and $\theta_{i,j}$ denotes the angle in polar coordinates of the $(x_i,x_j)$-plane; see for instance \cite[\S 4.5]{DX}, and \cite[Prop.\@  3.3]{DX11}.

We find it convenient to work with the function spaces $\mathcal{H}^s=\mathcal{H}^s(\sph{d-1})$, which for $d=3$ were introduced in \cite{CS12b}.
To extend the definition to general dimensions $d\geq 2$, recall  \eqref{eq:Dij}, where we introduced the derivatives $X_{i,j}={\partial}/{\partial\theta_{i,j}}$. We can equivalently view $X_{i,j}$ as the $C^\infty$-vector field on $\sph{d-1}$ which generates rotations about the $(x_i,x_j)$-plane, for each $1\leq i<j\leq d$. In this way, for each $\nu=(\nu_1,\dots,\nu_d)\in \sph{d-1}$, $\exp(tX_{i,j})(\nu)$ is obtained by rotating the vector $(\nu_i,\nu_j)$ by $t$ radians.
We note that $\{X_{i,j}:1\leq i<j\leq d\}$ forms a basis for $\mathfrak{so}(d)$, the Lie algebra of $SO(d)$.

Observe that the following quantity defines an equivalent norm on the space $\Lambda_\alpha(\sph{d-1})$, provided $\alpha\in(0,1)$: 
\[ \norma{f}_{C^0(\sph{d-1})}+\max_{1\leq i<j\leq d}\sup_{\omega\in\sph{d-1}}\sup_{t\in\R}\ab{t}^{-\alpha}\ab{f(e^{t X_{i,j}}(\omega))-f(\omega)}. \]
Given $s\in(0,1)$, the space $\mathcal{H}^s$ is defined as the set of all functions $f\in L^2(\sph{d-1})$ for which the norm
\begin{equation}\label{eq:mathcalHsNorm}
\|f\|_{\mathcal{H}^s}=\|f\|_{L^2(\sph{d-1})}+\sum_{1\leq i<j\leq d} \sup_{|t|\leq 1}|t|^{-s}\|f\circ e^{t X_{i,j}}- f\|_{L^2(\sph{d-1})}
\end{equation}
is finite. 
We further set $\mathcal{H}^0=L^2(\sph{d-1})$.
Similarly to the case of Euclidean space, the notion of weak differentiability of a function with respect to the vector field $X_{i,j}$ is made precise by the use of identity \cite[Eq.\@ (5.4)]{OSQ19} which states that, for any complex-valued functions $f,g\in C^1(\sph{d-1})$,
\begin{equation}\label{eq:defWeakDer}
\int_{\sph{d-1}} (X_{i,j}f)\, \overline{g}\d\sigma_{d-1}=-\int_{\sph{d-1}} f\, \overline{(X_{i,j}g)}\d\sigma_{d-1}.
\end{equation}
In this way, we say that $f\in L^2(\sph{d-1})$ is weakly differentiable with respect to the vector field $X_{i,j}$ if there exists a function, denoted $X_{i,j}f$, which belongs to $L^1(\sph{d-1})$ and satisfies \eqref{eq:defWeakDer} for all $g\in C^\infty(\sph{d-1})$.
 
If $s=k+\alpha$, with $k\in\N$ and $\alpha\in (0,1)$,  then the space $\mathcal{H}^s$ consists of all functions $f\in L^2(\sph{d-1})$ for which the norm
\begin{equation}\label{eq:mathcalHsNormBig}
\|f\|_{\mathcal{H}^s}=\|f\|_{L^2(\sph{d-1})}+\sum_Y\sum_{1\leq i<j\leq d} \sup_{|t|\leq 1}|t|^{-\alpha}\|Yf\circ e^{t X_{i,j}}- Yf\|_{L^2(\sph{d-1})}
\end{equation}
is finite, where $Y$ ranges over the finite set of all compositions $X_{i_1,j_1}\circ X_{i_2,j_2}\circ\cdots\circ X_{i_\ell,j_\ell}$ with $0\leq \ell\leq k$ factors, and $f$ itself is viewed as $Yf$ where $Y$ has zero factors.
We implicitly assume the function $f$ to be weakly differentiable with respect to the vector fields $\{X_{i,j}\}_{1\leq i<j\leq d}$ as many times as required by the definition of the norm.

The next result explores the relationship between the function spaces $\mathcal{H}^s$ and the usual Sobolev spaces $H^t$.
\begin{lemma}\label{lem:SobolevHolder}
For every $0\leq t<s$, $s\notin\N$, $\mathcal H^s$ is contained in the Sobolev space $H^t$, and
\begin{equation}\label{eq:SobHolIneq}
\|f\|_{H^t}\leq C(s,t) \|f\|_{\mathcal H^s},
\end{equation}
for all $f\in \mathcal H^s$ and some constant $C(s,t)<\infty$.
\end{lemma}

\noindent Estimate \eqref{eq:SobHolIneq} was noted in  \cite[Lemma 2.1]{CS12b} in the three-dimensional case $d=3$ when $s<1$. From Lemma \ref{lem:SobolevHolder} it follows at once that, given $s\in(1,\infty)\setminus\N$, $f\in \mathcal H^s$, and $X\in \{X_{i,j}: \;1\leq i<j\leq d\}$, then  $\norma{Xf}_{L^2(\sph{d-1})}\lesssim \norma{f}_{\mathcal{H}^s}$ and therefore $Xf\in \mathcal H^{s-1}$. This observation will be useful in the sequel.

 As a preliminary step towards the proof of Lemma \ref{lem:SobolevHolder}, we recall the Euclidean Sobolev spaces $H^s(\R^d)$, and define the spaces $\mathcal{H}^s(\R^d)$ in analogy to the spherical ones, $\mathcal{H}^s$. Given  $f:\R^d\to\R$ and $s=k+\alpha$ with $k\in\N_0$ and $\alpha\in(0,1)$, we consider the norms 
\begin{equation}\label{eq:SobRd}
\|f\|_{H^s(\R^d)}^2:=\int_{\R^d}(1+|\xi|^2)^s|\widehat{f}(\xi)|^2\d\xi,
\end{equation}
\begin{equation}\label{eq:LipRd}
\|f\|_{\mathcal{H}^s(\R^d)}:=\|f\|_{L^2(\R^d)}+\sum_{\ell} \sup_{|x|\leq 1}|x|^{-\alpha}\|D^\ell f\circ \tau_x- D^\ell f\|_{L^2(\R^d)},
\end{equation}
where 
the sum in \eqref{eq:LipRd} runs over all multiindices $\ell=(\ell_1,\dotsc,\ell_d)\in\N^d$ satisfying $0\leq |\ell|\leq k$,
$D^\ell:={\partial^{\ab{\ell}}}/{\partial x_d^{\ell_d}\dotsb\partial x_1^{\ell_1}}$ denotes the partial derivative,
$\tau_x:\R^d\to\R^d, y\mapsto x+y$ denotes  translation by $x=(x_1,\dotsc,x_d)\in \R^d$, and $f$ itself is viewed as $D^0f$. 
For $1\leq i\leq d$, let $e_i$ denote the $i$-th canonical vector $e_i=(0,\dotsc,0,1,0,\dotsc,0)$, with the $1$ in the $i$-th position. 
For every $s=k+\alpha$, $k\in\N_0$, $\alpha\in(0,1)$, the following is an equivalent norm for $\mathcal{H}^s(\R^d)$, perhaps more reminiscent to that for $\sph{d-1}$ in \eqref{eq:mathcalHsNormBig} :
\begin{equation}\label{eq:LipRdV2}
\|f\|_{L^2(\R^d)}+\sum_{0\leq \ab{\ell}\leq k}\sum_{1\leq i\leq d} \sup_{\ab{t}\leq 1}\ab{t}^{-\alpha}\|D^\ell f\circ \tau_{te_i}- D^\ell f\|_{L^2(\R^d)}.
\end{equation}
It is also worth observing that, by the triangle inequality and the translation invariance of the Lebesgue measure in $\R^d$, an equivalent norm to that in \eqref{eq:LipRd} or \eqref{eq:LipRdV2} is obtained by replacing $\sup_{\ab{t}\leq1}$ by $\sup_{\ab{t}\leq \eps}$, for any $\eps>0$. Likewise, by the triangle inequality and the $SO(d)$-invariance of the measure $\sigma_{d-1}$ in $\sph{d-1}$, an equivalent norm for $\mathcal{H}^s$ is obtained from \eqref{eq:mathcalHsNormBig} by replacing $\sup_{\ab{t}\leq 1}$ by $\sup_{\ab{t}\leq \eps}$, for any $\eps>0$.

\begin{proof}[Proof of Lemma \ref{lem:SobolevHolder}]
We discuss the analogous Euclidean statement, for the case of the sphere then follows by working in local coordinates. 
In fact, as already mentioned, the $H^t$-norm on $\sph{d-1}$ can be defined by  considering a smooth partition of unity and diffeomorphisms onto the unit ball in $\R^{d-1}$ together with the usual Sobolev norm on $\R^{d-1}$, as in \eqref{eq:SobRd}.
In order to handle the $\mathcal H^s$-norm on $\sph{d-1}$, we observe that it is likewise amenable to the use of local coordinates: given a smooth partition of unity $\{\vphi_i\}_{1\leq i\leq N}$ on $\sph{d-1}$, then $f\in\mathcal{H}^s$ if and only if $\vphi_if\in\mathcal{H}^s$, for every $1\leq i\leq N$, and $\norma{f}_{\mathcal{H}^s}\simeq \sum_{1\leq i\leq N}\norma{\vphi_i f}_{\mathcal H^s}$. Let $\mathcal{O}_i$ denote the support of $\vphi_i$, which we may take to be connected and of small diameter if necessary, and let $\{(\Omega_i,\psi_i)\}_{1\leq i\leq N}$ denote a system of local coordinates for $\sph{d-1}$ subordinate to $\{\mathcal{O}_i\}_{1\leq i\leq N}$; i.e., $\Omega_i$ is open and connected, $\psi_i:\Omega_i\to \operatorname{int}(B_1)$ is a diffeomorphism onto the open unit ball in $\R^{d-1}$, and $\mathcal{O}_i\Subset \Omega_i$ is compactly contained in $\Omega_i$. If $\eps>0$ is small enough, it then follows that $\norma{\vphi_if}_{L^2(\sph{d-1})}=\norma{\vphi_i f}_{L^2(\Omega_i)}$ and, for every $\ab{t}\leq\eps$,
\[ \|Y(\vphi_if)\circ e^{t X_{k,l}}- Y(\vphi_if)\|_{L^2(\sph{d-1})}=\|Y(\vphi_if)\circ e^{t X_{k,l}}- Y(\vphi_if)\|_{L^2(\Omega_i\cap e^{-tX_{k,l}}(\Omega_i))}, \]
for every $Y$ and $(k,l)$ as in \eqref{eq:LipRd}. In this way, in order to show that $\norma{\vphi_if}_{\mathcal{H}^s}\simeq\norma{(\vphi_if)\circ\psi_i^{-1}}_{\mathcal{H}^s(\R^{d-1})}$, one may appeal to the theory of differentiability along noncommuting vector fields, as developed in \cite[\S 4]{Ho67}; see, in particular, Lemmata 4.1, 4.2, and Theorem 4.3 in \cite{Ho67}.

In light of the previous paragraph, we can assume that in the Euclidean case the relevant supports are contained in the unit ball of $\R^d$. 
This will be useful later on in the argument. More precisely, the task is now to show that there exists $C(s,t)<\infty$ such that for every $f\in\mathcal{H}^s(\R^d)$ whose support is contained in the unit ball of $\R^d$, it holds that
\begin{equation}\label{eq:SobHolIneqRd}
\norma{f}_{H^t(\R^d)}\leq C(s,t)\norma{f}_{\mathcal{H}^s(\R^d)},
\end{equation}
whenever $0\leq t<s\notin \N$.

We start by considering the case $0<t<s<1$ (the case $t=0$ being trivial), and recalling the equivalent formulation of Sobolev spaces in terms of the Riesz potential.
Fix $t\in (0,1)$, and let $f:\R^d\to\R$ be given. 
From Plancherel's Theorem, we have that 
\begin{align}
 \int_{(\R^d)^2}\frac{|f(x+y)-f(y)|^2}{|x|^{d+2t}}\d x\d y 
&=\int_{\R^d}\int_{\R^d} |e^{ix\cdot\xi}-1|^2|\widehat{f}(\xi)|^2\d\xi\frac{\d x}{|x|^{d+2t}}\\
&=A_{t,d}\int_{\R^d} |\xi|^{2t}|\widehat{f}(\xi)|^2 \d\xi, \label{eq:Atd}
\end{align}
where we used the fact that the integral
\[I_{t,d}(\xi):=\int_{\R^d}\frac{|e^{ix\cdot \xi}-1|^2}{|x|^{d+2t}}\d x\]
satisfies $I_{t,d}(\lambda\xi)=\lambda^{2t} I_{t,d}(\xi)$, for every $\lambda>0$.
The constant $A_{t,d}$ in \eqref{eq:Atd} satisfies $A_{t,d}=I_{t,d}(\omega)$, for any $\omega\in\sph{d-1}$, and 
is finite as long as $t\in (0,1)$.
In turn, since $t\in (0,1)$, we have that 
\[(1+|\xi|^2)^t\leq 1+|\xi|^{2t}\leq 2(1+|\xi|^2)^t,\]
for every $\xi\in\R^d$. In particular, the following two-sided estimate holds:
\begin{equation}\label{eq:SobolevRiesz}
\|f\|_{H^t(\R^d)}^2\simeq_{t,d}\|f\|_{L^2(\R^d)}^2+\int_{(\R^d)^2}\frac{|f(x+y)-f(y)|^2}{|x|^{d+2t}}\d x\d y.
\end{equation}
Given $s\in(t,1)$, we use H\"older's inequality to estimate:
\begin{align*}
\int_{(\R^d)^2}&\frac{|f(x+y)-f(y)|^2}{|x|^{d+2t}}\d x\d y\\
&=\int_{|x|\leq 1}\int_{\R^d}\frac{|f(x+y)-f(y)|^2}{|x|^{d+2t}}\d y\d x
+\int_{|x|> 1}\int_{\R^d}\frac{|f(x+y)-f(y)|^2}{|x|^{d+2t}}\d y\d x\\
&\leq \biggl(\sup_{|x|\leq 1}\int_{\R^d}\frac{|f(x+y)-f(y)|^2}{|x|^{2s}}\d y\biggr)\int_{|x|\leq 1}\frac{\d x}{|x|^{d-2(s-t)}}\\
&\qquad+2\int_{|x|> 1}\int_{\R^d}\frac{|f(x+y)|^2+|f(y)|^2}{|x|^{d+2t}}\d y\d x \\
&\lesssim_{d} (s-t)^{-1} \|f\|_{\mathcal H^s(\R^d)}^2+t^{-1}\|f\|_{L^2(\R^d)}^2.
\end{align*}
In light of \eqref{eq:SobolevRiesz}, this establishes \eqref{eq:SobHolIneqRd} in the particular case when $0<t<s<1$.

We now consider the case when $s=k+\alpha$, with $k\in\N$ and $\alpha\in (0,1)$. 
No generality is lost in assuming that $t\in(k,s)$, and specializing to the case $k=1$, so that the desired conclusion would follow from the estimate $\norma{D^\ell f}_{H^{t-1}(\R^d)}\lesssim\norma{f}_{\mathcal{H}^{s}(\R^d)}$, for every $\ab{\ell}\leq 1$. 
In this case, the previous argument applies to $D^\ell f$ {\it provided} $D^\ell f\in L^2(\R^d)$ for any $|\ell|=1$, with appropriately bounded $L^2(\R^d)$-norm, which we now verify in the special case when the support of $f$ is contained in the unit ball of $\R^d$. As discussed above, this will suffice for our application to $\sph{d-1}$.

Fix a multiindex $\ell\in\N^d$,  $|\ell|=1$, and write $D:=D^\ell$. 
Let $g\in C_0^\infty(\R^d)$ 
be such that supp$(g)\subset B_1$. 
Then $(D g)(y-t\ell)=-\frac{\d}{\d t}(g(y-t\ell))=D(g(\cdot-t\ell))(y)$.
By Fubini's Theorem and the definition of weak derivative of $f$, it follows that 
\[ 0=\int_{-2}^{2}\int_{\R^d} \overline{(D g)}(y-t\ell){f(y)}\d y \d t
=-\int_{-2}^2\int_{\R^d} \overline{g}(y){(Df)(y+t\ell)}\d y \d t.\]
As a consequence,
\begin{align*}
\int_{\R^d} Df(y)\overline{g}(y) \d y
&=\frac14\int_{-2}^2\int_{\R^d}\bigl(Df(y)-Df(y+t \ell)\bigr)\overline{g}(y)\d y\d t.
\end{align*}
By adding and subtracting appropriate terms, the triangle and Cauchy--Schwarz inequalities and the invariance of the Lebesgue measure in $\R^d$ with respect to translations together imply
\begin{align}
\biggl\vert\int_{\R^d} Df(y)\overline{g}(y) \d y\biggr\vert
&\leq \frac{1}{4}\int_{-2}^{2}\norma{Df\circ \tau_{t\ell}-Df\circ\tau_{\frac{t}{2}\ell} }_{L^2(\R^d)}\norma{g}_{L^2(\R^d)}\d t\nonumber\\
&\qquad+\frac{1}{4}\int_{-2}^{2}\norma{Df\circ \tau_{\frac{t}{2}\ell}-Df}_{L^2(\R^d)}\norma{g}_{L^2(\R^d)}\d t\nonumber\\
&\lesssim\int_{0}^{1}\norma{Df\circ \tau_{t\ell}-Df}_{L^2(\R^d)}\d t\,\norma{g}_{L^2(\R^d)}\nonumber\\
\label{eq:L2pairingDerivative}
&\leq\left( \int_{0}^{1}t^\alpha\d t\right)\,\norma{f}_{\mathcal{H}^{1+\alpha}(\R^d)}\norma{g}_{L^2(\R^d)}.
\end{align}
Consequently, $Df\in L^2(\R^d)$ and $\norma{Df}_{L^2(\R^d)}\lesssim\norma{f}_{\mathcal{H}^{s}(\R^d)}$. 
 This concludes the proof of the lemma.
\end{proof}
\begin{remark}\label{rem:SobolevHolderIntegerSphere}
For our purposes later on, it will suffice to invoke the following simpler consequence of Lemma \ref{lem:SobolevHolder}: For any $0<s\notin\N$, there is a continuous embedding $\mathcal{H}^s\subseteq H^{\lfloor s\rfloor}$. We now provide a short proof of this fact which is intrinsic to the sphere. 
The case $s\in(0,1)$ is clear since then $H^{\lfloor s\rfloor}=L^2(\sph{d-1})$. For $s>1$, $s\notin \N$, and $f\in \mathcal{H}^s$, it suffices to show that $\norma{Xf}_{L^2(\sph{d-1})}\lesssim\norma{f}_{\mathcal{H}^s}$, where $X$ ranges over the finite set of all compositions $X_{i_1,j_1}\circ X_{i_2,j_2}\circ\dotsm\circ X_{i_\ell,j_\ell}$ with $1\leq \ell\leq \lfloor s\rfloor$ factors. \footnote{As stated in \eqref{eq:HsNorm} and explained in the references thereafter, it suffices to consider the case $\ell=\lfloor s\rfloor$.} As noted in the course of the proof of Lemma \ref{lem:SobolevHolder}, we may specialize to the case $\ell=1$ since the general case follows in the same way. We then simply note that, for any $g\in C^\infty(\sph{d-1})$, $X\in\{X_{i,j}\colon 1\leq i<j\leq d\}$, and $\omega\in\sph{d-1}$, it holds that $(Xg)(e^{tX}\omega)=(X(g\circ e^{tX}))(\omega)=\frac{d}{dt}(g(e^{tX}\omega))$, so that $\int_{0}^{2\pi}(Xg)(e^{tX}\omega)\d t=0$, and the desired estimate,
\[
\biggl\vert\int_{\sph{d-1}} Xf(\omega)\overline{g}(\omega) \d\sigma_{d-1}(\omega)\biggr\vert
\lesssim\norma{f}_{\mathcal{H}^s}\norma{g}_{L^2(\sph{d-1})},
\]
follows in the same way as \eqref{eq:L2pairingDerivative}. See also \cite[Cor.\@ 7]{Co83} for a discussion of this embedding using an equivalent definition\footnote{We comment on various equivalent definitions of the space $\mathcal H^s$ in \S \ref{subseq:secondDifferences} below.} of $\mathcal{H}^s$.
\end{remark}

\section{Preliminary inequalities}\label{sec:prelimest}
We establish some linear and multilinear inequalities which will be used to analyze the solutions of equation \eqref{eq:generalEL}.
Our first result translates into a modest amount of control over the regularity of  convolution measures
in a number of situations of interest.

\begin{proposition}\label{lem:HolderregularityConvo}
	Given integers $d,m\geq 2$, set $\alpha=\frac{1}{2}(d-1)(m-2)-1$. 
	Let $\{f_j\}_{j=1}^m\subset C^\infty(\mathbb{S}^{d-1})$. 
		If $\alpha>0$, then $f_1\sigma_{d-1}\ast\cdots\ast f_m\sigma_{d-1}\in\Lambda_{\alpha}(\R^d)$.
\end{proposition}

\noindent The proof of Proposition \ref{lem:HolderregularityConvo} is based on the classical Littlewood--Paley characterization of the H\"older spaces $\Lambda_{\alpha}(\R^d)$; see \cite[\S 6.3]{Gr14} and \cite[Ch.\@ VI \S 5]{St93}.

\begin{proof}[Proof of Proposition \ref{lem:HolderregularityConvo}]
Consider a smooth partition of unity in $\R^d$. 
More precisely, fix $\eta\geq 0$, a nonnegative, decreasing and radial $C^\infty$-function of compact support, defined on $\R^d$, with the properties that
$\eta(x)=1$ for $|x|\leq 1$, and $\eta(x)=0$ for $|x|\geq 2$.
Together with $\eta$, define another function $\delta$, by $\delta(x):=\eta(x)-\eta(2x)\geq 0$.
For each integer $j\geq 1$, consider the function $\vphi_j:=\delta(2^{-j}\cdot)$,
which is  supported on the spherical shell $\{x\in\R^d\colon 2^{j-1}\leq\ab{x}\leq 2^{j+1}\}$, and let $\vphi_0=\eta$, so that  
\[\sum_{j=0}^\infty\vphi_j(x)=1, \quad\text{for every } x\in\R^d.\]
For $\alpha>0$, a function $G:\R^d\to\Co$ belongs to $\Lambda_\alpha(\R^d)$ if and only if 
\begin{equation}\label{eq:LPcharac}
 \sup_{j\in\N_0} 2^{j\alpha}\norma{(\widehat{G}\vphi_j)^\vee}_{L^\infty(\R^d)}<\infty. 
 \end{equation}
Moreover,
the expression on the left-hand side of \eqref{eq:LPcharac} produces a norm which is equivalent to any other norm for $\Lambda_{\alpha}(\R^d)$;
see \cite[Theorem 6.3.7]{Gr14}.
The Hausdorff--Young inequality implies that  estimate \eqref{eq:LPcharac} is fulfilled if
\[ \int_{\R^d}\ab{\widehat{G}(x)\vphi_j(x)}\d x\lesssim 2^{- j\alpha},\quad j=0,1,2,\ldots, \]
for some implicit constant which does not depend on $j$. 
Now, the Fourier transform of $F:=f_1\sigma_{d-1}\ast\cdots\ast f_m\sigma_{d-1}$ is given by $\widehat{F}=\prod_{k=1}^m\widehat{f_k\sigma}_{d-1}$, which leads to the analysis of the integrals
\[ \int_{B_2}\prod_{k=1}^m\big|\widehat{f_k\sigma}_{d-1}(x)\big| \d x,\quad
\int_{B_{2^{j+1}}\setminus B_{2^{j-1}}}{\prod_{k=1}^m\big|\widehat{f_k\sigma}_{d-1}(x)\big|} \d x, \quad j=1,2,\ldots \]
A well-known stationary phase argument applied to each $f_k\in C^\infty(\sph{d-1})$ yields the following decay estimate:
\[ \ab{\widehat{f_k\sigma}_{d-1}(x)}\lesssim (1+\ab{x})^{-\frac{d-1}{2}}, \text{ for every }x\in\R^d,\]
where the implicit constant depends only on the dimension $d$ and the function $f_k$; see \cite[Chapter VIII, \S 3.1]{St93}. 
Using polar coordinates, it is then direct to check that
\[ \int_{B_{2^{j+1}}\setminus B_{2^{j-1}}} \prod_{k=1}^m\big|\widehat{f_k\sigma}_{d-1}(x)\big|  \d x\lesssim 2^{jd}2^{-\frac{jm(d-1)}{2}}=2^{-j((d-1)(\frac{m}{2}-1)-1)}, \]
 for every $j\in\N$.
The desired conclusion follows from this and from the observation that $\widehat{F}$ defines a continuous function on $\R^d$, and is thus bounded on the ball $B_2\subset\R^d$.
\end{proof}

\begin{remark}
We find it convenient to consider the ``universe'' of admissible parameters
 \[\mathfrak{U}=\{(d,m)\in\N^2\colon d=2 \text{ and } m\geq 4, \text{ or }  d\geq 3\text{ and } m\geq 2\},\] 
 together with its ``boundary''
\begin{equation}\label{eq:EECdef}
\EEC=\{(2,4),(3,3)\}\cup\{(d,2)\colon d\geq 3\}.
\end{equation}
Note that the set $\mathfrak{U}$ encapsulates the hypotheses on $d,m$ imposed by Theorem \ref{thm:smoothnessTheorem}.
On the other hand, with the exception of $(d,m)=(3,3)$, the set $\EEC$ contains precisely those values $(d,m)$ for which $m$ is the smallest even integer such that ${\bf T}_{d,m+2}<\infty$, and therefore the corresponding inequality \eqref{eq:TS} holds. As the upcoming sections will reveal, the analysis simplifies considerably if $(d,m)\in\mathfrak{U}\setminus \partial\frak{U}$, which is the reason to treat the boundary set $\partial\frak{U}$ separately.
As a first instance of this phenomenon, note that, given $(d,m)\in\mathfrak{U}$, we have that $(d,m)\notin\EEC$ if and only if $\frac{1}{2}(d-1)(m-2)-1>0$. 
These are precisely the cases covered by Proposition \ref{lem:HolderregularityConvo}.
See also the comments following Lemma \ref{lem:MLambdaBound}, and Remark \ref{eq:freeGain} below.
\end{remark}

Recall the operator
 $\Mop\colon L^2(\mathbb S^{d-1})^{m+1}\to L^2(\mathbb S^{d-1})$, which was defined in \eqref{eq:definitionMop} as
\[\Mop(f_1,\dotsc,f_{m+1})=(f_1\sigma_{d-1}\ast \dotsm\ast f_{m+1}\sigma_{d-1})\Big\vert_{\mathbb S^{d-1}}.\]

\begin{lemma}\label{prop:basicM}
	The operator $\Mop$ defined in \eqref{eq:definitionMop} satisfies the following properties: 
	\begin{itemize}
		\item[(i)]
		$\Mop$  is an $(m+1)$-linear operator; 
		\item[(ii)]
		$\Mop$ is symmetric in the sense that, given any permutation $\tau$ of $\{1,2,\dotsc,m+1\}$,
		\begin{equation}\label{eq:symmetryM}
		\Mop(f_1,\dotsc,f_{m+1})=\Mop(f_{\tau(1)},\dotsc,f_{\tau(m+1)});
		\end{equation}
		\item[(iii)]
		For any $\Theta\in \mathrm{SO}(d)$, the following identities hold:
		\begin{equation}\label{eq:compositionThetaM}
		\Mop(f_1,\dotsc,f_{m+1})\circ\Theta=\Mop(f_1\circ\Theta,\dotsc,f_{m+1}\circ\Theta);
		\end{equation}
		
		\begin{equation}
		\begin{split}
		(\Theta-I)\Mop(f_1,\dotsc,f_{m+1})
		&=\sum_{j=1}^{m+1}\Mop(f_1,\dotsc,f_{j-1},(\Theta-I)f_j,\Theta f_{j+1}, \dotsc, \Theta f_{m+1})\\
		&=\Mop((\Theta-I)f_1,\Theta f_2,\dotsc,\Theta f_{m+1})\\
		&\quad+\Mop(f_1,(\Theta-I)f_2,\dotsc,\Theta f_{m+1})\\
		&\qquad\vdots\\
		&\quad+\Mop(f_1,f_2,\dotsc,(\Theta-I) f_{m+1});\label{eq:expansionM}
		\end{split}
		\end{equation}
		\item[(iv)] For any $s\geq 0$, there exists $A_s<\infty$ such that, if $\{f_j\}_{j=1}^{m+1}\subset H^s$, then
		\begin{equation}\label{eq:SobolevboundforM}
		\norma{\Mop(f_1,\dotsc,f_{m+1})}_{H^s}\leq A_s\prod_{j=1}^{m+1}\norma{f_j}_{H^s};
		\end{equation}
		\item[(v)] If\footnote{Recall the definition \eqref{eq:Dij} of $X_{i,j}=\frac{\partial}{\partial\theta_{i,j}}$.}  $X=X_{i,j}$, for some $1\leq i< j \leq d$, and $\{f_k\}_{k=1}^{m+1}\subset H^1$, 
		then
		\begin{equation}\label{eq:derivativeMop}
		X\Mop(f_1,\dotsc,f_{m+1})=\sum_{k=1}^{m+1}\Mop(f_1,\dots,f_{k-1},X f_k,f_{k+1},\ldots,f_{m+1});
		\end{equation}
		\item[(vi)]
		For any $0<s\notin\Z$, there exists $C_s<\infty$ such that, if $\{f_j\}_{j=1}^{m+1}\subset \mathcal{H}^s$, then
		\begin{equation}\label{eq:HsboundforM}
		\norma{\Mop(f_1,\dotsc,f_{m+1})}_{\mathcal H^s}\leq C_s\prod_{j=1}^{m+1}\norma{f_j}_{\mathcal H^s}.
		\end{equation}
	\end{itemize}
\end{lemma}

\noindent We record the basic $L^2$-estimate, which coincides with the  case $s=0$ of \eqref{eq:SobolevboundforM}:
\begin{equation}\label{eq:basicL2}
	\norma{\Mop(f_1,\dotsc,f_{m+1})}_{L^2(\mathbb{S}^{d-1})}\lesssim\prod_{j=1}^{m+1}\norma{f_j}_{L^2(\mathbb{S}^{d-1})}.
\end{equation}

\begin{proof}[Proof of Lemma \ref{prop:basicM}]
	We prove estimate \eqref{eq:HsboundforM} only, the rest being direct from the definitions, or simple to verify; in particular, the proof of (iv) is analogous to that of \cite[Lemma 2.2]{CS12b}. 
	Let us first assume that $s\in(0,1)$.
	Given $\{f_k\}_{k=1}^{m+1}\subset \mathcal{H}^s$,  set $g:=\Mop(f_1,\dots f_{m+1})$. 
	 Let $\Theta=e^{tX}\in \textup{SO}(d)$, where $X=X_{i,j}$, for some $1\leq i<j\leq d$.
	 In light of \eqref{eq:expansionM}, we then have that
	\begin{equation}\label{eq:expansionfHsbound}
	\Theta g-g=\sum_{k=1}^{m+1}\Mop(f_1,\dotsc,f_{k-1},(\Theta-I)f_k,\Theta f_{k+1},\dotsc, \Theta f_{m+1}).
	\end{equation}
         By \eqref{eq:basicL2}, the first summand on the right-hand side of \eqref{eq:expansionfHsbound} satisfies
	\[\norma{\Mop((\Theta-I) f_1,f_2,\dots, f_{m+1})}_{L^2(\sph{d-1})}
	\lesssim\|\Theta f_1-f_1\|_{L^2(\sph{d-1})}\prod_{\ell=2}^{m+1}\|f_\ell\|_{L^2(\sph{d-1})},\]
	and similarly for the other $m$ summands.
	It follows that
	\begin{align*}
	\sup_{|t|\leq 1}|t|^{-s}\|g\circ e^{tX}-g\|_{L^2(\sph{d-1})}
	&\lesssim\sum_{k=1}^{m+1} \sup_{|t|\leq 1} |t|^{-s}\|e^{tX}f_k-f_k\|_{L^2(\sph{d-1})}\prod_{\ell:\,\ell\neq k} \|f_\ell\|_{L^2(\sph{d-1})}\\
	&\leq\sum_{k=1}^{m+1} \norma{f_k}_{\mathcal H^s}\prod_{\ell:\, \ell\neq k} \|f_\ell\|_{L^2(\sph{d-1})}
	\leq \prod_{k=1}^{m+1} \|f_k\|_{\mathcal{H}^s}.
	\end{align*}
	Since this holds whenever $X$ is any of the vector fields $\{X_{i,j}\}_{1\leq i<j\leq d}$, estimate \eqref{eq:HsboundforM} follows, settling (vi) in the special case when $s\in(0,1)$. Now suppose $s=k+\alpha$, with $k\in\N$  and $\alpha\in(0,1)$. Let $1\leq \ell\leq k$, and consider a composition $Y$ with $\ell$ factors as in \eqref{eq:mathcalHsNormBig}. 
 Remark \ref{rem:SobolevHolderIntegerSphere} and estimate \eqref{eq:SobolevboundforM} imply that $g\in H^k$.
	In light of \eqref{eq:derivativeMop}, we then see that $Yg$ can be written as a sum of terms of the form
	$\Mop(Y_1f_1,\dots,Y_{m+1}f_{m+1})$,
	where $Y_1,\dots,Y_{m+1}$ are compositions of $i_1,\dots,i_{m+1}$ vector fields $X_{i,j}$, and $\sum_{j=1}^{m+1}i_j=\ell$. Note that $Y_jf_j\in \mathcal H^\alpha$ for all such vector fields, and $\norma{Y_jf_j}_{\mathcal H^\alpha}\leq \norma{f_j}_{\mathcal H^s}$. Expanding $(\Theta-I)Yg$ as in \eqref{eq:expansionfHsbound}, we find in the same way as before
	 that
	\[ \sup_{\ab{t}\leq 1}\ab{t}^{-\alpha}\norma{Yg\circ e^{tX}-Yg}_{L^2(\sph{d-1})}\lesssim \prod_{j=1}^{m+1}\norma{f_j}_{\mathcal H^s}. \]
	This implies the desired $\mathcal H^s$-bound for the function $g$, and concludes the proof of the lemma.
\end{proof}

The following result details a sense in which $\Mop$ can be viewed as a smoothing operator, but requires $(d,m)\notin\EEC$. 

\begin{lemma}\label{lem:MLambdaBound}
	Given $(d,m)\in\mathfrak{U}\setminus \EEC$, set $\alpha_{d,m}=\frac{1}{2}(d-1)(m-2)-1$. If $\alpha\in(0,1)$ is such that $\alpha\leq\alpha_{d,m}$,  $\{\vphi_j\}_{j=1}^m\subset C^\infty(\sph{d-1})$, and $g\in L^2(\sph{d-1})$, then $\Mop(\vphi_{1},\dotsc,\vphi_{m},g)\in \mathcal H^{\alpha}$.
	 Moreover, the following estimate holds:
	\begin{multline}\label{eq:HnormMsmoothcase}
	\norma{\Mop(\vphi_{1},\dotsc,\vphi_{m},g)}_{\mathcal H^{\alpha}}\\
	\lesssim\biggl(\prod_{j=1}^m\norma{\vphi_j}_{L^2(\sph{d-1})}+\norma{\vphi_1\sigma_{d-1}\ast\dotsm\ast\vphi_m\sigma_{d-1}}_{\Lambda_{\alpha_{d,m}}(\R^d)}\biggr)\norma{g}_{L^2(\sph{d-1})}.
	\end{multline}
\end{lemma}

\noindent It is natural to wonder whether a similar gain in regularity holds in the case when $(d,m)\in\EEC$.
The (affirmative) answer is more subtle, and
we postpone the discussion until \S\ref{sec:HsEEC}; see Lemma \ref{cor:EEC} below.

\begin{proof}[Proof of Lemma \ref{lem:MLambdaBound}]
	Recall that $\alpha_{d,m}>0$ since $(d,m)\in\mathfrak{U}\setminus \EEC$. 
	It then follows from Proposition \ref{lem:HolderregularityConvo} that $\vphi_1\sigma_{d-1}\ast\dotsm\ast\vphi_m\sigma_{d-1}\in\Lambda_{\alpha_{d,m}}(\R^d)$. 
	For notational convenience, we  shall only consider the special case when $\vphi_j=\vphi$, for all $j$.
	Given $\Theta\in\textup{SO}(d)$ and $\omega\in\mathbb{S}^{d-1}$, estimate: 
	\begin{align*}
	\vert\Mop(\vphi,\dotsc,\vphi,g)\circ&\Theta(\omega)-\Mop(\vphi,\dotsc,\vphi,g)(\omega)\vert\\
	&\leq \int_{\mathbb{S}^{d-1}} \Bigl\vert(\vphi\sigma_{d-1})^{\ast m}(\Theta\omega-\eta)-(\vphi\sigma_{d-1})^{\ast m}(\omega-\eta)\Bigr\vert \ab{g(\eta)}\d\sigma_{d-1}(\eta).
	\end{align*}
	If $\alpha\in(0,1)$ is such that $\alpha\leq \alpha_{d,m}$, then $(\vphi\sigma_{d-1})^{\ast m}\in \Lambda_\alpha(\R^d)$, and consequently
	\begin{align*}
	\vert\Mop(\vphi,\dotsc,\vphi,g)\circ\Theta(\omega)-\Mop(\vphi,\dotsc,&\vphi,g)(\omega)\vert\\
	&\leq \ab{(\Theta-I)\omega}^\alpha\norma{(\vphi\sigma_{d-1})^{\ast m}}_{\Lambda_\alpha(\R^d)}\norma{g}_{L^1(\sph{d-1})}\\
	&\lesssim\ab{\Theta-I}^\alpha\norma{(\vphi\sigma_{d-1})^{\ast m}}_{\Lambda_\alpha(\R^d)}\norma{g}_{L^2(\sph{d-1})}.
	\end{align*}
	Letting $\Theta=e^{tX}$ for some $X\in\{X_{i,j}\}_{1\leq i<j\leq d}$, and integrating the square of both sides of the latter estimate, we obtain
	\begin{multline*}
	\sup_{\ab{t}\leq 1}\ab{t}^{-\alpha}\norma{\Mop(\vphi,\dotsc,\vphi,g)\circ\Theta-\Mop(\vphi,\dotsc,\vphi,g)}_{L^2(\mathbb{S}^{d-1})}\\\lesssim\sup_{\ab{t}\leq 1}\ab{t}^{-\alpha}\ab{e^{tX}-I}^\alpha\norma{(\vphi\sigma_{d-1})^{\ast m}}_{\Lambda_\alpha(\R^d)}\norma{g}_{L^2(\sph{d-1})}.
	\end{multline*}
	In turn, this and the basic $L^2$-estimate \eqref{eq:basicL2} together imply
	\[ \norma{\Mop(\vphi,\dotsc,\vphi,g)}_{\mathcal H^\alpha}
	\lesssim(\norma{\vphi}_{L^2(\sph{d-1})}^{m}+\norma{(\vphi\sigma_{d-1})^{\ast m}}_{\Lambda_\alpha(\R^d)})\norma{g}_{L^2(\sph{d-1})}. \]
	To obtain \eqref{eq:HnormMsmoothcase}, simply rerun the argument with the $\varphi_j$'s in place of $\varphi$.
	This completes the proof of the lemma.
\end{proof}

\section{H\"older regularity}\label{sec:Holder}

In this section, we prove H\"older-type estimates for certain convolution measures, which will pave the way towards finding a suitable replacement for Lemma \ref{lem:MLambdaBound} in the case when $(d,m)\in\EEC$.

\subsection{Two-fold convolutions}\label{subseq:weakHolder}
The purpose of this subsection is to generalize \cite[Lemma 2.3]{CS12b} to arbitrary dimensions $d\geq 2$.
While for the most part  the analysis follows similar lines to those of \cite{CS12b}, we include it for the sake of completeness.
Start by recalling that the 2-fold convolution $\sigma_{d-1}\ast \sigma_{d-1}$ defines a measure supported on the ball $B_2\subset\R^d$, which is absolutely continuous with respect to Lebesgue measure on $B_2$, and whose Radon--Nikodym derivative equals
\begin{equation}\label{eq:2fold}
(\sigma_{d-1}\ast \sigma_{d-1})(x) 
=\frac{\omega_{d-2}}{2^{d-3}}\frac{1}{|x|} (4-|x|^2)_+^{\frac{d-3}2}.
\end{equation}
Here, $\omega_{d-2}:=\sigma_{d-2}(\sph{d-2})=2 \pi^{\frac{d-1}2} \Gamma(\frac{d-1}2)^{-1}$ denotes the surface area of $\sph{d-2}$, $y_+:=\max\{0,y\}$ for $y\in\R$, and 
\[(4-\ab{x}^2)_+^{\frac{d-3}{2}}:=\bigl((4-\ab{x}^2)_+\bigr)^{\frac{d-3}{2}};\]
see for instance \cite[Lemma 5]{COS15}.

Let $h_1, h_2\in \textup{Lip}(\sph{d-1})$.
From \cite[Appendix A.2]{FOS17}, we know that the function $u_{12}$ defined by the relation
$(h_1\sigma_{d-1}\ast h_2\sigma_{d-1})(x)=u_{12}(x)(\sigma_{d-1}\ast\sigma_{d-1})(x)$, for $0<\ab{x}\leq 2$, and $u_{12}(x)=0$ for $\ab{x}>2$, can be expressed as
\begin{equation}\label{eq:defu12alld}
u_{12}(x)=\fint_{\Gamma_x} h_1(\nu) h_2(x-\nu)\,\d\sigma_x(\nu),
\end{equation}
where $\Gamma_x=\sph{d-1}\cap(x+\sph{d-1})$, and $\fint$ denotes the averaged integral on the $(d-2)$-dimensional sphere $\Gamma_x$; see also \cite{CS12a} for a careful discussion of the case $d=3$. 

The case $d=2$ merits some further remarks. In this case, if  $0<\ab{x}<2$, then $\Gamma_x$ consists of two points, which we identify with $\sph{0}$. Let $x^\perp$ be the $90^\circ$-counterclockwise rotation of $x$, so that $x^\perp\cdot x=0$ and $\ab{x^\perp}=\ab{x}$. Given $x\in B_2\setminus\{0\}\subset\R^2$, there exist unique-up-to-permutation $x_1,x_2\in\sph{1}$, such that $x=x_1+x_2$. The vectors $x_1,x_2$ are explicitly given by
\[ x_1=\frac{x}{2}+\biggl(1-\frac{\ab{x}^2}{4}\biggr)^{\frac{1}{2}}\frac{x^\perp}{\ab{x}},\quad x_2=\frac{x}{2}-\biggl(1-\frac{\ab{x}^2}{4}\biggr)^{\frac{1}{2}}\frac{x^\perp}{\ab{x}}. \]
Given $h_1,h_2\in \mathrm{Lip}(\sph{1})$, the convolution $h_1\sigma_1\ast h_2\sigma_1$ can be written in the following way: if 
$0<\ab{x}\leq 2$, then
\begin{align*}
(h_1\sigma_1\ast h_2\sigma_1)(x)
&=2\frac{h_1(x_1)h_2(x_2)+h_1(x_2)h_2(x_1)}{\ab{x}\sqrt{4-\ab{x}^2}},
\end{align*}
while for $\ab{x}>2$ one obviously has that $(h_1\sigma_1\ast h_2\sigma_1)(x)=0$.
 In this case, identity  \eqref{eq:defu12alld} is then seen to reduce to
\[u_{12}(x)=\frac{1}{2}(h_1(x_1)h_2(x_2)+h_1(x_2)h_2(x_1)),\text{ if } 0<\ab{x}\leq 2.\]

\begin{lemma}\label{lem:weakHolderd}
	Let $d\geq 2$, and $x,x'\in B_2\setminus\{0\}\subset\R^d$. Then
	\[ \ab{u_{12}(x)-u_{12}(x')}\leq C\norma{h_1}_{\textup{Lip}(\sph{d-1})}\norma{h_2}_{\textup{Lip}(\sph{d-1})}\biggl(\ab{x-x'}^{1/2}+\Abs{\frac{x}{\ab{x}}-\frac{x'}{\ab{x'}}}\biggr), \]
	for some universal constant $C<\infty$.
\end{lemma}
\begin{proof}
The integral \eqref{eq:defu12alld} defining $u_{12}$ can be equivalently written as
	\[u_{12}(x)=\omega_{d-2}^{-1}\int_{\sph{d-2}_x} h_1(\tfrac{x}{2}+\rho(x)\omega) h_2(\tfrac{x}{2}-\rho(x)\omega)\,\d\sigma_{d-2}(\omega),  \] 
	where the function $\rho\geq0$ satisfies $\rho(x)^2+(\ab{x}/2)^2=1$, and the unit sphere $\sph{d-2}_x$ is contained in the $(d-1)$-dimensional subspace of $\R^d$ orthogonal to $x$, and is therefore parallel to the hyperplane containing $\Gamma_x$. It is elementary to check that $\ab{\rho(x)-\rho(x')}\leq\ab{\ab{x}-\ab{x'}}^{1/2}$, for every $x,x'\in B_2\setminus\{0\}$. 
	
	Let us start by considering the case $x'=\la x$, for some $\la>0$. We then have that $\sph{d-1}_x=\sph{d-1}_{x'}$, and so
	\begin{align*}
	\abs{\Bigl(\frac{x}{2}+\rho(x)\omega\Bigr)-\Bigl(\frac{x'}{2}+\rho(x')\omega\Bigr)}\leq \frac{1}{2}\ab{x-x'}+\ab{\rho(x)-\rho(x')}\lesssim \ab{x-x'}^{1/2}.
	\end{align*}
	In a similar way,
	\[\Big|{\Big(\frac{x}{2}-\rho(x)\omega\Big)-\Big(\frac{x'}{2}-\rho(x')\omega\Big)}\Big|\lesssim \ab{x-x'}^{1/2}.\]
	Denote $x_1=\frac{x}{2}+\rho(x)\omega,x_2=\frac{x}{2}-\rho(x)\omega$, and $\tilde u(x)=h_1(x_1)h_2(x_2)$. Since $h_1,h_2$ are Lipschitz functions, we have that
	\begin{align*}
	\ab{\tilde u(x)-\tilde u(x')}&=\ab{h_1(x_1)h_2(x_2)-h_1(x_1')h_2(x_2')}\\
	&\leq \ab{h_2(x_2)}\ab{h_1(x_1)-h_1(x_1')}+\ab{h_1(x_1')}\ab{h_2(x_2)-h_2(x_2')}\\
	&\leq \norma{h_2}_{L^\infty}\norma{h_1}_{\textrm{Lip}}\ab{x_1-x_1'}+\norma{h_1}_{L^\infty}\norma{h_2}_{\textrm{Lip}}\ab{x_2-x_2'}\\
	&\lesssim\norma{h_1}_{\textrm{Lip}}\norma{h_2}_{\textrm{Lip}}\ab{x-x'}^{1/2}.
	\end{align*}
	It then follows by integration over $\sph{d-2}_x$ that 
	\[\ab{u_{12}(x)-u_{12}(x')}\lesssim \norma{h_1}_{\textrm{Lip}}\norma{h_2}_{\textrm{Lip}}\ab{x-x'}^{1/2}.\]
	
	We now consider the case $\ab{x}=\ab{x'}\in (0,2]$. 
	We then have that $\rho(x)=\rho(x')$. Let $\Theta\in \mathrm{SO}(d)$ denote a rotation that fixes the space $(\operatorname{span}\{x,x'\})^\perp$ and sends $x/\ab{x}$ to $x'/\ab{x'}$. It is not difficult to see that $\ab{\Theta-I}\leq \ab{x/\ab{x}-x'/\ab{x'}}$. We can then write 
	\[ u_{12}(x')=\omega_{d-2}^{-1}\int_{\sph{d-2}_x} h_1(\tfrac{x'}{2}+\rho(x)\Theta\omega) h_2(\tfrac{x'}{2}-\rho(x')\Theta\omega)\,\d\sigma_{d-2}(\omega), \]
	so that, for $\epsilon\in\{-1,1\}$,
	\begin{align*}
	\abs{\Bigl(\frac{x}{2}+\epsilon\rho(x)\omega\Bigr)-\Bigl(\frac{x'}{2}+\epsilon\rho(x')\Theta\omega\Bigr)}&\leq \frac{1}{2}\ab{x-x'}+\rho(x)\ab{(\Theta-I)\omega}\\
	&\leq \frac{1}{2}\ab{x-x'}+\rho(x)\Abs{\frac{x}{\ab{x}}-\frac{x'}{\ab{x'}}}\\
	&= \Bigl(\rho(x)+\frac{\ab{x}}{2}\Bigr)\Abs{\frac{x}{\ab{x}}-\frac{x'}{\ab{x'}}}\lesssim \Abs{\frac{x}{\ab{x}}-\frac{x'}{\ab{x'}}}.
	\end{align*}
 Reasoning as before, we conclude that
	\[ \ab{u_{12}(x)-u_{12}(x')}\lesssim \norma{h_1}_{\textrm{Lip}}\norma{h_2}_{\textrm{Lip}}\Abs{\frac{x}{\ab{x}}-\frac{x'}{\ab{x'}}}. \]
	For general $x,x'\in B_2\setminus\{0\}$ we proceed as follows. Let $y=\ab{x}x'/\ab{x'}$, so that $\ab{y}=\ab{x}$ and $x'=\la y$ for $\la=\ab{x'}/\ab{x}>0$. Then
	\begin{align*}
	\ab{u_{12}(x)-u_{12}(x')}&\leq \ab{u_{12}(x)-u_{12}(y)}+\ab{u_{12}(y)-u_{12}(x')}\\
	&\lesssim \norma{h_1}_{\textrm{Lip}}\norma{h_2}_{\textrm{Lip}}\biggl(\ab{x'-y}^{1/2}+\Abs{\frac{y}{\ab{y}}-\frac{x}{\ab{x}}}\biggr)\\
	&=\norma{h_1}_{\textrm{Lip}}\norma{h_2}_{\textrm{Lip}}\biggl(\ab{\ab{x}-\ab{x'}}^{1/2}+\Abs{\frac{x'}{\ab{x'}}-\frac{x}{\ab{x}}}\biggr)\\
	&\leq \norma{h_1}_{\textrm{Lip}}\norma{h_2}_{\textrm{Lip}}\biggl(\ab{x-x'}^{1/2}+\Abs{\frac{x}{\ab{x}}-\frac{x'}{\ab{x'}}}\biggr).
	\end{align*}
	This completes the proof of the lemma.
\end{proof}
The following consequence of Lemma \ref{lem:weakHolderd} will be useful in the forthcoming analysis. 
\begin{corollary}\label{cor:multByX}
	Let $d\geq 3$, and $x,x'\in B_2\setminus\{0\}\subset \R^d$. Then
	\begin{multline*} 
	\abs{\ab{x}(h_1\sigma_{d-1}\ast h_2\sigma_{d-1})(x)-\ab{x'}(h_1\sigma_{d-1}\ast h_2\sigma_{d-1})(x') }\\\leq C\norma{h_1}_{\textup{Lip}(\sph{d-1})}\norma{h_2}_{\textup{Lip}(\sph{d-1})}\biggl(\ab{x-x'}^{1/2}+\Abs{\frac{x}{\ab{x}}-\frac{x'}{\ab{x'}}}\biggr), 
	\end{multline*}
	for some universal constant $C<\infty$.
\end{corollary}
\begin{proof}
	From \eqref{eq:2fold} and \eqref{eq:defu12alld}, for $\ab{x}\leq 2$ we have that
	\[ \ab{x}(h_1\sigma_{d-1}\ast h_2\sigma_{d-1})(x)=2^{-d+3}\omega_{d-2}(4-\ab{x}^2)^{\frac{d-3}{2}}u_{12}(x). \]
	The function $(4-\ab{x}^2)^{\frac{d-3}{2}} \mathbbm{1}_{B_2}(x)$ belongs to $\Lambda_{1/2}(\R^d)$ if $d\geq 4$, and to $\Lambda_{1/2}(B_2)$ if $d=3$. The desired conclusion follows easily from this and Lemma \ref{lem:weakHolderd}.
\end{proof}

\subsection{The case $(d,n)=(3,3)$.}
In the course of this subsection only, we shall simplify the notation by writing $\d\sigma=\d\sigma_2$.
Our goal is to establish a H\"older estimate for the 3-fold convolution $h_1\sigma\ast h_2\sigma\ast h_3\sigma$, where $\{h_j\}_{j=1}^3$ are Lipschitz functions on the unit sphere $\sph{2}$. 
	
\begin{proposition}\label{prop:lip3S2}
	Given  $h_1,h_2,h_3\in \textup{Lip}(\sph{2})$,
	let
	$H=h_1\sigma\ast h_2\sigma\ast h_3\sigma$.
	Then there exists a universal constant $C<\infty$ such that, for every $x,x'\in\R^3$,
	\[|H(x)-H(x')|\leq C\prod_{j=1}^3\|h_j\|_{\textup{Lip}(\sph{2})}|x-x'|^{1/3}.\]
\end{proposition}

\begin{proof}
	By homogeneity, we may assume $\|h_j\|_{\textup{Lip}}=1$, $1\leq j\leq 3$.
	Since the function $H$ is compactly supported, it is enough to consider $x,x'\in\R^3$ for which\footnote{We will write $|x-x'|\ll 1$ to mean that the quantity $|x-x'|$ is sufficiently small for the purposes of the corresponding proof. For instance, in the course of the proof of Proposition \ref{prop:lip3S2}, we can and will assume that $|x-x'|\leq {100}^{-1}$.} $|x-x'|\ll 1$.
	From \eqref{eq:2fold} and \eqref{eq:defu12alld}, the function 
	$u_{12}(x):=(2\pi)^{-1}|x|(h_1\sigma\ast h_2\sigma)(x)$ is given by
	\begin{equation}\label{eq:defu12}
	u_{12}(x)=\fint_{\Gamma_x} h_1(\nu) h_2(x-\nu)\,\d\sigma_x(\nu),
	\end{equation}
	where $\Gamma_x=\sph{2}\cap(x+\sph{2})$. 
	We further have that
	\begin{align*}
	H(x)
	&=\int_{\sph{2}} (h_1\sigma\ast h_2\sigma)(x-\omega)h_3(\omega)\,\d \sigma(\omega)
	=2\pi\int_{\sph{2}}\frac{\mathbbm{1}_{|x-\omega|<2}(\omega)}{|x-\omega|} u_{12}(x-\omega)h_3(\omega)\,\d\sigma(\omega),
	\end{align*}
	and so
	\begin{multline*}
	(2\pi)^{-1}(H(x)-H(x'))
	=\int_{\sph{2}} \frac{\mathbbm{1}_{|x'-\omega|<2}(\omega)}{|x'-\omega|}\left(u_{12}(x-\omega)-u_{12}(x'-\omega)\right) h_3(\omega)\,\d \sigma(\omega)\\
	+\int_{\sph{2}} \left(\frac{\mathbbm{1}_{|x-\omega|<2}(\omega)}{|x-\omega|}-\frac{\mathbbm{1}_{|x'-\omega|<2}( \omega)}{|x'-\omega|}\right)u_{12}(x-\omega) h_3(\omega)\,\d \sigma(\omega).
	\end{multline*}
	We denote the integrals on the right-hand side of the latter identity by $I$ and $II$, respectively.
	We start by estimating the first integral.	\\
	
	{\bf Estimating $I$}. The fist step is to restrict the domain of integration to the region where $x-\omega,x'-\omega\in B_2$, plus a remainder which is $O(\ab{x-x'})$. With this purpose in mind, decompose $\sph{2}=U\cup U'\cup V\cup W$, where 

	\begin{gather}
	\label{eq:decomposeB2d3}
	U:=\{\omega\in \sph{2}: |x'-\omega|<2\leq |x-\omega|\},\, U':=\{\omega\in \sph{2}: |x-\omega|<2\leq |x'-\omega|\} ,\\
	\nonumber
 V:=\{\omega\in \sph{2}: |x-\omega|,|x'-\omega|<2 \},\,	W:=\{\omega\in \sph{2}: 2\leq  |x'-\omega|,|x-\omega| \}.
	\end{gather}
	The integrand of $I$ vanishes on the region $U' \cup W$, and so  we are left to analyze the integrals over $U$ and $V$. We claim that $\sigma(U)=O(\ab{x-x'})$. Indeed, if $\omega\in U$, then $\ab{x'-\omega}<2\leq \ab{x-\omega}$, so that as $\ab{x'-\omega}\geq \ab{x-\omega}-\ab{x-x'}\geq 2-\ab{x-x'}$ we obtain
	\begin{equation}\label{eq:UcontainedAnnulus}
	U\subseteq \{\omega\in\sph{2}\colon 2-\ab{x-x'}\leq \ab{x'-\omega}\leq 2 \}.
	\end{equation}
	This shows that the region $U$ is contained in the intersection of $\sph{2}$ with a spherical shell of thickness $\ab{x-x'}$ centered at $x'$. The claim follows.
		The contribution of $U$ to the integral $I$  can then be bounded in the following way:
	\begin{align*}
	\int_{U}\frac{\mathbbm{1}_{|x'-\omega|<2}(\omega)}{|x'-\omega|} \ab{u_{12}(x'-\omega) h_{3}(\omega)}\,\d\sigma(\omega)
	\leq
	\int_{U}\frac{\mathbbm{1}_{|x'-\omega|<2}(\omega)}{|x'-\omega|} \d \sigma(\omega).
	\end{align*}
	If $\omega\in U$, then $|x'-\omega|\geq 2-|x-x'|>1$ since $|x-x'|\ll1$.
	As a consequence, the latter integral can be crudely bounded as follows
	\begin{equation}\label{eq:boundOnU}
	\int_{U} \frac{\d \sigma(\omega)}{|x'-\omega|}\leq \sigma(U)\lesssim |x-x'|.
	\end{equation}
	To handle the contribution of the region $V$, note that Lemma \ref{lem:weakHolderd} implies the pointwise estimate
	\begin{equation}\label{eq:ChristShao}
	|u_{12}(x-\omega)-u_{12}(x'-\omega)|\lesssim |x-x'|^{1/2}+\left|\frac{x-\omega}{|x-\omega|}-\frac{x'-\omega}{|x'-\omega|}\right|.
	\end{equation}
	The contribution of the region 
	\begin{equation}\label{eq:defR}
	R:=\bigg\{\omega\in V: \left|\frac{x-\omega}{|x-\omega|}-\frac{x'-\omega}{|x'-\omega|}\right|\leq|x-x'|^{1/2}\bigg\}
	\end{equation}
	to the integral $I$ is easy to estimate.
	In view of \eqref{eq:ChristShao} and \eqref{eq:defR}, 
	\begin{align*}
	\Big|\int_{R} \frac{\mathbbm{1}_{|x'-\omega|<2}(\omega)}{|x'-\omega|}&\left(u_{12}(x-\omega)-u_{12}(x'-\omega)\right) h_3(\omega)\,\d\sigma(\omega) \Big|\\
	&\lesssim  \left(\int_R \frac{\mathbbm{1}_{|x'-\omega|<2}(\omega)}{|x'-\omega|}\,\d \sigma(\omega)\right)  |x-x'|^{1/2}
	\lesssim  |x-x'|^{1/2}.
	\end{align*}
	In the second estimate, we used the elementary fact that there exists a universal constant $C<\infty$, such that 
	$$\int_R \frac{\mathbbm{1}_{|x'-\omega|<2}(\omega)}{|x'-\omega|}\,\d \sigma(\omega)\leq
	\int_{\sph{2}} \frac{\d \sigma(\omega)}{|x'-\omega|}\leq C<\infty,$$
	 for all $x'\in\R^3$.
	If $\omega\in V\setminus R$, then 
	\begin{equation}\label{eq:smallThanSqrt}
	|x-x'|^{1/2}<\left|\frac{x-\omega}{|x-\omega|}-\frac{x'-\omega}{|x'-\omega|}\right|
	\leq \frac{2\ab{x-\omega}\ab{x-x'}}{|x-\omega|\,|x'-\omega|},
	\end{equation}
	from where we obtain $|x'-\omega|\leq 2|x-x'|^{1/2}$. The contribution of this region can then be estimated as follows:
	\begin{align*}
	\Big|\int_{V\setminus R} \frac{\mathbbm{1}_{|x'-\omega|<2}(\omega)}{|x'-\omega|}&\left(u_{12}(x-\omega)-u_{12}(x'-\omega)\right)h_3(\omega)\,\d \sigma(\omega)\Big|\\
	&\lesssim\int_{V\setminus R} \frac{\mathbbm{1}_{|x'-\omega|<2}(\omega)}{|x'-\omega|}\left|\frac{x-\omega}{|x-\omega|}-\frac{x'-\omega}{|x'-\omega|}\right| h_3(\omega)\,\d \sigma(\omega)\\
	&\lesssim  \left(\int_{\sph{2}\cap B(x', 2|x-x'|^{1/2})} \frac{\mathbbm{1}_{|x'-\omega|<2}(\omega)}{|x'-\omega|}\,\d \sigma(\omega)\right) \|h_3\|_{L^\infty} \\
	&\lesssim  |x-x'|^{1/2}.
	\end{align*}
	From the third to the fourth line, we used the fact that
	\begin{equation}\label{eq:boundOnSmallBall}
	\phi(x'):=\int_{\sph{2}\cap B(x',\eps)} \frac{\d \sigma(\omega)}{ |x'-\omega|}
	\end{equation}
	defines a radial function of $x'$ which satisfies  
	\[\phi(x')\lesssim {\sigma(\sph{2}\cap B(x',\eps) )}^{1/2}\lesssim \eps.\] 
	This concludes the verification of the bound $|I|\lesssim|x-x'|^{1/2}$.\\
	
	{\bf Estimating $II$.}
	The integral $II$ is bounded by
	$$\int_{\sph{2}} \left|\frac{\mathbbm{1}_{|x-\omega|<2}(\omega)}{|x-\omega|}-\frac{\mathbbm{1}_{|x'-\omega|<2}(\omega)}{|x'-\omega|}\right|\,\d\sigma (\omega).$$
	By symmetry, it is enough to consider
	\begin{equation}\label{eq:aftersymmetry}
	\int_{T} \left(\frac{\mathbbm{1}_{|x-\omega|<2}(\omega)}{|x-\omega|}-\frac{\mathbbm{1}_{|x'-\omega|<2}(\omega)}{|x'-\omega|}\right)\,\d\sigma(\omega),
	\end{equation}
	where the integral is taken over the region 
	$$T:=\biggl\{\omega\in \sph{2}:\frac{\mathbbm{1}_{|x-\omega|<2}(\omega)}{|x-\omega|}>\frac{\mathbbm{1}_{|x'-\omega|<2}(\omega)}{|x'-\omega|}\biggr\}.$$
	Decompose $T=U''\cup V''$, where
	\begin{align*}
	U''&:=\{\omega\in T: |x-\omega|<2\leq |x'-\omega|\},\\
	V''&:=\{\omega\in T: |x-\omega|<|x'-\omega|<2 \}.
	\end{align*}
	We have that $U''=U'\cap T$, and therefore  $\sigma(U'')=O(|x-x'|)$.
	Moreover,
	$$\int_{U''} \left(\frac{\mathbbm{1}_{|x-\omega|<2}(\omega)}{|x-\omega|}-\frac{\mathbbm{1}_{|x'-\omega|<2}(\omega)}{|x'-\omega|}\right)\d\sigma(\omega) 
	=
	\int_{U''} \frac{\d \sigma(\omega)}{|x-\omega|}\lesssim\ab{x-x'},$$
	where the last inequality follows as in \eqref{eq:boundOnU}.
	The contribution of the region $V''$ to the integral in \eqref{eq:aftersymmetry} is slightly more delicate to estimate.
	We consider two cases as before. Outside the ball $|x'-\omega|\geq |x-x'|^{1/3}$, 
	we use the estimate $\ab{x-\omega}\geq\ab{x'-\omega}-\ab{x-x'}\gtrsim |x-x'|^{1/3}$, which implies
	\begin{align*}
	\left|\frac{1}{|x-\omega|}-\frac{1}{|x'-\omega|}\right|&=\left\vert\frac{\ab{x'-\omega}-\ab{x-\omega}}{\ab{x-\omega}\ab{x'-\omega}}\right\vert\leq \frac{\ab{x'-x}}{\ab{x-\omega}\ab{x'-\omega}}\\
	&\lesssim |x-x'|^{-2/3} |x-x'|=|x-x'|^{1/3}.
	\end{align*}
	Inside the ball $|x'-\omega|\leq |x-x'|^{1/3}$, we also have $\ab{x-\omega}\leq \ab{x-x'}^{1/3}$, as $\omega\in V''$. The contribution of this region to the integral in \eqref{eq:aftersymmetry} is at most two times the integral
	$$\phi(x')=\int_{\sph{2}\cap B(x',\delta)} \frac{\d \sigma(\omega)}{|x'-\omega|},$$
	where $\delta=|x-x'|^{1/3}$. 
	Proceeding as in \eqref{eq:boundOnSmallBall}, one is led to the bound $\phi(x')\lesssim\delta$, whence the term in question is $O(|x-x'|^{1/3})$.
	This establishes the bound $|II|\lesssim |x-x'|^{1/3}$.
	The proof of the proposition is now complete.
\end{proof}

\begin{remark}\label{rem:higherconv}
	Proposition \ref{prop:lip3S2} implies that if $n\geq 4$, then $G_n:=h_1\sigma\ast \dots\ast h_n\sigma\in \Lambda_{1/3}(\R^3)$ whenever $\{h_j\}_{j=1}^3\subset \textrm{Lip}(\sph{2})$ and $\{h_j\}_{j=4}^n\subset L^1(\sph{2})$. This can be improved under the additional assumption $\{h_j\}_{j=1}^n\subset \textrm{Lip}(\sph{2})$, in which case we have, for instance, that $G_6\in \Lambda_{2/3}(\R^3)$.
	In dimensions $d\geq 4$, a similar argument to that in the proof of Proposition \ref{prop:lip3S2} shows that, if $\{h_j\}_{j=1}^3\subset \textrm{Lip}(\sph{d-1})$, then $h_1\sigma_{d-1}\ast h_{2}\sigma_{d-1}\ast h_3\sigma_{d-1}\in\Lambda_{\alpha}(\R^d)$, for some $\alpha>0$. Consequently, if $n\geq 3$ and $\{h_j\}_{j=1}^n\subset \textrm{Lip}(\sph{d-1})$, then $h_1\sigma_{d-1}\ast \dots\ast h_n\sigma_{d-1}\in\Lambda_{\alpha}(\R^d)$, for some $\alpha>0$.
\end{remark}

\subsection{The case $(d,n)=(2,4)$.}\label{sec:24}
In the course of this subsection only, we shall simplify the notation by writing $\d\sigma=\d\sigma_1$.
Our goal is to establish a H\"older-type estimate for the 4-fold convolution $h_1\sigma\ast h_2\sigma\ast h_3\sigma\ast h_4\sigma$, where $\{h_j\}_{j=1}^4$ are Lipschitz functions on the unit circle $\sph{1}$. 
We start with some preparatory work. 
As in \S \ref{subseq:weakHolder}, let
\begin{align}
u_{12}(x)&=\frac{1}{2}(h_1(x_1)h_2(x_2)+h_1(x_2)h_2(x_1))\mathbbm{1}_{B_2}(x),\label{eq:u12def}\\
u_{34}(x)&=\frac{1}{2}(h_3(x_1)h_4(x_2)+h_3(x_2)h_4(x_1))\mathbbm{1}_{B_2}(x),\label{eq:u34def}
\end{align}
 both of which satisfy the conclusion of Lemma \ref{lem:weakHolderd}. 
For brevity, we write 
\begin{equation}\label{eq:defF}
F(x):=(\sigma\ast\sigma)(x)=4\ab{x}^{-1}(4-\ab{x}^2)^{-1/2}\mathbbm{1}_{B_2}(x),
\end{equation}
as in \eqref{eq:2fold} with $d=2$. We will make repeated use of the upper bound
\begin{align}\label{eq:upperBoundSigma2}
\frac{1}{\ab{x}\sqrt{4-\ab{x}^2}}=\frac{\sqrt{4-\ab{x}^2}}{4\ab{x}}+\frac{\ab{x}}{4\sqrt{4-\ab{x}^2}}\leq \frac{1}{\ab{x}}+\frac{1}{\sqrt{2-\ab{x}}},\text{ for all }\ab{x}\leq 2,
\end{align}
together with the estimate
\begin{equation}\label{eq:4foldsigma1}
\sigma^{\ast 4}(x)\lesssim(1+\ab{\log\ab{x}})\mathbbm{1}_{B_4}(x),\text{ for all } x\in\R^2.
\end{equation}
Inequality \eqref{eq:4foldsigma1} follows from \cite[Eq.\@ (3.21)]{OSQ19}, and in particular implies that 
\begin{equation}\label{eq:LinftySigma4}
\ab{\cdot}^{\beta}\sigma^{\ast 4}\in L^\infty(\R^2)\text{, for every } \beta>0.
\end{equation}
Setting $H_\gamma(x)=\ab{x}^\gamma\bigl((u_{12}F)\ast (u_{34}F)\bigr)(x)$,
we then have that $H_\gamma\in L^\infty(\R^2)$, for any $\gamma>0$ and $\{h_j\}_{j=1}^4\subset L^\infty(\sph{1})$.
This will be used in Proposition \ref{prop:lip4circle} below. 
The following preparatory result  quantifies the smallness of the function $(\mathbbm{1}_E(\sigma\ast\sigma))\ast (\sigma\ast\sigma)$, for certain sets $E\subset\R^2$ of small Lebesgue measure.

\begin{lemma}\label{lem:smallIntegralFF}
Set $F=\sigma\ast\sigma$.  
Let $x\in B_4\subset\R^2$. 
Then, for every $\gamma\in(0,1]$ and $s\in(0,\frac{\gamma}{2(\gamma+1)})$, there exists a constant $C_{\gamma,s}<\infty$ such that, for all $\eps\in(0,1)$, 
	\begin{gather}
	\label{eq:ineqInU}
	\ab{x}^\gamma\int_{A(x,\eps)} F(y)F(x-y)\d y\leq C_{\gamma,s}\eps^{\min\{\frac{1}{6},\frac{\gamma}{2(\gamma+1)}-s\}},\\
	\label{eq:ineqInV}
	\ab{x}^\gamma\int_{B_2\cap B(x,\eps)} F(y)F(x-y)\d y\leq C_{\gamma,s}\eps^{\min\{\frac{1}{2},\gamma-s\}},
	\end{gather}
	where $A(x,\eps):=\{y\in B_2\colon 2-\eps\leq \ab{x-y}\leq 2 \}$. 
\end{lemma}
Before embarking in the proof of Lemma \ref{lem:smallIntegralFF}, we discuss a coordinate system which will prove convenient for the argument. Let $x\in\R^2$, $x\neq 0$ be given. A point $y\in\R^2$ is uniquely determined by the pair $(\ab{y},\ab{x-y})$, up to reflection with respect to the line spanned by $x$. This gives rise to the so-called (two-center) \textit{bipolar coordinates}, defined by $(r,s)=(\ab{y},\ab{x-y})$; see \cite[\S 2]{Fo00} and \cite[\S 2.2]{FOS17} for the use of this coordinate system in a related setting. The map $y\mapsto (r,s)=(\ab{y},\ab{x-y})$ is a two-to-one map from $\R^2\setminus \operatorname{span}\{x\}$ to the region determined by the relations $\ab{r-s}<\ab{x}<r+s$, whose Jacobian is given by
	\begin{equation}\label{eq:bipolarJacobian}
	\d y=\frac{2rs}{(\ab{x}^2-(r-s)^2)^{\frac{1}{2}}((r+s)^2-\ab{x}^2)^{\frac{1}{2}}}\d r\d s.
	\end{equation} 
After the change of variables $a=r-s$ $b=r+s$, the Jacobian becomes
	\begin{equation}\label{eq:bipolar2Jacobian}
	\d y=\frac{(a+b)(b-a)}{4(\ab{x}^2-a^2)^{\frac{1}{2}}(b^2-\ab{x}^2)^{\frac{1}{2}}}\d a\d b.
	\end{equation} 
\begin{proof}[Proof of Lemma \ref{lem:smallIntegralFF}]
	From \eqref{eq:LinftySigma4}, it follows that the left-hand sides of \eqref{eq:ineqInU}, \eqref{eq:ineqInV} define bounded functions of $x$, and therefore $\eps>0$ can be taken as small as needed in the argument below. We may also assume that $x\neq 0$ otherwise \eqref{eq:ineqInU} and \eqref{eq:ineqInV} are trivial.
	
	Let us start with \eqref{eq:ineqInU}. Note that $\ab{A(x,\eps)}\lesssim \eps$,
	and that if $y\in A(x,\eps)$, then $|x-y|\geq 2-\eps>1$.
	As a consequence, the left-hand side of \eqref{eq:ineqInU} can be bounded as follows:
	$$\ab{x}^\gamma\int_{A(x,\eps)} \frac{\d y}{|y|\sqrt{4-\ab{y}^2}\ab{x-y}\sqrt{4-|x-y|^2}}\lesssim\ab{x}^\gamma\int_{A(x,\eps)} \frac{\d y}{|y|\sqrt{4-\ab{y}^2}\sqrt{2-|x-y|}}.$$
	We then use the upper bound \eqref{eq:upperBoundSigma2},
	\begin{equation}\label{eq:decomposeSigma22}
	\frac{1}{\ab{y}\sqrt{4-\ab{y}^2}}\leq \frac{1}{\ab{y}}+\frac{1}{\sqrt{2-\ab{y}}},
	\text{ for }\ab{y}\leq 2,
	\end{equation}
	and are left to analyze the following integrals:
	\[ \phi_1(x,\eps):=\ab{x}^\gamma\int_{A(x,\eps)}\frac{\d y}{\ab{y}\sqrt{2-\ab{x-y}}},\quad \phi_2(x,\eps):=\ab{x}^\gamma\int_{A(x,\eps)}\frac{\d y}{\sqrt{2-\ab{y}}\sqrt{2-\ab{x	-y}}}.\]
	{\bf Analysis of $\phi_1(x,\eps)$.}
	We perform a dyadic decomposition of $A(x,\eps)$ via
	\begin{equation}\label{eq:dyadicDecAnnulus}
	A_j=\{y\in B_2\colon 2-2^{-j}\eps\leq \ab{x-y}\leq 2-2^{-(j+1)}\eps \},\quad j\in\N_0,
	\end{equation}
so that
	\begin{align*}
	\phi_1(x,\eps)=\ab{x}^\gamma\int_{A(x,\eps)}\frac{\d y}{\ab{y}\sqrt{2-\ab{x-y}}}
	&\simeq \ab{x}^\gamma\sum_{j=0}^\infty(2^{-j}\eps)^{-1/2}\int_{A_j}\frac{\d y}{\ab{y}}.
	\end{align*}
	Further, consider $\delta\in(0,\frac{1}{2})$ and decompose $A_j=A_{j,1}\cup A_{j,2}$ where $A_{j,1}=\{y\in A_j:\ab{y}> (2^{-j}\eps)^{1/2-\delta}\}$ and $A_{j,2}=\{y\in A_j:\ab{y}\leq  (2^{-j}\eps)^{1/2-\delta}\}=A_j\cap B_{(2^{-j}\eps)^{1/2-\delta}}$. Then the contribution of $\{A_{j,1}\}_{j\geq 0}$ to $\phi_1(x,\eps)$ can be bounded as follows:
\begin{equation}\label{eq:contribAj1}
\ab{x}^\gamma\sum_{j=0}^\infty(2^{-j}\eps)^{-1/2}\int_{A_{j,1}}\frac{\d y}{\ab{y}}\lesssim 
\sum_{j=0}^\infty(2^{-j}\eps)^{-1+\delta}\ab{A_{j,1}}\lesssim \sum_{j=0}^\infty(2^{-j}\eps)^{-1+\delta}2^{-j}\eps\lesssim_\delta\eps^{\delta},
\end{equation}
where we used that $\ab{A_{j,1}}\leq\ab{A_j}\lesssim 2^{-j}\eps$. We now proceed to bound the contribution of the $A_{j,2}$'s with the help of bipolar coordinates. We have
	\begin{align*}
	\int_{A_{j,2}}\frac{\d y}{\ab{y}}&\simeq \int_{\substack{
	\ab{r-s}<\ab{x}<r+s\\
	0\leq r\leq (2^{-j}\eps)^{1/2-\delta}\\
	2-2^{-j}\eps\leq s\leq 2-2^{-(j+1)}\eps
	}}
	\frac{s\d r\d s}{\sqrt{\ab{x}^2-(r-s)^2}\sqrt{(r+s)^2-\ab{x}^2}}\\
	&\simeq \int_{\substack{
	\ab{a}<\ab{x}<b\\
	0\leq \frac{a+b}{2}\leq (2^{-j}\eps)^{1/2-\delta}\\
	2-2^{-j}\eps\leq \frac{b-a}{2}\leq 2-2^{-(j+1)}\eps
	}}
	\frac{\d a\d b}{\sqrt{\ab{x}^2-a^2}\sqrt{b^2-\ab{x}^2}}\\
	&\simeq \int_{\substack{
	\ab{a}<\ab{x}<b\\
	0\leq \frac{a+b}{2}\leq (2^{-j}\eps)^{1/2-\delta}\\
	2-2^{-j}\eps\leq \frac{b-a}{2}\leq 2-2^{-(j+1)}\eps
	}}
	\frac{\d a\d b}{\sqrt{\ab{x}-\ab{a}}\sqrt{b-\ab{x}}}\\
	&\simeq \int_{\substack{
	\ab{a}<\ab{x}\\
	-2+2^{-(j+1)}\eps\leq a\leq -2+2^{-j}\eps+(2^{-j}\eps)^{1/2-\delta}
	}}\int_{\substack{
	b>\ab{x}\\
	-a\leq b\leq -a+2(2^{-j}\eps)^{1/2-\delta}\\
	a+4-2^{-j+1}\eps\leq b\leq a+4-2^{-j}\eps
	}}
	\frac{\d a\d b}{\sqrt{\ab{x}-\ab{a}}\sqrt{b-\ab{x}}}\\
	&\lesssim \max\{(2^{-j}\eps)^{1/2},(2^{-j}\eps)^{1/4-\delta/2} \}\min\{(2^{-j}\eps)^{1/2},(2^{-j}\eps)^{1/4-\delta/2} \}\\
	&=(2^{-j}\eps)^{1/2}(2^{-j}\eps)^{1/4-\delta/2},
	\end{align*}
	where we used that in the domain of integration $\ab{a}\simeq b\simeq 1$, so that $\sqrt{b+\ab{x}}\simeq 1$ and $\sqrt{\ab{x}+\ab{a}}\simeq 1$. Therefore, the contribution of $\{A_{j,2}\}_{j\geq 0}$ to $\phi_1(x,\eps)$ can be bounded as follows:
	\begin{align}
	\ab{x}^\gamma\sum_{j=0}^\infty(2^{-j}\eps)^{-1/2}\int_{A_{j,2}}\frac{\d y}{\ab{y}}&\lesssim \sum_{j\geq 0}(2^{-j}\eps)^{-1/2}(2^{-j}\eps)^{1/2}(2^{-j}\eps)^{1/4-\delta/2}\nonumber\\
	\label{eq:contribAj2}
	&= \sum_{j\geq 0}(2^{-j}\eps)^{1/4-\delta/2}\simeq_\delta \eps^{1/4-\delta/2}.
	\end{align}
	Taking $\delta=\frac16$,  we conclude from \eqref{eq:contribAj1} and \eqref{eq:contribAj2} that $\phi_1(x,\eps)\lesssim \eps^{1/6}$, which is an acceptable contribution, in the sense that it is smaller than a multiple of the right-hand side of \eqref{eq:ineqInU}.\\
	
	\noindent{\bf Analysis of $\phi_2(x,\eps)$.} The contribution of 
	 the region $A':=\{y\in A(x,\eps): \sqrt{2-\ab{y}}\geq \eps^\delta\}$ to $\phi_2(x,\eps)$ can be estimated as follows: 
	\begin{align}
	\ab{x}^\gamma\int_{A'}\frac{\d y}{\sqrt{2-\ab{y}}\sqrt{2-\ab{x-y}}}&\leq\ab{x}^\gamma \eps^{-\delta}\int_{A'}\frac{\d y}{\sqrt{2-\ab{x-y}}}
	\lesssim \eps^{-\delta}\int_{A(x,\eps)}\frac{\d y}{\sqrt{4-\ab{x-y}^2}}\nonumber\\
	\label{eq:yNear2}
	&=\eps^{-\delta}\int_0^{2\pi}\int_{2-\eps}^2 \frac{r}{\sqrt{4-r^2}}\d r\d\te
	\lesssim \eps^{\frac{1}{2}-\delta}.
	\end{align}
	If $y\in A'':=A(x,\eps)\setminus A'$, then $2-\eps^{2\delta}\leq \ab{y}\leq 2$ and $2-\eps\leq \ab{x-y}\leq 2$. Therefore $A''$ is contained in the intersection of two annuli of small thickness and located at distance comparable to $2$ from the origin. We may further assume that $\ab{x}\geq \eps^{\delta}$, for otherwise, given any $s\in(0,\gamma)$,
	\begin{equation}\label{eq:phi2SmallX}
	\phi_2(x,\eps)\leq (\eps^\delta)^{(\gamma-s)}\ab{x}^{s}\sigma^{\ast 4}(x)\lesssim_{s} \eps^{(\gamma-s){\delta}},
	\end{equation}
	so that $\phi_2(x,\eps)=O_\alpha(\eps^\alpha)$, for every $\alpha\in(0,\gamma\delta)$. We now apply the same dyadic decomposition of $A(x,\eps)$ as in \eqref{eq:dyadicDecAnnulus} together with a similar one on the second annulus,
	\[ D_k=\{y\in B_2\colon 2-2^{-k}\eps^{2\delta}\leq\ab{y}<2-2^{-(k+1)}\eps^{2\delta} \},\; k\in\N_0, \]
	so that $A''=\cup_{j,k\geq 0}A_j\cap D_k$. This yields
	\begin{equation}\label{eq:contribAdoubleprime}
	\ab{x}^\gamma\int_{A''}\frac{1}{\sqrt{2-\ab{y}}}\frac{1}{\sqrt{2-|x-y|}}\d y
	\lesssim\ab{x}^\gamma\eps^{-1/2-\delta}\sum_{j,k\geq 0} 2^{(j+k)/2}\ab{A_j\cap D_k}.
	\end{equation}
	We now use bipolar coordinates to bound $\ab{A_j\cap D_{k}}$. First consider the case where, in addition to $\ab{x}\geq \eps^\delta$, we have $\ab{x}\leq 4-\eps^\delta$, so that $A_{j}$ and $D_k$ intersect {\it transversely};  
	the intersection consists of two connected components which are symmetric with respect to the line spanned by $x$.
	Using bipolar coordinates, we have that
\begin{align}
\ab{A_j\cap D_{k}} 
	&\simeq 
	\int_{\substack{
	\ab{r-s}< \ab{x}<r+s\\
	2-2^{-k}\eps^{2\delta}\leq r\leq 2-2^{-(k+1)}\eps^{2\delta}\\
	2-2^{-j}\eps\leq s\leq 2-2^{-(j+1)}\eps
	}} \frac{rs\d r\ds}{\sqrt{\ab{x}^2-(r-s)^2}\sqrt{(r+s)^2-\ab{x}^2}}\nonumber\\
	\label{eq:bipolarAjDk}
	&\simeq\int_{\substack{
	\ab{a}< \ab{x}<b\\
	2-2^{-k}\eps^{2\delta}\leq \frac{a+b}{2}\leq 2-2^{-(k+1)}\eps^{2\delta}\\
	2-2^{-j}\eps\leq \frac{b-a}{2}\leq 2-2^{-(j+1)}\eps
	}} \frac{\d a\d b}{\sqrt{\ab{x}^2-a^2}\sqrt{b^2-\ab{x}^2}}.
\end{align}
Given $(a,b)$ in the domain of integration from \eqref{eq:bipolarAjDk}, it holds that $0\leq 4-b\leq \eps^{2\delta}$ and $\ab{a}\leq \eps^{2\delta}$, so that under the working assumption $\ab{x}\geq \eps^\delta$, we have $\ab{x}-\ab{a}\gtrsim_{\delta} \ab{x}$ and $\ab{x}+\ab{a}\simeq\ab{x}$. If, in addition, $\ab{x}\leq 4-\eps^\delta$, then $b-\ab{x}\gtrsim_{\delta} \eps^\delta$, and therefore \eqref{eq:bipolarAjDk} yields
\[ \ab{A_j\cap D_{k}}\lesssim_\delta \ab{x}^{-1}\eps^{-\delta}(2^{-j}\eps)(2^{-k}\eps^{2\delta})=\ab{x}^{-1}2^{-(j+k)}\eps^{1+\delta}, \]
and \eqref{eq:contribAdoubleprime} can be bounded as follows:
\begin{equation}\label{eq:contribAdpCase1}
\ab{x}^\gamma\int_{A''}\frac{1}{\sqrt{2-\ab{y}}}\frac{1}{\sqrt{2-|x-y|}}\d y
	\lesssim_\delta\ab{x}^{-(1-\gamma)}\sum_{j,k\geq 0}2^{-(j+k)/2}\eps^{1/2}\lesssim \eps^{\frac{1}{2}-\delta(1-\gamma)}.
\end{equation}
In the complementary case when $4-\eps^\delta\leq \ab{x}\leq 4$, we have $\ab{x}-\ab{a}\simeq 1$ and $b+\ab{x}\simeq 1$ in the domain of integration from \eqref{eq:bipolarAjDk}, so that
\begin{align*}
	\ab{A_j\cap D_{k}}&\simeq\int_{\substack{
	\ab{a}< \ab{x}<b\\
	2-2^{-k}\eps^{2\delta}\leq \frac{a+b}{2}\leq 2-2^{-(k+1)}\eps^{2\delta}\\
	2-2^{-j}\eps\leq \frac{b-a}{2}\leq 2-2^{-(j+1)}\eps
	}} \frac{\d a\d b}{\sqrt{b-\ab{x}}}\\
	&=\int_{\substack{
	b>\ab{x}\\
	b\geq 4-2^{-j}\eps-2^{-k}\eps^{2\delta}\\
	b\leq 4-2^{-(j+1)}\eps-2^{-(k+1)}\eps^{2\delta}
	}}\int_{\substack{
	\ab{a}< \ab{x}\\
	2-2^{-k}\eps^{2\delta}\leq \frac{a+b}{2}\leq 2-2^{-(k+1)}\eps^{2\delta}\\
	2-2^{-j}\eps\leq \frac{b-a}{2}\leq 2-2^{-(j+1)}\eps
	}}\d a \frac{1}{\sqrt{b-\ab{x}}} \d b\\
&\lesssim\min\{2^{-j}\eps,2^{-k}\eps^{2\delta}\}\int_{\max\{4-2^{-j}\eps-2^{-k}\eps^{2\delta},\ab{x} \}}^{4-2^{-(j+1)}\eps-2^{-(k+1)}\eps^{2\delta}}\frac{\d b}{\sqrt{b-\ab{x}}}\\
	&\lesssim \min\{2^{-j}\eps,2^{-k}\eps^{2\delta}\}\max\{(2^{-j}\eps)^{1/2},2^{-k/2}\eps^{\delta}\}\\
	&=2^{-(j+k)/2}\eps^{1/2+\delta}\min\{2^{-j/2}\eps^{1/2},2^{-k/2}\eps^{\delta}\}.
\end{align*}
In this way, \eqref{eq:contribAdoubleprime} is bounded as follows:
\begin{align}
\ab{x}^\gamma\int_{A''}\frac{1}{\sqrt{2-\ab{y}}}\frac{1}{\sqrt{2-|x-y|}}\d y
	&\lesssim\ab{x}^{\gamma}\sum_{j,k\geq 0}\min\{2^{-j/2}\eps^{1/2},2^{-k/2}\eps^{\delta}\}\nonumber\\
&\lesssim \sum_{j\geq 0}(2^{-j}\eps)^{1/2}+\sum_{j\geq 0}(2^{-j}\eps)^{1/2}\ab{\log(2^{-j}\eps^{1-2\delta})}\nonumber\\
&\lesssim_s\sum_{j\geq 0}(2^{-j}\eps)^{1/2}+\sum_{j\geq 0}\eps^{\delta}(2^{-j}\eps^{1-2\delta})^{1/2-s}\nonumber\\
\label{eq:contribAdpCase2}
&\lesssim_s\eps^{\frac{1}{2}-s},
\end{align}
for any $s\in (0,\frac{1}{2})$. In the passage from the first to the second line above, we split the sum in $k$ according to the partition $\N_0=\{k\in\N_0:k\geq \ab{\log_2(2^{-j}\eps^{1-2\delta})} \}\cup \{k\in\N_0:k< \ab{\log_2(2^{-j}\eps^{1-2\delta})}\}=:E_1\cup E_2$, respectively, where $\log_2(\cdot)$ denotes the base-2 logarithm. On $E_1$, the sum is a convergent geometric series whose value is proportional to the first term, hence it is $\lesssim(2^{-j}\eps)^{1/2}$. On $E_2$, the summands are constant, and the contribution to the sum equals $(2^{-j}\eps)^{1/2}\ab{E_2}\simeq(2^{-j}\eps)^{1/2}\ab{\log(2^{-j}\eps^{1-2\delta})}$.
 As a result of \eqref{eq:yNear2}, \eqref{eq:phi2SmallX}, \eqref{eq:contribAdpCase1} and \eqref{eq:contribAdpCase2}, we conclude the upper bound $\phi_2(x,\eps)\lesssim_{\delta,s} \max\{\eps^{\frac{1}{2}-\delta},\eps^{(\gamma-s)\delta}\}$, for every $\delta\in(0,\frac{1}{2})$ and $s\in(0,\gamma)$.
Optimizing in $\delta$, we are thus led to the estimate $\phi_2(x,\eps)\lesssim_s\eps^{\frac{\gamma}{2(\gamma+1)}-s}$, for every $s>0$.
	This concludes the verification of \eqref{eq:ineqInU}.\\
	
	To handle \eqref{eq:ineqInV}, start by noting that
	\[ \ab{x}^\gamma\int_{B_2\cap B(x,\eps)} F(y)F(x-y) \d y=16\ab{x}^\gamma\int_{B_2\cap B(x,\eps)} \frac{1}{\ab{y}\sqrt{4-\ab{y}^2}}\frac{1}{\ab{x-y}\sqrt{4-\ab{x-y}^2}} \d y. \]
	Since $\eps<1$, we may remove the term $\sqrt{4-\ab{x-y}^2}$ from the latter integrand at the expense of a universal constant. After an application of  \eqref{eq:decomposeSigma22},
	we are then left to study the following integrals:
	\begin{equation}\label{eq:phi3andphi4}
	\phi_3(x,\eps):=\ab{x}^\gamma\int_{B_2\cap B(x,\eps)}\frac{\d y}{\ab{y}\ab{x-y}},\quad
	\phi_4(x,\eps):=\ab{x}^\gamma\int_{B_2\cap B(x,\eps)}\frac{\d y}{\sqrt{2-\ab{y}}\ab{x-y}}.
	\end{equation}
	{\bf Analysis of $\phi_3(x,\eps)$.}
	Decompose the region of integration $B_2\cap B(x,\eps)=A_1\cup A_2$, where
	 \begin{align*}
	 A_1&:= B(x,\eps)\cap\{y\in B_2: \ab{y}\geq\eps^{1/2}\},\\
	 A_2&:= B(x,\eps)\cap\{y\in B_2: \ab{y}<\eps^{1/2}\}.
	 \end{align*}
	On the region $A_1$, we may simply estimate
	\begin{align*}
	\ab{x}^\gamma\int_{A_1}\frac{\d y}{|y| |x-y|}&\leq \ab{x}^\gamma\eps^{-1/2}\int_{B(x,\eps)}\frac{\d y}{\ab{x-y}}=2\pi \ab{x}^\gamma\eps^{-1/2}\eps\lesssim \eps^{1/2}.
	\end{align*}
	We further split $A_2=A_2'\cup A_2''$, with 
	\[A_2':=A_2\cap\{y: \ab{y}\geq\ab{x-y}\},\text{ and }A_2'':=A_2\cap\{y: \ab{y}<\ab{x-y}\}.\] 
	If $y\in A_2'$, then $\ab{y}\geq\frac{1}{2}\ab{x}$, and therefore
	\begin{align*}
	\ab{x}^\gamma\int_{A_2'} \frac{\d y}{|y| |x-y|}\lesssim \int_{A_2'} \frac{\d y}{\ab{y}^{1-\gamma}|x-y|}.
	\end{align*}
	Now, $\ab{y}^{-(1-\gamma)}\mathbbm{1}_{B_2}\in L^{p}(\R^2)$ for every $1\leq p< \frac{2}{1-\gamma}$, and $\ab{y-x}^{-1}\mathbbm{1}_{B_2}\in L^q(\R^2)$ for every $1\leq q<2$. Taking $2<p<\frac{2}{1-\gamma}$, its conjugate  satisfies $\frac{2}{1+\gamma}<p'<2$, and so by H\"older's inequality we have that
	\begin{align*}
	\int_{A_2'} \frac{\d y}{\ab{y}^{1-\gamma}|x-y|}&\leq \Bigl(\int_{B_{{\eps}^{1/2}}}\frac{\d y}{\ab{y}^{p(1-\gamma)}}\Bigr)^{\frac1p}\Bigl(\int_{B(x,\eps)}\frac{\d y}{\ab{x-y}^{p'}}\Bigr)^{\frac1{p'}}\\
	&=\tfrac{2\pi}{(2-p(1-\gamma))^{\frac1p}(2-p')^{\frac1{p'}}}\eps^{\frac{p(1+\gamma)-2}{2p}}.
	\end{align*}
	Note that $(p(1+\gamma)-2)/(2p)$ strictly increases to $\gamma$ as $p\to 2/(1-\gamma)^-$. In this way, we obtain
	\[ \ab{x}^\gamma\int_{A_2'} \frac{\d y}{\ab{y}|x-y|}\lesssim_s \eps^{\gamma-s}, \]
	for every $s\in(0,\gamma)$.
	If $y\in A_2''$, then $\ab{x-y}\geq \frac{1}{2}\ab{x}$ and $\ab{y}<\ab{x-y}\leq\eps$; in particular, $A_2''\subset B_\eps$. 
	Therefore, if $2<p<\frac{2}{1-\gamma}$, then
	\begin{align*}
	\ab{x}^\gamma\int_{A_2''} \frac{\d y}{|y| |x-y|}&\lesssim \int_{A_2''} \frac{\d y}{|y|\ab{x-y}^{1-\gamma}}\leq \Bigl(\int_{B_\eps}\frac{\d y}{\ab{y}^{p'}}\Bigr)^{\frac1{p'}}\Bigl(\int_{B(x,\eps)}\frac{\d y}{\ab{x-y}^{p(1-\gamma)}}\Bigr)^{\frac1p}\\
	&\lesssim \eps^{\frac{2-p'}{p'}}\eps^{\frac{2-p(1-\gamma)}{p}}=\eps^{\gamma}.
	\end{align*}
	We conclude that $\phi_3(x,\eps)\lesssim_s \eps^{\min\{\frac{1}{2},\gamma-s\}}$, for every $s\in(0,\gamma)$.\\
	
	\noindent{\bf Analysis of $\phi_4(x,\eps)$.}
	Proceeding as before, we decompose the region of integration $B_2\cap B(x,\eps)=D_1\cup D_2$, where
	\begin{align*}
	D_1&:= B(x,\eps)\cap \{y\in B_2\colon \sqrt{2-\ab{y}}\geq \eps^{1/2} \},\\
	D_2&:= B(x,\eps)\cap \{y\in B_2\colon \sqrt{2-\ab{y}}<\eps^{1/2}\}.
	\end{align*}
	On the region $D_1$, we may simply estimate
	\begin{align*}
	\ab{x}^\gamma\int_{D_1}\frac{\d y}{\sqrt{2-\ab{y}}\ab{x-y}}\leq \ab{x}^\gamma \eps^{-1/2}\int_{B(x,\eps)}\frac{\d y}{\ab{x-y}}\lesssim \eps^{1/2}.
	\end{align*}
	If $y\in D_2$, then $2-\eps< \ab{y}\leq 2$, and so $2-2\eps\leq \ab{x}\leq 2+\eps$.
	We may apply a dyadic decomposition,
	\[ V_j=\{y\in D_2\colon 2^{-(j+1)}\eps\leq\ab{x-y}\leq 2^{-j}\eps \},\; j\in \N_0, \]
	so that, letting $P(V_j)$ denote the image of $V_j$ under the polar coordinate map, and further writing $P(V_j)=\{(r,\te)\colon \te\in\Theta, r\in R(\te)\}$  for some $\Theta\subseteq[0,2\pi)$ and $R(\te)\subseteq [0,\infty)$, 
	we have that
	\begin{align*}
	\ab{x}^\gamma\int_{D_2}&\frac{\d y}{\sqrt{2-\ab{y}}\ab{x-y}}\lesssim\ab{x}^\gamma\sum_{j=0}^\infty 2^j\eps^{-1}\int_{V_j}\frac{\d y}{\sqrt{4-\ab{y}^2}}=\ab{x}^\gamma\sum_{j=0}^\infty 2^{j}\eps^{-1}\int_{P(V_j)}\frac{r\d r\d\te}{\sqrt{4-r^2}}\\
	&\lesssim\ab{x}^\gamma\sum_{j=0}^\infty 2^{j}\eps^{-1}\int_{\Theta}\ab{R(\theta)}^{1/2}\d\te
	\lesssim
	\ab{x}^\gamma\sum_{j=0}^\infty (2^{-j}\eps)^{-1} (2^{-j}\eps)^{1/2} (2^{-j}\eps)
	\lesssim \eps^{1/2}.
	\end{align*}
	In the second-to-last inequality, we used the fact that the length of the intersection of any line with the annulus  $V_j$ is $O(2^{-j}\eps)$ (so that $\ab{R(\te)}^{1/2}\lesssim(2^{-j}\eps)^{1/2}$), whereas the angular span $\Theta$ has measure $O(2^{-j}\eps)$ given that $\ab{x}\gtrsim 1$ and $V_j\subseteq B(x,2^{-j}\eps)$. We conclude that $\phi_4(x,\eps)\lesssim \eps^{1/2}$, and therefore \eqref{eq:ineqInV} is verified.
	This finishes the proof of the lemma.
\end{proof}

\begin{proposition}\label{prop:lip4circle}
	Given $\gamma>0$ and $\{h_j\}_{j=1}^4\subset \textup{Lip}(\sph{1})$,
        let $H_\gamma=\ab{\cdot}^{\gamma}(h_1\sigma\ast h_2\sigma\ast h_3\sigma\ast h_4\sigma)$.
	Then there exist $\tau>0$ and $C<\infty$ such that, for every $x,x'\in\R^2$,
	\begin{equation}\label{eq:weakHolder4S1}
	|H_\gamma(x)-H_\gamma(x')|\leq C|x-x'|^{\tau},
	\end{equation}
	where $C\leq C_0\prod_{j=1}^4\|h_j\|_{\textup{Lip}(\sph{1})}$, for some constant $C_0<\infty$ depending only on $\gamma$.
\end{proposition}

\noindent The proof of Proposition \ref{prop:lip4circle} will reveal that one can take any $\tau<\min\{\frac{1}{14},\frac{\gamma}{2(3\gamma+2)}\}$. 
To a large extent, the proof follows similar lines to those of Proposition  \ref{prop:lip3S2}, and so at times we shall be brief. The main difference is that now the extra singularity of $(\sigma\ast\sigma)(x)$ along the boundary circle $\ab{x}=2$ also needs to be accounted for.

\begin{proof}[Proof of Proposition \ref{prop:lip4circle}]
	Since the case $\gamma>1$ follows from that of $\gamma\in(0,1]$,  the latter condition will be assumed throughout the proof. 
	By homogeneity, we may assume $\|h_j\|_{\textup{Lip}}=1$, $1\leq j\leq 4$.
	Since $H_\gamma$ is compactly supported, it is enough to consider $x,x'\in\R^2$ satisfying $|x-x'|\ll 1$; we  further assume  $\ab{x}\leq \min\{4,\ab{x'}\}$. 
	With the notation introduced above (recall \eqref{eq:u12def}--\eqref{eq:defF}), we have that
	\begin{multline}
	\ab{H_\gamma(x)-H_\gamma(x')}=\abs{\ab{x}^\gamma(u_{12}F\ast u_{34}F)(x)-\ab{x'}^\gamma(u_{12}F\ast u_{34}F)(x')}\\
	\label{eq:secondTermDifH}
	\leq \ab{x}^\gamma\abs{(u_{12}F\ast u_{34}F)(x)-(u_{12}F\ast u_{34}F)(x')}+\ab{\ab{x}^\gamma-\ab{x'}^\gamma}\abs{(u_{12}F\ast u_{34}F)(x')}.
	\end{multline}
	The second summand in \eqref{eq:secondTermDifH} satisfies the upper bound
	\begin{align*}
	\ab{\ab{x}^\gamma-\ab{x'}^\gamma}\abs{(u_{12}F\ast u_{34}F)(x')}&\leq\ab{\ab{x}^\gamma-\ab{x'}^\gamma} \sigma^{\ast 4}(x')\lesssim_\gamma \ab{\ab{x}-\ab{x'}}^\gamma\sigma^{\ast 4}(x')\\
	&\leq \ab{x-x'}^{s}\ab{x'}^{\gamma-s}\sigma^{\ast 4}(x')\\
	&\lesssim_{\gamma,s} \ab{x-x'}^{s},
	\end{align*}
	for any $s\in (0,\gamma)$, where in the third inequality we used $\ab{x}\leq \ab{x'}$ to obtain $\ab{\ab{x}-\ab{x'}}\leq \ab{x'}$, and in the last inequality we invoked \eqref{eq:LinftySigma4}. The first summand in \eqref{eq:secondTermDifH} can be rewritten as the sum of two integrals,
	\begin{multline*}
	\ab{x}^\gamma\Bigl((u_{12}F\ast u_{34}F)(x)-(u_{12}F\ast u_{34}F)(x')\Bigr)\\
	=\ab{x}^\gamma\int_{B_2} u_{12}(y)F(y) F(x'-y)\left(u_{34}(x-y)-u_{34}(x'-y)\right)\,\d y\\
	+\ab{x}^\gamma\int_{B_2}u_{12}(y)F(y) \left(F(x-y)-F(x'-y)\right)u_{34}(x-y)\,\d y.
	\end{multline*}
	We denote the integrals on the right-hand side of the latter identity by $I$ and $II$, respectively, and proceed to estimate them separately.\\
	
	\noindent {\bf Estimating $I$.} The fist step is to restrict the domain of integration to the region where $x-y,x'-y\in B_2$, plus a
	$O(\ab{x-x'}^\alpha)$ remainder, for some $\alpha>0$ to be determined. 
	With this purpose in mind, decompose $B_2=U\cup U'\cup V\cup W$, where
	\begin{gather}
	\label{eq:decomposeB2}
	U:=\{y\in B_2: |x'-y|<2\leq |x-y|\},\, U':=\{y\in B_2: |x-y|<2\leq |x'-y|\},\\
	\nonumber
	V:=\{y\in B_2: |x-y|,|x'-y|<2 \}, \, W:=\{y\in B_2: 2\leq  |x'-y|,|x-y| \}.
	\end{gather}
	The integrand of $I$ vanishes on $U'\cup W$, and so  we are left to analyze the integrals over the regions $U$ and $V$. As in \eqref{eq:UcontainedAnnulus}, we have that
	\begin{equation}\label{eq:defAxeps}
	 U\subseteq\{y\in B_2\colon 2-\ab{x-x'}\leq \ab{x'-y}\leq 2\}:=A(x',\ab{x-x'}), 
	 \end{equation}
	and therefore
	\begin{align*}
	\ab{x}^\gamma\int_{U} \ab{u_{12}(y)u_{34}(x'-y)}F(y)F(x'-y)\,\d y
	&\leq \ab{x}^\gamma\int_{A(x',\ab{x-x'})} F(y)F(x'-y)\d y\\
	&\lesssim_{\gamma,s}\ab{x-x'}^{\min\{\frac{1}{6},\frac{\gamma}{2(\gamma+1)}-s\}},
	\end{align*}
	for every $s\in(0,\frac{\gamma}{2(\gamma+1)})$, where the latter inequality follows from estimate \eqref{eq:ineqInU}.  
	We now consider the integral over the set $V$. To begin with, note that Lemma \ref{lem:weakHolderd} implies the pointwise estimate
	\begin{equation}\label{eq:ChristShaod2}
	|u_{34}(x-y)-u_{34}(x'-y)|\lesssim |x-x'|^{1/2}+\left|\frac{x-y}{|x-y|}-\frac{x'-y}{|x'-y|}\right|,
	\end{equation}
	provided $x-y,x'-y\in B_2$.
	The contribution of the region 
	\begin{equation}\label{eq:defRd2}
	R:=\biggl\{y\in V: \left|\frac{x-y}{|x-y|}-\frac{x'-y}{|x'-y|}\right|\leq|x-x'|^{1/2}\biggr\}
	\end{equation}
	to the integral $I$ is easy to estimate.
	In view of \eqref{eq:ChristShaod2} and \eqref{eq:defRd2}, since $|x|\leq|x'|$,
	\begin{align*}
	\ab{x}^\gamma\Big|\int_R u_{12}(y)F(y) F(x'-y)&\left(u_{34}(x-y)-u_{34}(x'-y)\right)\,\d y\Big|\\
	&\lesssim \ab{x'}^\gamma\left(\int_R F(y)F(x'-y)\,\d y\right) |x-x'|^{1/2}\\
	&\leq \ab{x'}^\gamma\sigma_2^{\ast 4}(x')  |x-x'|^{1/2}
	\lesssim_\gamma |x-x'|^{1/2},
	\end{align*}
	where in the latter inequality we invoked \eqref{eq:LinftySigma4}. 
	If $y\in V\setminus R$, then $|x'-y|\leq 2\ab{x-x'}^{1/2}$ as in \eqref{eq:smallThanSqrt}.
	The contribution of the region $V\setminus R$ can then be estimated as follows:
	\begin{align*}
	\ab{x}^\gamma\Big|\int_{V\setminus R}u_{12}(y)F(y)&F(x'-y) \left(u_{34}(x-y)-u_{34}(x'-y)\right)\,\d y\Big|\\
	&\lesssim\ab{x}^\gamma\int_{V\setminus R}|u_{12}(y)|F(y)F(x'-y) \left|\frac{x-y}{|x-y|}-\frac{x'-y}{|x'-y|}\right|\,\d y\\
	&\leq 2 \ab{x}^\gamma \int_{V\cap B(x', 2|x-x'|^{1/2})} F(y)F(x'-y) \,\d y\\
	&\lesssim_{\gamma,s} |x-x'|^{\min\{\frac{1}{4},\frac{\gamma}{2}-s\}},
	\end{align*}
	for every $s\in(0,\frac{\gamma}2)$.	The latter inequality is a consequence of estimate \eqref{eq:ineqInV}.\\
	
	\noindent {\bf Estimating $II$.}
	The integral $II$ is bounded in absolute value by
	$$\ab{x}^\gamma\int_{B_2}F(y) \left|F(x-y)-F(x'-y)\right|\,\d y.$$
	Decompose $B_2=U\cup U'\cup V\cup W$ as in \eqref{eq:decomposeB2}, and note that the integrand of $II$ vanishes on $W$. The contribution of the region $U\cup U'$ can be handled with estimate \eqref{eq:ineqInU} as follows (recall \eqref{eq:defAxeps}):
	\begin{align*}
	\ab{x}^\gamma\int_{U\cup U'} F(y) \vert F(x-y)-F(x'-y)\vert \d y
	&\leq 
	2\ab{x}^\gamma\int_{A(x,\ab{x-x'})} F(y)F(x-y)\d y\\
	&\lesssim_\gamma \ab{x-x'}^{\min\{\frac{1}{6},\frac{\gamma}{2(\gamma+1)}-s\}},
	\end{align*}
	for every $s\in(0,\frac{\gamma}{2(\gamma+1)})$.	The estimate on the region $V$ is more delicate, and we split the analysis into two cases. Inside the ball $|x-y|\leq |x-x'|^{1/4}$, we also have that $\ab{x'-y}\leq \ab{x-x'}+\ab{x-y}\lesssim\ab{x-x'}^{1/4}$. In order to bound the corresponding piece of $II$, it suffices to consider the integral
	\[\varphi(x,x'):=\ab{x'}^\gamma\int_{V\cap B(x',|x-x'|^{1/4})} F(y)F(x'-y)\d y,\]
	which by \eqref{eq:ineqInV} 
	satisfies $\varphi(x,x')\lesssim_s\ab{x-x'}^{\frac{1}{4}\min\{\frac12,{\gamma}-s\}}$, for every $s\in (0,\gamma)$. We proceed with the analysis of the complementary region, i.e.\@ where $|x-y|>|x-x'|^{1/4}$. If $y\in B_2$, then
	\[ F(y)=\frac{4}{\ab{y}\sqrt{4-\ab{y}^2}}=\frac{\sqrt{4-\ab{y}^2}}{\ab{y}}+\frac{\ab{y}}{\sqrt{4-\ab{y}^2}}, \]
	and, as a consequence,
	\begin{align*}
	&\ab{F(x-y)-F(x'-y)} \\
	&\leq
	\Abs{\frac{\sqrt{4-\ab{x-y}^2}}{\ab{x-y}}-\frac{\sqrt{4-\ab{x'-y}^2}}{\ab{x'-y}}}
	+\Abs{\frac{\ab{x-y}}{\sqrt{4-\ab{x-y}^2}}-\frac{\ab{x'-y}}{\sqrt{4-\ab{x'-y}^2}}}\\
	&\leq \sqrt{4-\ab{x-y}^2}\Abs{\frac{1}{\ab{x-y}}-\frac{1}{\ab{x'-y}}}+\frac{1}{\ab{x'-y}}\ab{\sqrt{4-\ab{x-y}^2}-\sqrt{4-\ab{x'-y}^2}}\\
	&\quad+\ab{x-y}\Abs{\frac{1}{\sqrt{4-\ab{x-y}^2}}-\frac{1}{\sqrt{4-\ab{x'-y}^2}}}+\frac{1}{\sqrt{4-\ab{x'-y}^2}}\abs{\ab{x-y}-\ab{x'-y}}.
	\end{align*}
	Using the triangle inequality and recalling that $F(x'-y)=\frac{4}{\ab{x'-y}\sqrt{4-\ab{x'-y}^2}}$,
	\begin{align*}
	\ab{F(x-y)-F(x'-y)}&\lesssim \frac{\ab{x-x'}}{\ab{x-y}\ab{x'-y}}+\Abs{\frac{1}{\sqrt{4-\ab{x-y}^2}}-\frac{1}{\sqrt{4-\ab{x'-y}^2}}}\\
	&\quad+\frac{\ab{x-x'}^{1/2}}{\ab{x'-y}}+\frac{\ab{x-x'}}{\sqrt{4-\ab{x'-y}^2}}\\
	&\lesssim \frac{\ab{x-x'}}{\ab{x-y}\ab{x'-y}}+\Abs{\frac{1}{\sqrt{4-\ab{x-y}^2}}-\frac{1}{\sqrt{4-\ab{x'-y}^2}}}\\
	&\quad+\ab{x-x'}^{1/2}F(x'-y).
	\end{align*}
	If $|x-y|> |x-x'|^{1/4}$, then $\ab{x'-y}\geq \ab{x-y}-\ab{x-x'}\gtrsim \ab{x-x'}^{1/4}$. 
	Then for
	$y\in V\cap B(x,\ab{x-x'}^{1/4})^\complement$ we obtain	
	\begin{align*}
	\ab{F(x-y)-F(x'-y)}&\lesssim \ab{x-x'}^{1/2}+\Abs{\frac{1}{\sqrt{4-\ab{x-y}^2}}-\frac{1}{\sqrt{4-\ab{x'-y}^2}}}\\
	&\quad+\ab{x-x'}^{1/2}F(x'-y).
	\end{align*}	
	It follows that the contribution of this region to the integral $II$ is bounded by
	\begin{align*}
	&\ab{x}^\gamma\ab{x-x'}^{1/2}\int_VF(y)\d y+\ab{x}^\gamma\ab{x-x'}^{1/2}\int_{V}F(y)F(x'-y)\d y\\
	&\quad+\ab{x}^\gamma\int_{V\cap B(x,\ab{x-x'}^{1/4})^\complement}F(y)\Abs{\frac{1}{\sqrt{4-\ab{x-y}^2}}-\frac{1}{\sqrt{4-\ab{x'-y}^2}}}\d y\\
	&\lesssim \ab{x-x'}^{1/2}+\ab{x}^\gamma\int_{V\cap B(x,\ab{x-x'}^{1/4})^\complement}F(y)\Abs{\frac{1}{\sqrt{4-\ab{x-y}^2}}-\frac{1}{\sqrt{4-\ab{x'-y}^2}}}\d y,
	\end{align*}
	where we used that $\ab{x}\leq\ab{x'}$, $\ab{x}^\gamma\int_{V}F(y)F(x'-y)\d y\leq \ab{x'}^\gamma\sigma^{\ast 4}(x')\leq C_\gamma<\infty$, and $\int_VF(y)\d y\leq \sigma(\sph{1})^2$. 
	The last integral left to analyze is
	\begin{equation}\label{eq:differenceSqrtF}
	\ab{x}^\gamma\int_{V\cap B(x,\ab{x-x'}^{1/4})^\complement}F(y)\Abs{\frac{1}{\sqrt{4-\ab{x-y}^2}}-\frac{1}{\sqrt{4-\ab{x'-y}^2}}}\d y. 
	\end{equation}
	Given 
	$\delta\in(0,\frac{1}{2})$, 
	we further decompose the domain of integration, $V\cap B(x,\ab{x-x'}^{1/4})^\complement$, into the subregion where $4-\ab{x-y}^2\geq \ab{x-x'}^{\delta}$ and its complement.
	If $y\in V$ satisfies $4-\ab{x-y}^2\geq \ab{x-x'}^{\delta}$, then $4-\ab{x'-y}^2\gtrsim \ab{x-x'}^{\delta}$, and so
	\begin{align*}
	\Abs{\frac{1}{\sqrt{4-\ab{x-y}^2}}-\frac{1}{\sqrt{4-\ab{x'-y}^2}}}\lesssim\frac{\ab{x-x'}^{1/2}}{\sqrt{4-\ab{x-y}^2}\sqrt{4-\ab{x'-y}^2}}\lesssim \ab{x-x'}^{\frac{1}{2}-\delta}.
	\end{align*}
	Therefore, the contribution of this region to the integral \eqref{eq:differenceSqrtF} is bounded by
	\[ \ab{x}^\gamma\ab{x-x'}^{\frac{1}{2}-\delta}\int_{B_2} F(y)\d y\lesssim  \ab{x-x'}^{\frac{1}{2}-\delta}.\]
	Finally, if $4-\ab{x-y}^2< \ab{x-x'}^{\delta}$, then $2-\ab{x-y}\leq \frac{1}{2}\ab{x-x'}^{\delta}$, so that this region is contained in the annular domain
	\[ A(x,\eps):=\{y\in B_2\colon 2-\eps\leq \ab{x-y}\leq 2 \}, \]
	for $\eps=\frac{1}{2}\ab{x-x'}^{\delta}$. Since we also have $2-\ab{x'-y}\leq \ab{x-x'}^{\delta}$ if $\ab{x-x'}\ll 1$,  the region is also contained in $A(x',2\eps)$. The triangle inequality implies that the integral over the latter region is  bounded by (two times) the quantity
	\[ \tilde\vphi(x,x'):=\ab{x}^\gamma\int_{A(x,\ab{x-x'}^{\delta})}F(y)F(x-y)\d y. \]
	One last application of estimate
	\eqref{eq:ineqInU} reveals that $\tilde\vphi(x,x')\lesssim_{\gamma,s}\ab{x-x'}^{\delta\min\{\frac{1}{6},\frac{\gamma}{2(\gamma+1)}-s\}}$, for every $s\in(0,\frac{\gamma}{2(\gamma+1)})$.
	This concludes the proof of the proposition. 
\end{proof}

\begin{remark}\label{rem:higherConvosS1}
	More generally, all higher convolutions $G_n:=h_1\sigma\ast\dotsm\ast h_n\sigma$, $n\geq 5$, are H\"older continuous functions whenever $\{h_j\}_{j=1}^n\subset \textup{Lip}(\sph{1})$. Indeed, this can be verified for the fifth convolution $G_5=h_1\sigma\ast\dotsb\ast h_5\sigma$ by writing $G_5=(\ab{\cdot}^{-\gamma}H_{\gamma})\ast h_5\sigma$, for any $\gamma\in(0,1)$, studying the differences $\ab{G_5(x)-G_5(x')}$, and using  Proposition \ref{prop:lip4circle} together with the methods employed in its proof. Once it is known that $G_5\in\Lambda_\alpha(\R^2)$, for some $\alpha>0$, it is immediate that $G_n\in \Lambda_\alpha(\R^2)$, for every $n\geq 5$. This can be improved, e.g.\@ by noting that $G_{10}\in\Lambda_{2\alpha}(\R^2)$.
\end{remark}

\section{$\mathcal H^s$-bound for a restricted convolution operator}\label{sec:HsEEC}
Consider a function $H:\R^d\to\Co$ supported on the ball $B_R\subset\R^d$, for some $R>0$, satisfying, for some $\alpha\in(0,1)$ and $C<\infty$,
\begin{equation}\label{eq:weakHolder}
\ab{H(x)-H(x')}\leq C\ab{x-x'}^{\alpha}+C\left\vert\frac{x}{\ab{x}}-\frac{x'}{\ab{x'}}\right\vert,\;\text{ for every } x,x'\in B_R\setminus\{0\}.
\end{equation}
Then $H\in L^\infty(\R^d)$ and 
is continuous in $B_R\setminus\{0\}$. 
Given $\gamma\in[0,1]$, let $K_\gamma=\ab{\cdot}^{-\gamma}H$, and define the corresponding linear operator $\mathcal{K}_\gamma:C^0(\sph{d-1})\to L^2(\sph{d-1})$ via
\begin{equation}\label{eq:mathcalKgamma}
 (\mathcal K_\gamma f)(\omega)=\int_{\sph{d-1}}f(\nu)K_\gamma(\omega-\nu)\d\sigma_{d-1}(\nu).
 \end{equation}

\begin{lemma}\label{lem:boundedOpKernel}
	Let $d\geq 3$ and $\gamma\in [0,1]$, or $d=2$ and $\gamma\in[0,1)$.
	Let  $R>0$ 
	and $\mathcal{K}_\gamma$ be the linear operator defined in \eqref{eq:mathcalKgamma} above. Then there exists $\delta=\delta(d,\gamma,R)>0$, such that $\mathcal K_\gamma$ extends to a bounded operator from $L^2(\sph{d-1})$ to $\mathcal H^{\delta}(\sph{d-1})$.
\end{lemma}
\begin{proof}
	Let us start by considering the case $\gamma=1$ in dimensions $d\geq 3$. 
	Henceforth, $K_1, \mathcal K_1$ will be denoted  by $K, \mathcal K$, respectively. 
	Implicit constants may depend on $d, R$, as well as on the constant $C$ from \eqref{eq:weakHolder}. Consider the function $\delta(x)$ as in the proof of Lemma \ref{lem:HolderregularityConvo}. Introduce a radial partition of unity on $B_R$, $\{\phi_j\}_{j\geq 0}$, where $\phi_j=\delta(2^{j}R^{-1}\cdot)$ is supported where $2^{-j-1}R\leq \ab{x}\leq 2^{-j+1}R$, and $\sum_{j\geq 0}\phi_j(x)=1$, for every $x\in B_R\setminus\{0\}$. 
	Let $K_j=K\phi_j$, so that $\norma{K_j}_{L^\infty}\leq 2^{j+1}R^{-1}\norma{H}_{L^\infty}$, and $K_j$ is supported in the spherical shell \[A_j(R):=\{x\in \R^d\colon 2^{-j-1}R\leq\ab{x}\leq 2^{-j+1}R \}.\] 
	For $x,x'\in A_j(R)$, we have that 
	\begin{align}
	\ab{K_j(x)-K_j(x')}&=\ab{\ab{x}^{-1}H(x)\phi_j(x)-\ab{x'}^{-1}H(x')\phi_j(x')}\notag\\
	&\leq \ab{\ab{x}^{-1}-\ab{x'}^{-1}}\ab{H(x)}\phi_j(x)+\ab{x'}^{-1} \ab{H(x)-H(x')}\phi_j(x)\notag\\
	&\quad+\ab{x'}^{-1}\ab{H(x')}\ab{\phi_j(x)-\phi_j(x')}\notag\\
	&\lesssim \left\vert\frac{1}{\ab{x}}-\frac{1}{\ab{x'}}\right\vert+2^{j}\ab{x-x'}^\alpha+
	2^j\left\vert\frac{x}{\ab{x}}-\frac{x'}{\ab{x'}}\right\vert+2^{2j}\ab{x-x'}\notag\\
	&\lesssim 2^{2j}\ab{x-x'}+2^{j}\ab{x-x'}^\alpha+2^{j}\biggl(\frac{1}{\ab{x}}+\frac{1}{\ab{x'}}\biggr)\ab{x-x'}\notag\\
	&\lesssim 2^{2j}\ab{x-x'}^\alpha.\label{eq:HoEstK}
	\end{align}
	If $x,x'\in B_R$, $x\in \operatorname{supp}(K_j)$ but $x'\notin \operatorname{supp}(K_j)$, then $\ab{K_j(x)-K_j(x')}=\ab{K_j(x)}\lesssim 2^j$.
	
	To each $K_j$ there is a corresponding operator $\mathcal K_j$, so that $\mathcal K=\sum_{j\geq 0}\mathcal K_j$.
	The claimed boundedness of $\mathcal K$ is ensured if the  operator norms of the $\mathcal K_j$ are summable in $j$. In turn, the operator $\mathcal K_j$ is bounded on $L^2(\sph{d-1})$, with operator norm $\norma{\mathcal K_j}_{L^2\to L^2}=O(2^{-(d-2)j})$. Indeed, by Schur's test, we have that
	\begin{align*}
	\sup_{\nu\in \sph{d-1}}\int_{\sph{d-1}}\ab{K_j(\omega-\nu)}\d\sigma_{d-1}(\omega)
	&=\sup_{\omega\in \sph{d-1}}\int_{\sph{d-1}}\ab{K_j(\omega-\nu)}\d\sigma_{d-1}(\nu)\\
	&\lesssim 2^j\sup_{\omega\in \sph{d-1}}\int_{\sph{d-1}}\mathbbm{1}_{\{2^{-j-1}R\leq\ab{\omega-\nu}\leq  2^{-j+1}R\}}(\nu)\d\sigma_{d-1}(\nu)\\
	&\lesssim 2^{-(d-2)j}.
	\end{align*}
	Moreover, $\mathcal K_j$ maps $L^2(\sph{d-1})$ to $\Lambda_{\alpha}(\sph{d-1})$. 
	To see why this is the case, given $\omega,\omega'\in \sph{d-1}$, define the sets 
	\begin{align*}
	U(\omega,\omega')&:=\{\nu\in\sph{d-1}\colon \omega-\nu\in \textup{supp}(K_j),\omega'-\nu\notin\textup{supp}(K_j) \},\\
	U(\omega',\omega)&:=\{\nu\in\sph{d-1}\colon \omega'-\nu\in \textup{supp}(K_j),\omega-\nu\notin\textup{supp}(K_j) \},\\
	V&:=\{\nu\in\sph{d-1}\colon \omega-\nu,\omega'-\nu\in\textup{supp}(K_j) \}.
	\end{align*}
	Observe that
	\[ \sigma_{d-1}(V)\leq \int_{\sph{d-1}}\mathbbm{1}_{\{\ab{\omega-\nu}\leq 2^{-j+1}R\}}(\nu)\d\sigma_{d-1}(\nu)\lesssim 2^{-(d-1)j}. \]
	On the other hand, and similarly to \eqref{eq:UcontainedAnnulus}, the following inclusion holds:
	\begin{multline*}
	U(\omega,\omega')\subseteq \{\nu\in\sph{d-1}\colon 2^{-j-1}R-\ab{\omega-\omega'}\leq \ab{\omega'-\nu}\leq 2^{-j-1}R\}\\
	\cup\{\nu\in\sph{d-1}\colon 2^{-j+1}R-\ab{\omega-\omega'}\leq \ab{\omega-\nu}\leq 2^{-j+1}R\}.
	\end{multline*} 
	In particular, $\sigma_{d-1}(U(\omega,\omega'))\lesssim 2^{-(d-2)j}\ab{\omega-\omega'} $. 
	By the same argument, we also have that $\sigma_{d-1}(U(\omega',\omega))\lesssim 2^{-(d-2)j}\ab{\omega-\omega'}$.
	Then we may use \eqref{eq:HoEstK} and estimate 
	\begin{equation}\label{eq:HolderK}
	\begin{split}
	\ab{(\mathcal K_jf)(\omega)-(\mathcal K_jf)(\omega')}
	&\leq 
	\int_{U(\omega,\omega')\cup U(\omega',\omega)\cup V}\ab{K_j(\omega-\nu)-K_j(\omega'-\nu)}\ab{f(\nu)}\d\sigma_{d-1}(\nu)\\
	&\lesssim 2^{2j}\ab{\omega-\omega'}^\alpha\int_{V}\ab{f(\nu)}\d\sigma_{d-1}(\nu)
	+2^j\int_{U(\omega,\omega')}\ab{f(\nu)}\d\sigma_{d-1}(\nu)\\
	&\lesssim 2^{2j}\ab{\omega-\omega'}^\alpha 2^{-\frac{d-1}{2}j}\norma{f}_{L^2}+2^{-\frac{d-4}{2}j}\ab{\omega-\omega'}^{1/2}\norma{f}_{L^2}\\
	&\lesssim 2^{-\frac{d-5}{2}j}\ab{\omega-\omega'}^{\min\{\frac{1}{2},\alpha\}}\norma{f}_{L^2}.
	\end{split}
	\end{equation}
	No generality is lost in assuming that $\alpha\leq\frac12$.  
	Inequality \eqref{eq:HolderK} implies that $\mathcal K_j$ maps $L^2$ to $\mathcal H^\alpha$ boundedly, and moreover
	\begin{equation}\label{eq:Halpha} 
	\norma{\mathcal K_jf}_{\mathcal H^\alpha}\lesssim \norma{\mathcal K_jf}_{L^2}+2^{-\frac{d-5}{2}j}\norma{f}_{L^2}\lesssim 2^{-\frac{d-5}{2}j}\norma{f}_{L^2}. 
	\end{equation}
	From the definition of the $\mathcal H^s$-spaces, one directly checks the following interpolation bounds:
	\[ \norma{f}_{\mathcal H^{\te s+(1-\te)t}}\leq C\norma{f}_{\mathcal H^{s}}^{\te}\norma{f}_{\mathcal H^{t}}^{1-\te},\text{ for all }\te\in[0,1],\;0\leq s,t<1.  \]
	Using this to interpolate \eqref{eq:Halpha} with the $\mathcal H^0$-bound $\norma{\mathcal K_jf}_{L^2}\lesssim 2^{-(d-2)j} \|f\|_{L^2}$ reveals that, if $\delta>0$ is chosen sufficiently small depending on $d\in\{3,4,5\}$ and $\delta=\alpha$ if $d\geq 6$, then $\mathcal K_j$ maps $L^2$ to $\mathcal H^\delta$ boundedly, with operator norm $O(2^{-cj})$ for some $c>0$ which does not depend on $j$. This implies that $\norma{\mathcal K}_{L^2\to\mathcal H^\delta}<\infty$.
	
	We now discuss the case $\gamma\in [0,1)$. 
	If $d=2$, then the argument above works for the kernel $K_\gamma=\ab{\cdot}^{-\gamma}H$, for any $\gamma\in(0,1)$, since  the $L^2\to L^2$ operator norm of the corresponding $\mathcal K_{\gamma,j}$ is then $O(2^{-(1-\gamma)j})$. 
	If $d\geq 3$, then we  write $K_\gamma=\ab{\cdot}^{1-\gamma}K_1$, and see that the H\"older estimate for $K_1$ easily yields a corresponding statement for $K_\gamma$, for every $\gamma\in(0,1)$; in particular, the above argument also transfers. 
	The argument for $\mathcal K_0$ is similar but simpler (details omitted).
	The proof of the lemma is now complete.
\end{proof}

We are finally ready to establish a suitable replacement of Lemma \ref{lem:MLambdaBound} which handles the cases when $(d,m)\in \EEC$.
\begin{lemma}\label{cor:EEC}
	Given $(d,m)\in \EEC$, there exists $\alpha>0$ with the following property. If $\{h_j\}_{j=1}^m\subset \textup{Lip}(\sph{d-1})$  and $g\in L^2(\sph{d-1})$, then $\Mop(h_1,\dots,h_m,g)\in \mathcal H^\alpha$.
	Moreover, the following estimate holds:
	\[ \norma{\Mop(h_1,\dots,h_m,g)}_{\mathcal H^\alpha}\lesssim \prod_{j=1}^m\norma{h_j}_{\textup{Lip}(\sph{d-1})} \norma{g}_{L^2(\sph{d-1})}. \]
\end{lemma}
\begin{proof}
We consider three distinct cases:\\

	\noindent {\bf Case} $d\geq 3$, $m=2$. 
	From Corollary \ref{cor:multByX}, the function $G=\ab{\cdot}\,(h_1\sigma_{d-1}\ast h_2\sigma_{d-1})$ satisfies
	\begin{equation*}
	\ab{G(x)-G(x')}\lesssim \|h_1\|_{\textup{Lip}}\|h_2\|_{\textup{Lip}}\left(\ab{x-x'}^{1/2}+\left\vert\frac{x}{\ab{x}}-\frac{x'}{\ab{x'}}\right\vert\right).
	\end{equation*}
	The conclusion then follows from Lemma \ref{lem:boundedOpKernel} with $\gamma=1$.\\
	
	\noindent {\bf Case} $(d,m)=(3,3)$. 
	In view of Proposition \ref{prop:lip3S2},
	the function $h_1\sigma_2\ast h_2\sigma_2\ast h_3\sigma_2$ belongs to $\Lambda_{1/3}(\R^3)$. 
	 The conclusion then follows from Lemma \ref{lem:boundedOpKernel} with $\gamma=0$.\\
	
	\noindent {\bf Case} $(d,m)=(2,4)$. 
	In view of Proposition \ref{prop:lip4circle}, given $\gamma>0$, there exists $\tau\in (0,1)$, such that the function $\ab{\cdot}^\gamma (h_1\sigma_1\ast h_2\sigma_1\ast h_3\sigma_1\ast h_4\sigma_1)$ belongs to $\Lambda_\tau(\R^2)$.	
	The conclusion then follows from Lemma \ref{lem:boundedOpKernel} applied to any $\gamma\in(0,1)$.
\end{proof}

\section{Smoothness of critical points}\label{sec:Smoothness}

This section is devoted to the proof of Theorem \ref{thm:smoothnessTheorem}. 
Before starting the proof in earnest, we present two further results which will simplify the forthcoming analysis.

Given $(d,m)\in\mathfrak{U}$ and smooth functions $\{\vphi_j\}_{j=1}^m\subset C^\infty(\sph{d-1})$, 
we define the linear operator $\Lop=\Lop[\vphi_1,\dots,\vphi_m]\colon L^2(\sph{d-1})\to L^2(\sph{d-1})$ via
\[ \Lop[\vphi_1,\dots,\vphi_m](g)=\Mop(\vphi_1,\dots,\vphi_m,g). \]
Lemmata \ref{lem:MLambdaBound} and  \ref{cor:EEC} together imply  the bound $\norma{\Lop (g)}_{\mathcal H^\alpha}\leq C\norma{g}_{L^2}$, for some constant $C$ which depends on $d,m$, and on the functions $\{\vphi_j\}$. For our purposes, the precise dependence of the constant $C$ on $\{\vphi_j\}$ is not important; however, it is essential that $\Lop$ defines a bounded operator from $L^2(\sph{d-1})$ to $\mathcal H^\alpha$, for some exponent $\alpha>0$ which is independent of the functions $\{\vphi_j\}$.
Lemmata \ref{lem:MLambdaBound} and  \ref{cor:EEC} can be recast in terms of the operator $\Lop$, as follows.
\begin{corollary}\label{lem:HboundLop}
	Let $(d,m)\in\frak{U}$. There exists $\alpha>0$, such that $\Lop[\vphi_1,\dots,\vphi_m](g)\in\mathcal H^{\alpha}$, for any $\{\vphi_j\}_{j=1}^m\subset C^\infty(\sph{d-1})$ and $g\in L^2(\sph{d-1})$. Moreover, the following estimate holds:
	\begin{equation}\label{eq:boundLop}
	\norma{\Lop[\vphi_1,\dots,\vphi_m](g)}_{\mathcal H^{\alpha}}\leq C\norma{g}_{L^2(\sph{d-1})},
	\end{equation}
	where $C<\infty$ depends only on $d,m$, and on the functions $\{\vphi_j\}_{j=1}^m$. 
\end{corollary}

We shall find ourselves  in the need to expand the expressions $(\Theta-I)\Mop(f_1,\dots,f_{m+1})$ and $(\Theta-I)^2\Mop(f_1,\dots,f_{m+1})$, after a suitable decomposition $f_j=\vphi_{j,0}+\vphi_{j,1}$, $1\leq j\leq m+1$, has been performed. A model case for this situation is summarized in the following result.
The list of $\{\vphi_j\}$ with the $i$-th term removed will be denoted by $[\vphi_1,\dots,\mathring\vphi_i,\dots,\vphi_{m+1}]:=[\vphi_1,\dots,\vphi_{i-1},\vphi_{i+1},\dots,\vphi_{m+1}]$.

\begin{lemma}\label{lem:forBootstraping}
	Let $(d,m)\in\mathfrak{U}$, let $\eps\in(0,1)$, and let $\{f_j\}_{j=1}^{m+1}\subset L^2(\sph{d-1})$. 
	For each $j$, decompose $f_j=\vphi_{j,0}+\vphi_{j,1}$, with $\norma{\vphi_{j,0}}_{L^2(\sph{d-1})}<\eps\norma{f_j}_{L^2(\sph{d-1})}$ and $\vphi_{j,1}\in C^{\infty}(\sph{d-1})$. Then, for any $\Theta\in\mathrm{SO}(d)$, the following estimates hold:
	\begin{equation}\label{eq:multilinearBoundSplit}
	\begin{split}
	\norma{(\Theta-I)\Mop(f_1,\dotsc,f_{m+1}&)}_{L^2(\sph{d-1})}\\
	&\lesssim\sum_{i=1}^{m+1}\norma{(\Theta-I)\Lop[\vphi_{1,1},\dotsc,\mathring\vphi_{i,1},\dots,\vphi_{m+1,1}](\vphi_{i,0})}_{L^2(\sph{d-1})}\\
	&\quad+\sum_{i=1}^{m+1}\eps\norma{(\Theta-I)\vphi_{i,0}}_{L^2(\sph{d-1})}\prod_{j=1,j\neq i}^{m+1}\norma{f_j}_{L^2(\sph{d-1})}\\
	&\quad+\sum_{i=1}^{m+1}\norma{(\Theta-I)\vphi_{i,1}}_{L^2(\sph{d-1})}\prod_{j=1,j\neq i }^{m+1}\norma{f_j}_{L^2(\sph{d-1})},
	\end{split}
	\end{equation}
and
\begin{equation}\label{eq:2ndDiffBoundSplit}
\begin{split}
&\norma{(\Theta-I)^2\Mop(f_1,\dotsc,f_{m+1})}_{L^2(\sph{d-1})}\\
&\lesssim\sum_{i=1}^{m+1}\norma{(\Theta-I)\Lop[\vphi_{1,1},\dotsc,\mathring\vphi_{i,1},\dotsc,\vphi_{m+1,1}]((\Theta-I)\vphi_{i,0})}_{L^2(\sph{d-1})}\\
&\quad+\sum_{i=1}^{m+1}\eps\norma{(\Theta-I)^2\vphi_{i,0}}_{L^2(\sph{d-1})}\prod_{j=1,j\neq i}^{m+1}\norma{f_j}_{L^2(\sph{d-1})}\\
&\quad+\sum_{i=1}^{m+1}\norma{(\Theta-I)^2\vphi_{i,1}}_{L^2(\sph{d-1})}\prod_{j=1,j\neq i}^{m+1}\norma{f_j}_{L^2(\sph{d-1})}\\
&\quad+\sum_{\substack{1\leq i<j\leq m+1\\(\eps_i,\eps_j)\in\{0,1\}^2 }} \norma{(\Theta-I)\vphi_{i,\eps_i}}_{L^2(\sph{d-1})}\norma{(\Theta-I)\vphi_{j,\eps_j}}_{L^2(\sph{d-1})}\prod_{k=1,k\notin\{i,j\}}^{m+1}\norma{f_k}_{L^2(\sph{d-1})}.
\end{split}
\end{equation}

\end{lemma}
\noindent Estimates \eqref{eq:multilinearBoundSplit} and \eqref{eq:2ndDiffBoundSplit} exhibit a certain degree of asymmetry with respect to the role played by the functions $\vphi_{i,0}$ and $\vphi_{i,1}$. This is in order to ensure that the less smooth terms $\norma{(\Theta-I)\vphi_{i,0}}_{L^2(\sph{d-1})}$ and $\norma{(\Theta-I)^2\vphi_{i,0}}_{L^2(\sph{d-1})}$  always carry a mitigating factor of $\eps$. 

\begin{proof}[Proof of Lemma \ref{lem:forBootstraping}]
	Decompose each $f_j=\vphi_{j,0}+\vphi_{j,1}$ as in the statement of the lemma.
	Substituting this into $g:=\Mop(f_1,\dotsc,f_{m+1})$, and using the multilinearity of $\Mop$ together with the permutation symmetry \eqref{eq:symmetryM}, we have that
	\begin{align*}
	g&=\sum_{(\eps_1,\dotsc,\eps_{m+1})\in\{0,1\}^{m+1}} \Mop(\vphi_{1,\eps_1},\dotsc,\vphi_{m+1,\eps_{m+1}})\\
	&=\Mop(\vphi_{1,1},\dots,\vphi_{m+1,1})+\sum_{i=1}^{m+1}\Lop[\vphi_{1,1},\dotsc,\mathring\vphi_{i,1},\dotsc,\vphi_{m+1,1}](\vphi_{i,0})\\
	&\quad+\sum_{\substack{(\eps_1,\dotsc,\eps_{m+1})\in\{0,1\}^{m+1}\\\eps_1+\dots+\eps_{m+1}\leq m-1 }} \Mop(\vphi_{1,\eps_1},\dotsc,\vphi_{m+1,\eps_{m+1}}).
	\end{align*}
	The first, second and third summands in the latter expression correspond to those cases in which exactly none, one, or at least two of the $\eps_i$'s  are equal to 0, respectively. Therefore,
	\begin{align}
	(\Theta-I)g=&(\Theta-I)\Mop(\vphi_{1,1},\dots,\vphi_{m+1,1})+\sum_{i=1}^{m+1}(\Theta-I)\Lop[\vphi_{1,1},\dotsc,\mathring\vphi_{i,1},\dotsc,\vphi_{m+1,1}](\vphi_{i,0})\notag\\
	\label{eq:secondMterm}
	&+\sum_{\substack{(\eps_1,\dotsc,\eps_{m+1})\in\{0,1\}^{m+1}\\\eps_1+\dots+\eps_{m+1}\leq m-1 }} (\Theta-I)\Mop(\vphi_{1,\eps_1},\dotsc,\vphi_{m+1,\eps_{m+1}}).
	\end{align}
	In order to $L^2$-bound the terms coming from the latter sum in \eqref{eq:secondMterm}, we appeal to identity \eqref{eq:expansionM} for each summand, and obtain a further sum of terms of the form 
	$$\Mop(\vphi_{1,\eps_1},\dotsc,\vphi_{i-1,\eps_{i-1}},(\Theta-I)\vphi_{i,\eps_{i}},\Theta \vphi_{i+1,\eps_{i+1}},\dots,\Theta\vphi_{m+1,\eps_{m+1}} ).$$ 
	The corresponding $L^2$-norms can be bounded via the basic estimate \eqref{eq:basicL2}, yielding:
	\begin{align}
	\norma{\Mop(\vphi_{1,\eps_1},\dotsc,\vphi_{i-1,\eps_{i-1}},(\Theta-I)\vphi_{i,\eps_{i}},\Theta \vphi_{i+1,\eps_{i+1}},\ldots,\Theta\vphi_{m+1,\eps_{m+1}})}_{L^2(\sph{d-1})}\notag\\
	\lesssim\norma{(\Theta-I)\vphi_{i,\eps_{i}}}_{L^2(\sph{d-1})}\prod_{j=1,\,j\neq i}^{m+1}\norma{\vphi_{j,\eps_j}}_{L^2(\sph{d-1})}.\label{eq:firsttermsum}
	\end{align}
	As noted before, the condition $\eps_1+\dots+\eps_{m+1}\leq m-1$ implies the existence of at least two distinct indices $i'\neq j'$, such that $\eps_{i'}=\eps_{j'}=0$. In this way, \eqref{eq:firsttermsum} is bounded by
	$$\eps\norma{(\Theta-I)\vphi_{i,0}}_{L^2(\sph{d-1})}\prod_{j=1,j\neq i}^{m+1}\norma{f_j}_{L^2(\sph{d-1})}$$
	if $\eps_i=0$, or even better by 
	$$\eps^2\norma{(\Theta-I)\vphi_{i,1}}_{L^2(\sph{d-1})}\prod_{j=1,j\neq i}^{m+1}\norma{f_j}_{L^2(\sph{d-1})}$$
	if $\eps_i=1$. Finally, observe that
	\[ \norma{(\Theta-I)\Mop(\vphi_{1,1},\dots,\vphi_{m+1,1})}_{L^2(\sph{d-1})}\leq \sum_{i=1}^{m+1}\norma{(\Theta-I)\vphi_{i,1}}_{L^2(\sph{d-1})}\prod_{j=1,j\neq i}^{m+1}\norma{f_j}_{L^2(\sph{d-1})}.\]
	Adding up all the contributions, we obtain \eqref{eq:multilinearBoundSplit}.
	Considering now \eqref{eq:2ndDiffBoundSplit}, we start from \eqref{eq:secondMterm}, apply $\Theta-I$ to both sides, and obtain
	\begin{equation}\label{eq:secondDiffg}
	\begin{split}
	(\Theta-I)^2g=&(\Theta-I)^2\Mop(\vphi_{1,1},\dots,\vphi_{m+1,1})\\
	&\quad+\sum_{i=1}^{m+1}(\Theta-I)^2 \Mop(\vphi_{1,1},\dotsc,\mathring\vphi_{i,1},\dotsc,\vphi_{m+1,1},\vphi_{i,0})\\
	&\quad+\sum_{\substack{(\eps_1,\dotsc,\eps_{m+1})\in\{0,1\}^{m+1}\\\eps_1+\dots+\eps_{m+1}\leq m-1 }} (\Theta-I)^2\Mop(\vphi_{1,\eps_1},\dotsc,\vphi_{m+1,\eps_{m+1}}).
	\end{split}
	\end{equation}
	Using \eqref{eq:expansionM} twice together with the basic estimate \eqref{eq:basicL2}, the first term on the latter right-hand side can be bounded as follows:
	\begin{align*}
	&\norma{(\Theta-I)^2\Mop(\vphi_{1,1},\dots,\vphi_{m+1,1})}_{L^2(\sph{d-1})}\\
	&\lesssim\sum_{1\leq i<j\leq m+1} \norma{(\Theta-I)\vphi_{i,1}}_{L^2(\sph{d-1})}\norma{(\Theta-I)\vphi_{j,1}}_{L^2(\sph{d-1})}\prod_{k:k\notin\{i,j\}}\norma{f_k}_{L^2(\sph{d-1})}\\
	&\quad+\sum_{i=1}^{m+1}\norma{(\Theta-I)^2\vphi_{i,1}}_{L^2(\sph{d-1})}\prod_{j:j\neq i}\norma{f_j}_{L^2(\sph{d-1})}.
	\end{align*}
	An upper bound similar to the preceding one also applies to each term from the third sum in \eqref{eq:secondDiffg}, but this can be refined as follows: 
		\begin{align*}
	&\norma{(\Theta-I)^2\Mop(\vphi_{1,\eps_1},\dotsc,\vphi_{m+1,\eps_{m+1}})}_{L^2(\sph{d-1})}\\
	&\lesssim\sum_{1\leq i<j\leq m+1} \norma{(\Theta-I)\vphi_{i,\eps_i}}_{L^2(\sph{d-1})}\norma{(\Theta-I)\vphi_{j,\eps_j}}_{L^2(\sph{d-1})}\prod_{k:k\notin\{i,j\}}\norma{f_k}_{L^2(\sph{d-1})}\\
	&\quad+\sum_{i=1,\eps_i=1}^{m+1}\eps^2\norma{(\Theta-I)^2\vphi_{i,1}}_{L^2(\sph{d-1})}\prod_{j:j\neq i}\norma{f_j}_{L^2(\sph{d-1})}\\
	&\quad+\sum_{i=1,\eps_i=0}^{m+1}\eps\norma{(\Theta-I)^2\vphi_{i,0}}_{L^2(\sph{d-1})}\prod_{j:j\neq i}\norma{f_j}_{L^2(\sph{d-1})}.
	\end{align*}
	Lastly,  each of the terms coming from the second sum in \eqref{eq:secondDiffg} can be bounded as follows: 
	\begin{align*}
	&\norma{(\Theta-I)^2\Mop(\vphi_{1,1},\dotsc,\mathring\vphi_{i,1},\dotsc,\vphi_{m+1,1},\vphi_{i,0})}_{L^2(\sph{d-1})}\\
	&\lesssim\sum_{\substack{1\leq j<k\leq m+1\\ j\neq i, k\neq i }} \eps\norma{(\Theta-I)\vphi_{j,1}}_{L^2(\sph{d-1})}\norma{(\Theta-I)\vphi_{k,1}}_{L^2(\sph{d-1})}\prod_{\ell\notin\{j,k\}}\norma{f_\ell}_{L^2(\sph{d-1})}\\
	&\quad+\sum_{j=1,j\neq i}^{m+1}\norma{(\Theta-I)\vphi_{i,0}}_{L^2(\sph{d-1})}\norma{(\Theta-I)\vphi_{j,1}}_{L^2(\sph{d-1})}\prod_{k\notin\{i,j\}}\norma{f_k}_{L^2(\sph{d-1})}\\
	&\quad+\norma{(\Theta-I)\Lop[\vphi_{1,1},\dotsc,\mathring\vphi_{i,1},\dotsc,\vphi_{m+1,1}]((\Theta-I)\vphi_{i,0})}_{L^2(\sph{d-1})}.
	\end{align*}
	Adding up all the contributions yields \eqref{eq:2ndDiffBoundSplit}. This completes the proof of the lemma.
\end{proof}

\subsection{Proof of Theorem \ref{thm:smoothnessTheorem}}

We are now ready to start with the proof of Theorem \ref{thm:smoothnessTheorem} in earnest.
As a first step, we establish an initial regularity kick. 
Henceforth we  assume the parameter $\lambda$ in equation \eqref{eq:generalEL} to be nonzero,
in which case $\lambda$ can be absorbed into the function $a$; see the final remark in \S\ref{sec:finalrmk} below. We are thus interested in solutions of the equation
\begin{equation}\label{eq:newgeneralEL}
a\cdot \Mop(R^{k_1}(f),\ldots,R^{k_{m+1}}(f))=f,\quad\sigma_{d-1}\text{-a.e. on }\sph{d-1}.
\end{equation}

\begin{proposition}\label{prop:regularityGain}
	Let $(d,m)\in\mathfrak{U}$ and $(k_1,\ldots,k_{m+1})\in\{0,1\}^{m+1}$. Assume that  $a\in \Lambda_\kappa(\sph{d-1})$, for some  $\kappa\in(0,1)$. Then, given any complex-valued solution $f\in L^2(\mathbb S^{d-1})$ of equation \eqref{eq:newgeneralEL}, there exists $s>0$ such that $f\in \mathcal H^s$.
\end{proposition} 

\begin{proof}
	Let $f\in L^2(\mathbb{S}^{d-1})$ be a complex-valued solution of  \eqref{eq:newgeneralEL}, and let $\eps\in(0,1)$ be a small constant, to be chosen in the course of the argument. 
	We may decompose $f=g_\eps+\vphi_\eps$, where $\norma{g_\eps}_{L^2}<\eps\norma{f}_{L^2}$, and $\vphi_\eps\in C^\infty$. In this way, we have that $\norma{\vphi_\eps}_{L^2}\leq (1+\eps)\norma{f}_{L^2}\leq 2\norma{f}_{L^2}$; it is important that the latter bound is independent of $\eps$.
	By multilinearity of $\Mop$, no generality is lost in assuming that $f$ is $L^2$-normalized,  $\norma{f}_{L^2}=1$. 
In  \eqref{eq:newgeneralEL}, we further suppose that ${k_i}=0$, for every $1\leq i\leq m+1$. This assumption is made for notational purposes only, since the exact same argument applies in general.\footnote{Note that the operator $R$ is a linear isometry, and that $\norma{(\Theta-I)f}_{L^2}=\norma{(\Theta-I)f_\star}_{L^2}$, for every $\Theta\in \mathrm{SO}(d)$.
} Substituting $f=g_\eps+\vphi_\eps$ into the right-hand side of \eqref{eq:newgeneralEL}, we then see that the function $g_\eps$ satisfies the equation
	\[ g_\eps=a\cdot\Mop(f,\dots,f)-\vphi_\eps. \]
	Given $\Theta\in \textrm{SO}(d)$, apply $\Theta-I$ to both sides of the latter identity, yielding
	\begin{align*}
	(\Theta-I)g_\eps&=(\Theta-I)a\cdot \Theta\Mop(f,\dots,f)+a\cdot (\Theta-I)\Mop(f,\dots,f)-(\Theta-I)\vphi_\eps.
	\end{align*}
	Consequently,
	\begin{align*}
	\norma{(\Theta-I)g_\eps}_{L^2(\sph{d-1})}&\leq \norma{(\Theta-I)a}_{L^\infty(\sph{d-1})}\norma{ \Mop(f,\dots,f)}_{L^2(\sph{d-1})}+\norma{(\Theta-I)\vphi_\eps}_{L^2(\sph{d-1})}\\
	&\quad+\norma{a}_{L^\infty(\sph{d-1})}\norma{(\Theta-I)\Mop(f,\dots,f)}_{L^2(\sph{d-1})}.
	\end{align*}
	We estimate the third summand on the right-hand side of the latter inequality with the help of Lemma \ref{lem:forBootstraping}, yielding 
	\begin{align*}
	\norma{(\Theta-I)&g_\eps}_{L^2(\sph{d-1})}\lesssim \norma{(\Theta-I)a}_{L^\infty(\sph{d-1})}+(1+\norma{a}_{L^\infty(\sph{d-1})})\norma{(\Theta-I) \vphi_{\eps}}_{L^2(\sph{d-1})}\\
	&\quad+\norma{a}_{L^\infty(\sph{d-1})}\Bigl(\norma{(\Theta-I)\Lop[\vphi_{\eps},\dotsc,\vphi_{\eps}](g_{\eps})}_{L^2(\sph{d-1})}+\eps\norma{(\Theta-I) g_{\eps}}_{L^2(\sph{d-1})}\Bigr).
	\end{align*}
	We may now choose $\eps\in(0,1)$ small enough, depending on $d, m$, and on $\norma{a}_{L^\infty}$, so that the last term on the right-hand side can be absorbed into the left-hand side, yielding
	\begin{align*}
	\norma{(\Theta-I)g_\eps}_{L^2(\sph{d-1})}&\lesssim\norma{(\Theta-I)a}_{L^\infty(\sph{d-1})}
	+(1+\norma{a}_{L^\infty(\sph{d-1})})\norma{(\Theta-I)\vphi_\eps}_{L^2(\sph{d-1})}\\
	&\quad+\norma{a}_{L^\infty(\sph{d-1})} \norma{(\Theta-I)\Lop[\vphi_{\eps},\dotsc,\vphi_{\eps}](g_\eps)}_{L^2(\sph{d-1})}.
	\end{align*}
Choose $s\in(0,1)$ in such a way that $s\leq\kappa$ and $\Lop[\vphi_{\eps},\dotsc,\vphi_{\eps}]$ is bounded from $L^2$ to $\mathcal H^s$, as promised by Corollary \ref{lem:HboundLop}. Such an $s$ can be chosen independently of the function $\vphi_\eps$, and therefore does not depend on  $\eps$ either (but the implicit constant may depend on $\eps$, which we now take as fixed). Setting $\Theta=e^{tX_{i,j}}$, for some $1\leq i<j\leq d$, multiplying by $\ab{t}^{-s}$, and taking the supremum over $\ab{t}\in[0,1]$, yields
	\begin{multline}\label{eq:gisregular} 
	\sup_{\ab{t}\leq 1}\ab{t}^{-s}\norma{(e^{tX_{i,j}}-I)g_\eps}_{L^2(\sph{d-1})}\\
	\lesssim \norma{a}_{\Lambda_s(\sph{d-1})}+(1+\norma{a}_{L^\infty(\sph{d-1})})\norma{\vphi_\eps}_{\mathcal H^s}+C_\eps\norma{a}_{L^\infty(\sph{d-1})}\norma{g_\eps}_{L^2(\sph{d-1})}<\infty. 
	\end{multline}
	Here we are using that the $\Lambda_s$-norm can be controlled by the $\Lambda_\kappa$-norm since $s\leq\kappa$.
	Estimate \eqref{eq:gisregular} implies that $g_\eps\in \mathcal H^s$, and therefore $f\in \mathcal H^s$ as well.
	The proof of the proposition is now complete.
\end{proof}

\begin{remark}\label{eq:freeGain}
	If $(d,m)\in\mathfrak{U}\setminus \EEC$, then there is an automatic gain in the initial regularity of any complex-valued $f\in L^2(\sph{d-1})$ solution of equation \eqref{eq:newgeneralEL}. Indeed, we claim that in that case $f$ necessarily coincides with a continuous function on $\mathbb S^{d-1}$. To see why this must be so, start by considering the case $d, m\geq 3$. Writing $m+1=(m-1)+2$, where $m-1\geq 2$, we see that the convolution product on the left-hand side of \eqref{eq:generalEL} can be written as
	\[ (R^{k_{1}}(f)\sigma_{d-1}\ast \dotsm\ast  R^{k_{m-1}}(f)\sigma_{d-1})\ast(R^{k_m}(f)\sigma_{d-1}\ast R^{k_{m+1}}(f)\sigma_{d-1}). \]
	Since each of the two functions in the preceding convolution belongs to $L^2(\R^d)$, their convolution defines a continuous function of bounded support on $\R^d$. It follows that its restriction to the unit sphere also defines a continuous function on $\sph{d-1}$, as claimed. An analogous argument works for the case $d=2$ and $m\geq 5$.
\end{remark}

The second main step  is a bootstrapping procedure which will complete the proof of Theorem \ref{thm:smoothnessTheorem}. 
Indeed, in light of Remark \ref{rem:SobolevHolderIntegerSphere},
Propositions \ref{prop:regularityGain} and \ref{prop:bootstrapingprop} together imply that a solution $f$ of equation \eqref{eq:newgeneralEL} (and therefore of equation \eqref{eq:generalEL} if $\lambda\neq 0$) satisfies $f\in H^r$, for every $r\geq 0$.
From Sobolev embedding, see e.g.\@ \cite[Theorem 2.7]{He99}, it then follows that $f\in C^\infty(\sph{d-1})$.

\begin{proposition}\label{prop:bootstrapingprop}
	Let $(d,m)\in\mathfrak{U}$.
	Let $(k_1,\ldots,k_{m+1})\in\{0,1\}^{m+1}$, $\lambda\in\Co\setminus\{0\}$, and $a\in C^\infty(\sph{d-1})$. Then there exists $\alpha>0$ with the following property. Let $f$ be a solution of equation \eqref{eq:newgeneralEL} satisfying $f\in\mathcal{H}^s$, for some $s>0$. Then $f\in\mathcal{H}^t$, for every $t\in[0,s+\min\{s-\lfloor s\rfloor,\alpha\}]\setminus \Z$.
\end{proposition}

\begin{proof}
	We make a few initial simplifications.
	Firstly, we consider the special case $a\equiv 1$ only, since the general case $a\in C^\infty(\sph{d-1})$ brings no additional complications, as shown by the proof of Proposition \ref{prop:regularityGain}. Secondly, we further  assume that $k_i=0$, for every $1\leq i\leq m+1$; this considerably simplifies the forthcoming notation, but changes nothing fundamental in the analysis. Thirdly, we start by supposing that $s\in(0,1)$. The case $s\geq 1$ will be dealt with at a later stage in the proof.
	
	Assume $\norma{f}_{L^2}=1$, and let $\eps\in(0,1)$, to be chosen in the course of the argument. Decompose $f=g_\eps+\vphi_\eps$, with $\vphi_\eps\in C^\infty(\sph{d-1})$ and $\norma{g_\eps}_{L^2}<\eps$. In particular, $\norma{\vphi_\eps}_{L^2}\leq 2$. Since $f\in\mathcal H^s$, it follows that $g_\eps\in\mathcal H^s$ as well. The equation satisfied by $g_\eps$ is
	\[ g_\eps=\Mop(f,\dots,f)-\vphi_\eps. \]
	Given $\Theta\in \textrm{SO}(d)$, we have that 
	\[ (\Theta-I)^2g_\eps=(\Theta-I)^2\Mop(f,\dots,f)-(\Theta-I)^2\vphi_\eps,  \]
	and therefore
	\[ \norma{(\Theta-I)^2g_\eps}_{L^2(\sph{d-1})}\leq \norma{(\Theta-I)^2\Mop(f,\dots,f)}_{L^2(\sph{d-1})}+\norma{(\Theta-I)^2\vphi_\eps}_{L^2(\sph{d-1})}. \]
	Using Lemma \ref{lem:forBootstraping} to estimate the first term on the right-hand side of the preceding inequality, we obtain
	\begin{align*}
	\norma{(\Theta-I)^2g_\eps&}_{L^2(\sph{d-1})}\lesssim \norma{(\Theta-I)\vphi_\eps}_{L^2(\sph{d-1})}^2+
	\norma{(\Theta-I)^2\vphi_\eps}_{L^2(\sph{d-1})}\\
	&\quad+\norma{(\Theta-I)\vphi_\eps}_{L^2(\sph{d-1})}\norma{(\Theta-I)g_\eps}_{L^2(\sph{d-1})}
	+\norma{(\Theta-I)g_\eps}_{L^2(\sph{d-1})}^2\\
	&\quad+\norma{(\Theta-I)\Lop[\vphi_{\eps},\dotsc,\vphi_{\eps}]((\Theta-I) g_{\eps})}_{L^2(\sph{d-1})}+\eps\norma{(\Theta-I)^2g_\eps}_{L^2(\sph{d-1})}^2.
	\end{align*}
	Now choose $\eps\in(0,1)$ small enough, depending on $d,m$, in such a way that the last term on the latter left-hand side can be absorbed into the right-hand side. With such a choice of $\eps$, the following inequality holds:
	\begin{equation}\label{eq:ineqSecondDiff}
	\begin{split}
	\norma{(\Theta-I)^2g_\eps}_{L^2(\sph{d-1})}&\lesssim \norma{(\Theta-I)\vphi_\eps}_{L^2(\sph{d-1})}^2+
	\norma{(\Theta-I)^2\vphi_\eps}_{L^2(\sph{d-1})}\\
	&\quad+\norma{(\Theta-I)\vphi_\eps}_{L^2(\sph{d-1})}\norma{(\Theta-I)g_\eps}_{L^2(\sph{d-1})}
	+\norma{(\Theta-I)g_\eps}_{L^2(\sph{d-1})}^2\\
	&\quad+\norma{(\Theta-I)\Lop[\vphi_{\eps},\dotsc,\vphi_{\eps}]((\Theta-I) g_{\eps})}_{L^2(\sph{d-1})}.
	\end{split}
	\end{equation}
	Now that $\eps$ has been fixed, Corollary \ref{lem:HboundLop} implies that the operator $\Lop[\vphi_{\eps},\dotsc,\vphi_{\eps}]$ is bounded from $L^2$ to $\mathcal H^\alpha$, for some
	$\alpha\in(0,1)$ independent of $\eps$.
		
	Set $\delta=\min\{s,\alpha\}$, where $\alpha$  is as in the previous paragraph.
	In particular, $\Lop[\vphi_{\eps},\dotsc,\vphi_{\eps}]$ is bounded from $L^2$ to $\mathcal H^\delta$, with operator norm that may depend on $\eps$. 
	Henceforth we consider $\Theta=\Theta(t)=e^{tX_{k,\ell}}$, $1\leq k<\ell\leq d$, and $\ab{t}\leq 1$. 
	The following estimate holds:
	\begin{align*}
	\norma{(\Theta-I)\Lop[\vphi_{\eps},\dotsc,\vphi_{\eps}&]((\Theta-I)g_\eps)}_{L^2(\sph{d-1})}\\
	&\leq \ab{t}^{\delta}\sup_{\ab{\tau}\leq 1}\ab{\tau}^{-\delta}\norma{(\Theta(\tau)-I)
	\Lop[\vphi_{\eps},\dotsc,\vphi_{\eps}]((\Theta(t)-I)g_\eps)}_{L^2(\sph{d-1})}\\
	&\leq \ab{t}^{\delta}\norma{\Lop[\vphi_{\eps},\dotsc,\vphi_{\eps}]((\Theta(t)-I)g_\eps)}_{\mathcal H^\delta}\\
	&\leq C_\eps\ab{t}^{\delta}\norma{(\Theta-I)g_\eps}_{L^2(\sph{d-1})}\\
	&\leq C_\eps\ab{t}^{\delta+s}\norma{g_\eps}_{\mathcal H^s}.
	\end{align*}
	Multiplying \eqref{eq:ineqSecondDiff} by $\ab{t}^{-(s+\delta)}$ yields
	\begin{equation*}
	\begin{split}
	&\ab{t}^{-(s+\delta)}\norma{(\Theta-I)^2g_\eps}_{L^2(\sph{d-1})}\\
	&\lesssim \ab{t}^{-\delta}\norma{(\Theta-I)\vphi_\eps}_{L^2(\sph{d-1})}\ab{t}^{-s}\norma{(\Theta-I)\vphi_\eps}_{L^2(\sph{d-1})}+
	\ab{t}^{-(s+\delta)}\norma{(\Theta-I)^2\vphi_\eps}_{L^2(\sph{d-1})}\\
	&\quad+\ab{t}^{-\delta}\norma{(\Theta-I)\vphi_\eps}_{L^2(\sph{d-1})}\ab{t}^{-s}\norma{(\Theta-I)g_\eps}_{L^2(\sph{d-1})}\\
	&\quad+\ab{t}^{-\delta}\norma{(\Theta-I)g_\eps}_{L^2(\sph{d-1})}\ab{t}^{-s}\norma{(\Theta-I)g_\eps}_{L^2(\sph{d-1})}
	+C_\eps\norma{g_\eps}_{\mathcal H^s}.
	\end{split}
	\end{equation*}
	Now take the supremum over $\ab{t}\leq 1$, and use the facts that $\vphi_\eps\in\mathcal H^r$ for all $0\leq r\notin\Z$, and $g_\eps\in\mathcal H^s\cap \mathcal H^\delta$ (recall that $\delta\leq s$). Invoking the characterization of the $\mathcal H^{s+\delta}$-norm by means of second differences as detailed in \S \ref{subseq:secondDifferences} below, which applies since $s+\delta\in(0,2)$, we obtain that
	\begin{equation*}
	\begin{split}
	\sup_{\ab{t}\leq 1}\ab{t}^{-(s+\delta)}\norma{(\Theta-I)^2g_\eps}_{L^2(\sph{d-1})}&\lesssim \norma{\vphi_\eps}_{\mathcal H^\delta}\norma{\vphi_\eps}_{\mathcal H^s}+
	\norma{\vphi_\eps}_{\mathcal H^{s+\delta}}+\norma{\vphi_\eps}_{\mathcal H^\delta}\norma{g_\eps}_{\mathcal H^s}\\
	&\quad+\norma{g_\eps}_{\mathcal H^\delta}\norma{g_\eps}_{\mathcal H^s}
	+C_\eps\norma{g}_{\mathcal H^s}<\infty.
	\end{split}
	\end{equation*}
In this way, again via second differences, we see that $g_\eps\in\mathcal H^{s+\delta}$, and therefore $f\in \mathcal H^{s+\delta}$ as well.\footnote{If $s+\delta=1$, then $\mathcal{H}^{s+\delta}$ is not defined, but, by using any $\delta'<\delta$ in the reasoning above, the conclusion is that $g_\eps\in\mathcal H^{t}$, for every $t<1$, and therefore $f \in\mathcal H^{t}$, for every $t<1$.}
This concludes the proof of the proposition in the special case when $s\in (0,1)$.
	
Repeated applications of the previous step reveal that if $f\in \mathcal H^s$ for some $s\in (0,1)$, then $f\in \mathcal H^{1+\gamma}$ for some $\gamma\in (0,1)$. We complete the proof of the proposition by induction. In order to treat  exponents $s=k+\gamma$, with $k\in\N$ and $\gamma\in (0,1)$, we use the product rule \eqref{eq:derivativeMop}, and differentiate $k$ times  identity \eqref{eq:newgeneralEL} with respect to $X\in\{X_{i,j}\colon 1\leq i<j\leq d\}$, thus  obtaining an equation for $X^kf\in\mathcal H^\gamma$. Decomposing $X^kf=g_\eps+\vphi_\eps$, with $\vphi_\eps\in C^\infty(\sph{d-1})$ and $\norma{g_\eps}_{L^2}\leq \eps\norma{X^kf}_{L^2}$, we can use the same method as before to show that $g_\eps\in \mathcal H^{t}$, for any $t\in[s,s+\min\{\gamma,\alpha\}]\setminus\Z$. In a similar way, we may analyze the mixed derivatives $Yf:=Y_1\dotsc Y_k f$, where $Y_\ell\in\{X_{i,j}: 1\leq i<j\leq d\}$, $1\leq \ell\leq k$. 
In what follows, we provide the details.
	
For simplicity, we only consider powers of the same vector field $X$, but note that the exact same method would apply to a more general vector field $Y$ as in the previous paragraph. 
The equation satisfied by $X^k f$ is of the  form
	\begin{equation}\label{eq:EqXk}
	 X^kf=\sum_{\substack{\vec{k}:=(k_1,\dotsc,k_{m+1})\in\N_0^{m+1}\\
			k_1+\dotsb+k_{m+1}=k}}c_{\vec{k}}\,\Mop(X^{k_1}f,\dotsc,X^{k_{m+1}}f),
			\end{equation}
	for some constants $c_{\vec{k}}>0$. Note that  $X^{k_j}f\in\mathcal H^{1+\gamma}$ if $k_j<k$. Thus we are led to splitting the sum in \eqref{eq:EqXk} into two parts, one of them containing precisely those summands which carry the term $X^kf$. There are $m+1$ of them, and so
	\begin{equation}\label{eq:equationXkf}
	X^kf=\sum_{\substack{\vec{k}\in K}}c_{\vec{k}}\,\Mop(X^{k_1}f,\dotsc,X^{k_{m+1}}f)
	+(m+1)\Mop(f,\dotsc,f,X^{k}f),
	\end{equation}
	where $(k_1,\dotsc,k_{m+1})\in K$ if and only if $k_{j}<k$, for every $1\leq j\leq m+1$, and $k_1+\dots+k_{m+1}=k$. 
	The first term on the right-hand side of \eqref{eq:equationXkf} can be easily bounded in $\mathcal H^{1+\gamma}$ with \eqref{eq:HsboundforM}, yielding
	\[ \sum_{\substack{\vec{k}\in K}}c_{\vec{k}}\norma{\Mop(X^{k_1}f,\dotsc,X^{k_{m+1}}f)}_{\mathcal H^{1+\gamma}}\lesssim \norma{f}_{\mathcal H^s}^{m+1}. \]
	To handle the second term, let $\eps\in(0,1)$, and decompose $f=\vphi_0+\vphi_1$, $X^kf=\psi_{0}+\psi_{1}$, with $\vphi_{1},\psi_{1}\in C^\infty(\sph{d-1})$ and $\norma{\vphi_0}_{L^2}<\eps\norma{f}_{L^2},\norma{\psi_{0}}_{L^2}<\eps\norma{X^kf}_{L^2}$. Since $f\in\mathcal H^s$, we have that $\vphi_0\in\mathcal H^s$ and $\psi_{0}\in\mathcal H^\gamma$. Now take $\delta\in(0,1)$ satisfying $\delta\leq \min\{\gamma,\alpha\}$;
	recall that $\gamma=s-\lfloor s\rfloor$, and that $\alpha$ was chosen immediately following \eqref{eq:ineqSecondDiff}. The equation satisfied by $\psi_{0}$ may be derived from \eqref{eq:equationXkf}. 
	Applying $(\Theta-I)^2$ to both sides of that equation, and invoking Lemma \ref{lem:forBootstraping}, we find that, if $\eps>0$ is small enough, then
	\begin{align*}
	\norma{(\Theta-I)^2&\psi_{0}}_{L^2(\sph{d-1})}\lesssim \sum_{\substack{\vec{k}\in K}}c_{\vec{k}}\norma{(\Theta-I)^2\Mop(X^{k_1}f,\dotsc,X^{k_{m+1}}f)}_{L^2(\sph{d-1})}\\
	&\quad+(\norma{(\Theta-I)^2\vphi_0}_{L^2(\sph{d-1})}+\norma{(\Theta-I)^2\vphi_1}_{L^2(\sph{d-1})})\norma{X^kf}_{L^2(\sph{d-1})}\\
	&\quad+(\norma{(\Theta-I)\vphi_0}_{L^2(\sph{d-1})}+\norma{(\Theta-I)\vphi_1}_{L^2(\sph{d-1})})^2\norma{X^kf}_{L^2(\sph{d-1})}\\
	&\quad+(\norma{(\Theta-I)\vphi_0}_{L^2(\sph{d-1})}+\norma{(\Theta-I)\vphi_1}_{L^2(\sph{d-1})})\norma{(\Theta-I)\psi_{0}}_{L^2(\sph{d-1})}\\
	&\quad+(\norma{(\Theta-I)\vphi_0}_{L^2(\sph{d-1})}+\norma{(\Theta-I)\vphi_1}_{L^2(\sph{d-1})})\norma{(\Theta-I)\psi_{1}}_{L^2(\sph{d-1})}\\
	&\quad+C_\eps\ab{t}^\delta \norma{(\Theta-I)\psi_{0}}_{L^2(\sph{d-1})}+\norma{(\Theta-I)^2\psi_{1}}_{L^2(\sph{d-1})}.
	\end{align*}
	Consequently, by means of second differences, we obtain
	\[ \sup_{0<\ab{t}\leq 1}\ab{t}^{-(\delta+\gamma)}\norma{(\Theta-I)^2\psi_{0}}_{L^2(\sph{d-1})}<\infty, \]
	and as a result $\psi_{0}\in \mathcal H^{\gamma+\delta}$. It follows that $X^kf\in\mathcal H^{\gamma+\delta}$ and, since  $X\in\{X_{i,j}: 1\leq i<j\leq d\}$ was arbitrary,\footnote{Again, if $s+\delta\in\Z$, then the conclusion is that $f\in \mathcal H^t$, for every $t\in[0,s+\delta]\setminus\Z$.} $f\in\mathcal H^{s+\delta}$.
	The proof of the proposition is now complete.
\end{proof}

\subsection{Second differences}\label{subseq:secondDifferences}
Given $s\in(0,2)$, we define the space $\mathscr{H}^s=\mathscr{H}^s(\sph{d-1})$ of all functions $f\in L^2(\mathbb{S}^{d-1})$, for which the norm 
\begin{equation}\label{eq:normHHs} 
\norma{f}_{\mathscr{H}^s}=\norma{f}_{L^2(\sph{d-1})}+\sum_{1\leq i<j\leq d}\sup_{\ab{t}\leq 1}\ab{t}^{-s}\norma{(e^{tX_{i,j}}-I)^2f}_{L^2(\mathbb{S}^{d-1})} 
\end{equation}
is finite.
We see that
\[(e^{tX_{i,j}}-I)^2f=f\circ e^{2tX_{i,j}}-2f\circ e^{tX_{i,j}}+f\]
resembles a second difference of $f$. From the definition, it is immediate that  $\norma{f}_{\mathscr{H}^s}\leq 2\norma{f}_{\mathcal{H}^s}$ provided $s\in(0,1)$, and so  $\mathcal H^s\subseteq\mathscr{H}^s$. The reverse inclusion also holds.
Moreover, if $s\in (0,2)\setminus\{1\}$, then $\mathcal H^s=\mathscr{H}^s$, and the two norms given by \eqref{eq:mathcalHsNormBig} and \eqref{eq:normHHs} are equivalent.
These assertions have all appeared in the literature;
in what follows, we provide precise references.

Let us discuss the Euclidean case first. 
Given $s\in(0,1)$, we defined the H\"older space $\Lambda_s(\R^d)$ to contain precisely those functions $f:\R^d\to\Co$ for which the norm 
\[ \norma{f}_{L^\infty(\R^d)}+\sup_{\ab{t}>0}\ab{t}^{-s}\norma{f(x+t)-f(x)}_{L_x^\infty(\R^d)}\]
is finite, whereas for $s=k+\delta$, $1\leq k\in\N$, $\delta\in (0,1)$, we have that $f\in \Lambda_s(\R^d)$ if $f\in C^k(\R^d)$ and $\partial^{\alpha}f\in \Lambda_\delta(\R^d)$, for all multi-indices $\alpha\in\N_0^d$ with $\ab{\alpha}=k$.
Given $s\in(0,2)$, consider the norm (defined in terms of second differences),
\[ \norma{f}_{L^\infty(\R^d)}+\sup_{\ab{t}>0}\ab{t}^{-s}\norma{f(x+2t)-2f(x+t)+f(x)}_{L_x^\infty(\R^d)}, \]
and the corresponding space of functions for which the latter norm is finite. 
These two spaces coincide if $s\in(0,2)\setminus\{1\}$, as dictated by the classical equivalence between  H\"older and Zygmund spaces, the latter being defined through higher differences; precise references  include \cite[Ch.\@ V, Prop.\@ 8]{St70} and \cite[Ch.\@ 2, \S 2.6]{Tr92}. More generally, one may consider an $L^p$-norm in $x$, $1\leq p\leq \infty$, and possibly an additional $L^q$-norm in $t$, $1\leq q\leq\infty$;
see \cite[Ch.\@ V, Prop.\@ 8']{St70} and \cite[Ch.\@ 2, \S 2.6]{Tr92}.

For the case of  the unit sphere $\sph{d-1}$, the equivalence between the $\mathcal H^s$- and the $\mathscr H^s$-norms, and therefore the equality of the two corresponding spaces, can be found in \cite{Gr74,Gr77}. These works rely on harmonic extensions, in a similar spirit to the aforementioned chapter in \cite{St70}. Of particular relevance are Propositions 4.1 and 4.3 in \cite{Gr74}, and Proposition 1.8 in \cite{Gr77}. In the former article \cite{Gr74}, the function space $\Lambda(\alpha;p,q)$ is defined for $\alpha>0$, $1\leq p,q\leq\infty$, and shown to be equivalent to a variant thereof using first- and second-order differences; the special case $(p,q)=(2,\infty)$ and $\alpha=s\in(0,1)$ of this equivalence is used to establish that the spaces $\mathcal H^s$ and $\mathscr H^s$ coincide whenever $s\in(0,1)$. In the latter article \cite{Gr77}, spaces of index $\alpha=k+\gamma$, $k\in\N$, are related to those of index $\gamma$ in a precise way; in turn, this is used to establish the equivalence between the spaces $\mathcal H^s$ and $\mathscr H^s$ whenever $s\in(1,2)$. It should be pointed out that the norms in terms of first and second differences considered in \cite{Gr74} are slightly different from the ones which we are using to define $\mathcal H^s$ and $\mathscr H^s$. However, the norms are seen to be equivalent; see \cite[Cor.\@ 3.11]{DX10}. See also \cite{Co83} and \cite[Theorems 3.1 and 3.3]{LR94}

We proceed to describe an alternative approach to the equivalence discussed in the previous paragraph which perhaps requires less effort from the unfamiliar reader. 

Firstly, the equivalence between $\mathcal{H}^s$ and $\mathscr{H}^s$ when $s\in(0,1)$ follows directly from the combinatorial proof of \cite[Lemma 1.1]{Her68}, stated in \cite{Her68} for the case of $\R^d$. For the convenience of the reader, we provide a brief sketch of the argument. As mentioned already,  the estimate $\norma{f}_{\mathscr{H}^s}\leq 2\norma{f}_{\mathcal{H}^s}$ follows easily from the definitions. For the reverse inequality, consider the following identity, which is valid for every $t\in\R$, $X\in\{X_{i,j}:1\leq i<j\leq d \}$, and $m\in\N$:
\[ 2^m(e^{tX}-I)=(e^{2^mtX}-I)-\sum_{i=0}^{m-1}2^{m-1-i}(e^{2^itX}-I)^2. \]
Applying this operator to a function $f\in\mathscr{H}^s$, 
taking the $L^2(\sph{d-1})$-norm on both sides, and invoking the triangle inequality, 
 yields 
\begin{align*}
2^m\norma{(e^{tX}-I)f}_{L^2(\sph{d-1})}&\leq \norma{(e^{2^{m}tX}-I)f}_{L^2(\sph{d-1})}
+\sum_{i=0}^{m-1}2^{m-1-i}\norma{(e^{2^itX}-I)^2f}_{L^2(\sph{d-1})}.
\end{align*}
Dividing by $2^m$, using that $\norma{(e^{2^{m}tX}-I)f}_{L^2(\sph{d-1})}\leq 2\norma{f}_{L^2(\sph{d-1})}\leq 2\norma{f}_{\mathscr{H}^2}$,
 and letting $m\to\infty$, we then obtain
\begin{equation*}
\norma{(e^{tX}-I)f}_{L^2(\sph{d-1})}\leq \frac{1}{2}\sum_{i=0}^\infty 2^{-i} \norma{(e^{2^itX}-I)^2f}_{L^2(\sph{d-1})}. 
\end{equation*}
Multiplying by $\ab{t}^{-s}$ and taking the supremum over $t\in\R$ shows that, when $s<1$, the following holds:
\begin{align}
\sup_{t\in\R}\ab{t}^{-s}\norma{(e^{tX}-I)f}_{L^2(\sph{d-1})}&\leq \frac{1}{2}\biggl(\sum_{i=0}^\infty 2^{-i(1-s)}\biggr) \sup_{t\in\R}\ab{t}^{-s}\norma{(e^{tX}-I)^2f}_{L^2(\sph{d-1})}\nonumber \\
\label{eq:ineqSingleDouble}
&=\frac{1}{2(1-2^{-(1-s)})}\sup_{t\in\R}\ab{t}^{-s}\norma{(e^{tX}-I)^2f}_{L^2(\sph{d-1})}.
\end{align}
On the other hand, 
by $2\pi$-periodicity of $e^{tX}$ and $SO(d)$-invariance of $\sigma_{d-1}$, \[\sup_{t\in\R}\ab{t}^{-s}\norma{(e^{tX}-I)^kf}_{L^2(\sph{d-1})}\simeq \sup_{\ab{t}<1}\ab{t}^{-s}\norma{(e^{tX}-I)^kf}_{L^2(\sph{d-1})}\]
for $k\in\{1,2\}$.
This together with \eqref{eq:ineqSingleDouble} yields $\norma{f}_{\mathcal{H}^s}\leq C\norma{f}_{\mathscr{H}^s}$, for some $C<\infty$.

Secondly, the equivalence between  $\mathcal H^s$ and $\mathscr H^s$ when $s\in(1,2)$ can be obtained via the techniques in \cite[\S 3]{DX11} (especially Theorem 3.6) and \cite[\S 2.3]{DX10}, which rely on the modulus of smoothness and Marchaud-type inequalities. Indeed, the equivalence of the norms $\norma{\cdot}_{\mathcal{W}_p^{r,\alpha}}$ and $\norma{\cdot}_{H_p^{r+\alpha}}$ given by \cite[Theorem 3.6]{DX11} provides the answer after specializing to $(\ell,r\,p,\alpha)=(1,1,2,s-1)$, for in this case
$\norma{f}_{\mathcal{H}^s}\simeq \norma{f}_{\mathcal{W}_p^{r,\alpha}}+\norma{f}_{\mathcal{H}^{s-1}}$, $\norma{f}_{H_p^{r+\alpha}}\simeq \norma{f}_{\mathscr{H}^s}$ and, as already remarked, $\norma{f}_{\mathcal{H}^{s-1}}\simeq \norma{f}_{\mathscr{H}^{s-1}}\lesssim\norma{f}_{\mathscr{H}^{s}}$. The argument is straightforward but lengthy; thus the reader is directed to the aforementioned references.

\subsection{One final remark}\label{sec:finalrmk}
Our proof of Theorem \ref{thm:smoothnessTheorem} does not in general handle the case when $\la=0$ in \eqref{eq:generalEL}.
An exception corresponds to the case when $m=2k$ is an even integer, $\vec k\in\{0,1\}^{m+1}$ satisfies $k_1+\dots+k_{m+1}=k-1$,
and $a>0$ on $\sph{d-1}$ (or, more generally, $a=0$ on a set of $\sigma_{d-1}$-measure zero), 
which corresponds to the Euler--Lagrange equation \eqref{eq:simpleEL} with $\la=0$. In this case, by multiplying both sides of \eqref{eq:generalEL} by $\overline{f}$ and integrating over $\sph{d-1}$, one concludes that $\norma{\widehat{f\sigma}_{d-1}}_{L^{m+2}(\R^d)}=0$, which clearly forces $f=\mathbf{0}$. 
It remains unclear whether one should expect general solutions of \eqref{eq:generalEL} to be smooth when $\la=0$.

\section*{Acknowledgements}
D.O.S.\@ is supported by the EPSRC New Investigator Award ``Sharp Fourier Restriction Theory'', grant no.\@ EP/T001364/1,
and is grateful to
Michael Christ, Giuseppe Negro, and Po-Lam Yung
for valuable discussions during the preparation of this work.
The authors thank the anonymous referees for a careful reading of the manuscript and valuable suggestions.

\end{document}